\documentclass[9pt]{article}
\usepackage{amsmath, amssymb, amsthm, graphicx, cite, comment}
\usepackage[lmargin=0.9in,rmargin=0.9in,tmargin=1in,bmargin=1in]{geometry}
\usepackage[arrow,curve,matrix,arc,2cell]{xy}
\UseAllTwocells
\DeclareFontFamily{U}{rsfs}{} \DeclareFontShape{U}{rsfs}{n}{it}{<->
rsfs10}{} \DeclareSymbolFont{mscr}{U}{rsfs}{n}{it}
\DeclareSymbolFontAlphabet{\scr}{mscr}
\def\mathscr{\scr}

\usepackage{makeidx}
\usepackage[refpage]{nomencl}
\makenomenclature
\makeindex 

\begin{document}
%%%%%%%%%%%%%%%%%%%%%%%%%%%%%%%%%%%%%%%%%%%%%%%%%%%%%%%%%%%%%%%%%%%%%%%%
%%%%%%%%%%%%%%%%%%%%%%%%%%     Macros      %%%%%%%%%%%%%%%%%%%%%%%%%%%%%
%%%%%%%%%%%%%%%%%%%%%%%%%%%%%%%%%%%%%%%%%%%%%%%%%%%%%%%%%%%%%%%%%%%%%%%%

\renewcommand{\nomname}{Glossary of Notation}
\renewcommand{\pagedeclaration}[1]{, #1}
\setlength{\nomitemsep}{0pt}

\def\e#1\e{\begin{equation}#1\end{equation}}
\def\ea#1\ea{\begin{align}#1\end{align}}
\def\eq#1{{\rm(\ref{#1})}}
\theoremstyle{plain}% default
\newtheorem{thm}{Theorem}[section]
\newtheorem{prop}[thm]{Proposition}
\newtheorem{lem}[thm]{Lemma}
\newtheorem{cor}[thm]{Corollary}
\newtheorem{quest}[thm]{Question}
\newtheorem{ass}[thm]{Assumption}
\theoremstyle{definition}
\newtheorem{dfn}[thm]{Definition}
\newtheorem{ex}[thm]{Example}
\newtheorem{rem}[thm]{Remark}

\newtheorem{conj}[thm]{Conjecture}

\numberwithin{equation}{section}
\def\dim{\mathop{\rm dim}\nolimits}
\def\supp{\mathop{\rm supp}\nolimits}
\def\cosupp{\mathop{\rm cosupp}\nolimits}
\def\id{\mathop{\rm id}\nolimits}
\def\Hess{\mathop{\rm Hess}\nolimits}
\def\Crit{\mathop{\rm Crit}}
\def\Sym{\mathop{\rm Sym}\nolimits}
\def\DR{\mathop{\rm DR}\nolimits}
\def\HP{\mathop{\rm HP}\nolimits}
\def\HN{\mathop{\rm HN}\nolimits}
\def\HC{\mathop{\rm HC}\nolimits}
\def\NC{\mathop{\rm NC}\nolimits}
\def\PC{\mathop{\rm PC}\nolimits}
\def\CC{\mathop{\rm CC}\nolimits}
\def\inf{{\rm inf}}
\def\St{\mathop{\bf St}\nolimits}
\def\Art{\mathop{\bf Art}\nolimits}
\def\dSt{\mathop{\bf dSt}\nolimits}
\def\dArt{\mathop{\bf dArt}\nolimits}
\def\dSch{\mathop{\bf dSch}\nolimits}
\def\cl{{\rm cl}}
\def\fib{\mathop{\rm fib}\nolimits}
\def\Ho{\mathop{\rm Ho}\nolimits}
\def\Ker{\mathop{\rm Ker}}
\def\Coker{\mathop{\rm Coker}}
\def\GL{\mathop{\rm GL}}
\def\bSpec{\mathop{\bs{\rm Spec}}\nolimits}
\def\Im{\mathop{\rm Im}}
\def\inc{\mathop{\rm inc}}
\def\Ext{\mathop{\rm Ext}\nolimits}
\def\qcoh{{\rm qcoh}}
\def\id{{\mathop{\rm id}\nolimits}}
\def\cHom{\mathbin{\mathcal{H}om}}

\def\Aut{\mathop{\rm Aut}}
\newcommand{\dd}{{\rm d}_{dR}}
\newcommand{\Pic}{\mathop{\rm Pic}}
\newcommand{\End}{\mathop{\rm End}\nolimits}
\newcommand{\hd}{\mathop{\rm hd}}
\newcommand{\Hilb}{\mathop{\rm Hilb}\nolimits}
\newcommand{\cs}{{\rm cs}}
\newcommand{\Con}{{\sst{\rm Con}}}
\newcommand{\bdim}{{\mathbin{\bf dim}\kern.1em}}
\newcommand{\Stab}{\mathop{\rm Stab}\nolimits}
\newcommand{\stab}{\mathop{\mathfrak{stab}}\nolimits}
\newcommand{\Quot}{\mathop{\rm Quot}}
\newcommand{\CF}{\mathop{\rm CF}\nolimits}
\newcommand{\SU}{\mathop{\rm SU}}
\newcommand{\SL}{\mathop{\rm SL}}
\newcommand{\U}{\mathop{\ts\rm U}}
\newcommand{\Gr}{\mathop{\rm Gr}}
\newcommand{\Mo}{\mathop{\text{M\"o}}}
\newcommand{\Ch}{\mathop{\rm Ch}}
\newcommand{\Tr}{\mathop{\rm Tr}}
\newcommand{\Eu}{\mathop{\rm Eu}\nolimits}
\newcommand{\ch}{\mathop{\rm ch}\nolimits}
\newcommand{\num}{{\rm num}}
\newcommand{\vir}{{\rm vir}}
\newcommand{\stp}{{\rm stp}}
\newcommand{\fr}{{\rm fr}}
\newcommand{\stf}{{\rm stf}}
\newcommand{\rsi}{{\rm si}}
\newcommand{\na}{{\rm na}}
\newcommand{\stk}{{\rm stk}}
\newcommand{\rss}{{\rm ss}}

\newcommand{\st}{{\rm st}}

\newcommand{\all}{{\rm all}}
\newcommand{\rk}{{\rm rk}}
\newcommand{\vi}{{\rm vi}}
\newcommand{\cha}{\mathop{\rm char}}
\newcommand{\ind}{{\rm \kern.05em ind}}
\newcommand{\Exp}{\mathop{\rm Exp}}
\newcommand{\Per}{\mathop{\rm Per}}
\newcommand{\cone}{\mathop{\rm cone}\nolimits}
\newcommand{\At}{\mathop{\rm At}\nolimits}
\newcommand{\VS}{\mathop{\rm Vect}\nolimits}
\newcommand{\Sp}{\mathop{\rm Sp}\nolimits}
\newcommand{\Rep}{\mathop{\rm Rep}\nolimits}
\newcommand{\Iso}{\mathop{\rm Iso}\nolimits}
\newcommand{\LCF}{\mathop{\rm LCF}\nolimits}
\newcommand{\SF}{\mathop{\rm SF}\nolimits}
\newcommand{\SFa}{\mathop{\rm SF}\nolimits_{\rm al}}
\newcommand{\SFai}{\mathop{\rm SF}\nolimits_{\rm al}^{\rm ind}}
\newcommand{\uSF}{\mathop{\smash{\underline{\rm SF\!}\,}}\nolimits}
\newcommand{\uSFi}{\mathop{\smash{\underline{\rm SF\!}\,}}\nolimits^{\rm ind}}
\newcommand{\oSF}{\mathop{\bar{\rm SF}}\nolimits}
\newcommand{\oSFa}{\mathop{\bar{\rm SF}}\nolimits_{\rm al}}
\newcommand{\oSFai}{{\ts\bar{\rm SF}{}_{\rm al}^{\rm ind}}}
\newcommand{\uoSF}{\mathop{\bar{\underline{\rm SF\!}\,}}\nolimits}
\newcommand{\uoSFa}{\mathop{\bar{\underline{\rm SF\!}\,}}\nolimits_{\rm al}}
\newcommand{\uoSFi}{\mathop{\bar{\underline{\rm SF\!}\,}}\nolimits_{\rm al}^{\rm
ind}}
\newcommand{\LSF}{\mathop{\rm LSF}\nolimits}
\newcommand{\LSFa}{\mathop{\rm LSF}\nolimits_{\rm al}}
\newcommand{\uLSF}{\mathop{\smash{\underline{\rm LSF\!}\,}}\nolimits}
\newcommand{\dLSF}{{\dot{\rm LSF}}\kern-.1em\mathop{}\nolimits}
\newcommand{\doLSF}{{\dot{\bar{\rm LSF}}}\kern-.1em\mathop{}\nolimits}
\newcommand{\duoLSF}{{\dot{\bar{\underline{\rm LSF\!}\,}}}\kern-.1em\mathop{}
\nolimits}
\newcommand{\ouLSF}{{\bar{\underline{\rm LSF\!}\,}}\kern-.1em\mathop{}
\nolimits}
\newcommand{\dLSFi}{{\dot{\rm LSF}}\kern-.1em\mathop{}\nolimits^{\rm ind}}
\newcommand{\dLSFa}{{\dot{\rm LSF}}\kern-.1em\mathop{}\nolimits_{\rm al}}
\newcommand{\doLSFa}{{\dot{\bar{\rm LSF}}}\kern-.1em\mathop{}\nolimits_{\rm al}}
\newcommand{\dLSFai}{{\dot{\rm LSF}}\kern-.1em\mathop{}\nolimits^{\rm ind}_{\rm
al}}
\newcommand{\duLSF}{{\dot{\underline{\rm LSF\!}\,}}\kern-.1em\mathop{}\nolimits}
\newcommand{\duLSFi}{{\dot{\underline{\rm LSF\!}\,}}\kern-.1em\mathop{}
\nolimits^{\rm ind}}
\newcommand{\oLSF}{\mathop{\bar{\rm LSF}}\nolimits}
\newcommand{\oLSFa}{\mathop{\bar{\rm LSF}}\nolimits_{\rm al}}
\newcommand{\oLSFai}{\mathop{\bar{\rm LSF}}\nolimits_{\rm al}^{\rm ind}}
\newcommand{\uoLSF}{\mathop{\bar{\underline{\rm LSF\!}\,}}\nolimits}
\newcommand{\uoLSFa}{\mathop{\bar{\underline{\rm LSF\!}\,}}\nolimits_{\rm al}}
\newcommand{\uoLSFi}{\mathop{\bar{\underline{\rm LSF\!}\,}}\nolimits_{\rm
al}^{\rm ind}}
\newcommand{\modCQ}{\text{\rm mod-$\C Q$}}
\newcommand{\modCQI}{\text{\rm mod-$\C Q/I$}}
\newcommand{\modKQ}{\text{\rm mod-$\K Q$}}
\newcommand{\modKQI}{\text{\rm mod-$\K Q/I$}}
\newcommand{\projKQ}{\text{\rm proj-$\K Q$}}
\newcommand{\projKQI}{\text{\rm proj-$\K Q/I$}}
\newcommand{\Obj}{\mathop{\rm Obj\kern .1em}\nolimits}
\newcommand{\fObj}{\mathop{\mathfrak{Obj}\kern .05em}\nolimits}
\newcommand{\sA}{{\mathbin{\mathscr A}}}
\newcommand{\sAs}{{\mathscr A}\kern-1.5pt{}_{\rm si}}
\newcommand{\sB}{{\mathbin{\mathscr B}}}
\newcommand{\sE}{{\mathbin{\mathscr E}}}
\newcommand{\sF}{{\mathbin{\mathscr F}}}
\newcommand{\sK}{{\mathbin{\mathscr K}}}
\newcommand{\sY}{{\mathbin{\mathscr Y}}}
\newcommand{\sX}{{\mathbin{\mathscr X}}}
\newcommand{\sG}{{\mathbin{\mathscr G}}}
\newcommand{\sV}{{\mathbin{\mathscr V}}}
\newcommand{\sW}{{\mathbin{\mathscr W}}}
\newcommand{\sS}{{\mathbin{\mathscr S}\kern-3pt{}}}
\newcommand{\bG}{{\mathbin{\mathbb G}}}
\newcommand{\bH}{{\mathbin{\mathbb H}}}
\newcommand{\Bi}{{\mathbin{\mathbb B}}}
\newcommand{\bCF}{{\mathbin{\mathbb{CF}}}}
\newcommand{\Biv}{{\mathbin{\mathcal B}}}
\def\Perf{\mathop{\rm Perf}\nolimits}
\def\Perfis{\mathop{\text{\rm Perf-is}}\nolimits}
\def\Lbis{\mathop{\text{\rm Lb-is}}\nolimits}
\def\coh{\mathop{\rm coh}\nolimits}
\def\Hom{\mathop{\rm Hom}\nolimits}
\def\Perv{\mathop{\rm Perv}\nolimits}
\def\Sch{\mathop{\rm Sch}\nolimits}
\def\Var{\mathop{\rm Var}\nolimits}
\def\red{{\rm red}}
\def\DM{\mathop{\rm DM}\nolimits}
\def\SO{\mathop{\rm SO}\nolimits}
\def\Spec{\mathop{\rm Spec}\nolimits}
\def\rank{\mathop{\rm rank}\nolimits}
\def\bs{\boldsymbol}
\def\alg{{\rm alg}}
\def\ge{\geqslant}
\def\le{\leqslant\nobreak}
\def\bA{{\mathbin{\mathbb A}}}
\def\bL{{\mathbin{\mathbb L}}}
\def\bT{{\mathbin{\mathbb T}}}
\def\bR{{\bs R}}
\def\bS{{\bs S}}
\def\bU{{\bs U}}
\def\bV{{\bs V}}
\def\bW{{\bs W}}
\def\bY{{\bs Y}}
\def\bZ{{\bs Z}}
\def\bX{{\bs X}}
\def\CP{{\mathbin{\mathbb{CP}}}}
\newcommand{\RP}{{\mathbin{\mathbb{RP}}}}
\def\cF{{\mathbin{\cal F}}}
\def\cG{{\mathbin{\cal G}}}
\def\cK{{\mathbin{\cal K}}}
\def\cS{{\mathbin{\cal S}\kern -0.1em}}
\def\cSz{{\mathbin{\cal S}\kern -0.1em}^{\kern .1em 0}}
\def\cE{{\mathbin{\cal E}}}
\def\cL{{\mathbin{\cal L}}}
\def\cH{{\mathbin{\cal H}}}
\def\cO{{\mathbin{\cal O}}}
\def\cP{{\mathbin{\cal P}}}
\def\cS{{\mathbin{\cal S}}}
\def\PV{{\mathbin{\cal{PV}}}}
\def\TS{{\mathbin{\cal{TS}}}}
\def\cW{{\mathbin{\cal W}}}
\def\cM{{\mathbin{\cal M}}}
\def\oM{{\mathbin{\smash{\,\,\overline{\!\!\mathcal M\!}\,}}}}
\def\O{{\mathbin{\cal O}}}
\def\cA{{\mathbin{\cal A}}}
\def\cB{{\mathbin{\cal B}}}
\def\cC{{\mathbin{\cal C}}}

\def\cD{{\mathbin{\scr D}}}

\def\PV{{\mathbin{\cal{PV}}}}
\def\bG{{\mathbin{\mathbb G}}}
\def\bD{{\mathbin{\mathbb D}}}
\def\bH{{\mathbin{\mathbb H}}}
\def\bP{{\mathbin{\mathbb P}}}
\def\A{{\mathbin{\mathcal A}}}

\def\PP{{\mathbin{\mathcal P}}}

\def\fE{{\mathbin{\mathfrak E}}}
\newcommand{\fF}{{\mathbin{\mathfrak F}}}
\newcommand{\fH}{{\mathbin{\mathfrak H}}}
\newcommand{\cQ}{{\mathbin{\mathcal Q}}}
\newcommand{\fN}{{\mathbin{\mathfrak N}}}
\newcommand{\fU}{{\mathbin{\mathfrak{U}\kern .05em}}}
\newcommand{\fV}{{\mathbin{\mathfrak V}}}
\newcommand{\fX}{{\mathbin{\mathfrak X}}}
\newcommand{\M}{{\mathbin{\mathcal M}}}
\newcommand{\fM}{{\mathbin{\mathfrak M}}}
\newcommand{\h}{{\mathbin{\mathfrak h}}}
\newcommand{\m}{{\mathbin{\mathfrak m}}}
\newcommand{\fR}{{\mathbin{\mathfrak R}}}
\newcommand{\fS}{{\mathbin{\mathfrak S}}}
\newcommand{\fExact}{\mathop{\mathfrak{Exact}\kern .05em}\nolimits}
\newcommand{\Vect}{{\mathbin{\mathcal{V}ect}}}
\newcommand{\fVect}{{\mathbin{\mathfrak{Vect}}}}
\newcommand{\Hol}{\mathbin{\mathcal{H}ol}}
\newcommand{\fHol}{{\mathbin{\mathfrak{Hol}}}}
\newcommand{\fW}{{\mathbin{\mathfrak W}}}
\newcommand{\SHom}{\mathbin{\mathcal{H}om}}
\newcommand{\SExt}{\mathcal{E}xt}
\newcommand{\dtensor}{\stackrel{L}{\otimes}}
\newcommand{\tr}{\mathop{\rm tr}\nolimits}
\newcommand{\Exal}{\mathop{\rm Exal}\nolimits}
\def\fG{{\mathbin{\mathfrak G}}}
\def\fU{\mathbin{\mathfrak U}}
\def\C{{\mathbin{\mathbb C}}}
\def\K{{\mathbin{\mathbb K}}}
\def\N{{\mathbin{\mathbb N}}}
\def\Q{{\mathbin{\mathbb Q}}}
\def\Z{{\mathbin{\mathbb Z}}}
\def\K{{\mathbin{\mathbb K}}}
\def\R{{\mathbin{\mathbb R}}}
\newcommand{\Aff}{{\mathbin{\mathbb A}}}
\newcommand{\bde}{\bar\delta}
\newcommand{\bdss}{\bar\delta_{\rm ss}}
\newcommand{\bep}{\bar\epsilon}
\def\al{\alpha}
\def\be{\beta}
\def\ga{\gamma}
\def\de{\delta}
\def\io{\iota}
\def\ep{\epsilon}
\def\la{\lambda}
\def\ka{\kappa}
\def\th{\theta}
\def\ze{\zeta}
\def\up{\upsilon}
\def\vp{\varphi}
\def\si{\sigma}
\def\om{\omega}
\def\De{\Delta}
\def\La{\Lambda}
\def\Th{\Theta}
\def\Tau{{\rm T}}
\def\Om{\Omega}
\def\Ga{\Gamma}
\def\Si{\Sigma}
\def\Up{\Upsilon}
\def\pd{\partial}
\def\db{{\bar\partial}}
\def\ts{\textstyle}

\def\sst{\scriptscriptstyle}
\def\sm{\setminus}
\def\w{\wedge}
\def\lt{\ltimes}
\def\sh{\sharp}
\def\bu{\bullet}
\def\op{\oplus}
\def\ot{\otimes}
\def\od{\odot}
\def\bigot{\bigotimes}
\def\boxt{\boxtimes}
\def\ov{\overline}
\def\ul{\underline}
\def\bigop{\bigoplus}
\def\iy{\infty}
\def\es{\emptyset}
\def\ra{\rightarrow}
\def\ab{\allowbreak}
\def\longra{\longrightarrow}
\def\hookra{\hookrightarrow}
\def\dashra{\dashrightarrow}
\def\rra{\rightrightarrows}
\def\Ra{\Rightarrow}
\def\Longra{\Longrightarrow}
\newcommand{\el}{{\mathbin{\ell\kern .08em}}}
\def\ha{{\ts\frac{1}{2}}}
\def\ban#1{\bigl\langle #1 \bigr\rangle}

\def\an{{\rm an}}

\newcommand{\gr}{\grave}
\def\t{\times}
\def\ci{\circ}
\def\ti{\tilde}
\def\d{{\rm d}}
\newcommand{\rd}{{\rm d}}
\def\od{\odot}
\def\bd{\boxdot}
\def\md#1{\vert #1 \vert}
\def\ms#1{\vert #1 \vert^2}
\def\nm#1{\Vert #1 \Vert}
\def\bnm#1{\big\Vert #1 \big\Vert}
\def\bmd#1{\big\vert #1 \big\vert}
\def\HM{\mathop{\rm HM}\nolimits}
\def\MHM{\mathop{\rm MHM}\nolimits}
\def\rat{\mathop{\bf rat}\nolimits}
\def\Rat{\mathop{\bf Rat}\nolimits}
\def\HV{{\mathbin{\cal{HV}}}}
\def\cS{{\mathbin{\cal S}}}
\def\cSz{{\mathbin{\cal S}\kern -0.1em}^{\kern .1em 0}}
\def\otL{{\kern .1em\mathop{\otimes}\limits^{\sst L}\kern .1em}}
\def\boxtL{{\kern .2em\mathop{\boxtimes}\limits^{\sst L}\kern .2em}}
\def\boxtT{{\kern .1em\mathop{\boxtimes}\limits^{\sst T}\kern .1em}}

\def\Sh{{\mathop{\rm Sh}\nolimits}}
\def\Kalg{{\text{\rm $\K$-alg}}}
\def\Kvect{{\text{\rm $\K$-vect}}}
\def\fIso{{\mathop{\mathfrak{Iso}}\nolimits}}
\def\bcM{{\mathbin{\bs{\cal M}}}}
\def\nai{{\text{\rm na\"\i}}}
\def\Lm{{\mathop{\text{\rm Lis-me}}\nolimits}}
\def\Le{{\mathop{\text{\rm Lis-\'et}}\nolimits}}
\def\Lian{{\mathop{\text{\rm Lis-an}}\nolimits}}
\def\otL{{\kern .1em\mathop{\otimes}\limits^{\sst L}\kern .1em}}
\def\boxtL{{\kern .2em\mathop{\boxtimes}\limits^{\sst L}\kern .2em}}
\def\Lmm{{\mathop{\text{\rm LM-\'em}}\nolimits}}
\def\cR{{\mathbin{\cal R}}}
\def\stm{{{\rm st},\hat\mu}}

%%%%%%%%%%%%%%%%%%%%%%%%%%%%%%%%%%%%%%%%%%%%%%%%%%%%%%%%%%%%%%%%%%%%%%%%
%%%%%%%%%%%%%%%%%%%%%% Text of paper %%%%%%%%%%%%%%%%%%%%%%%%%%%%%%%%%%%
%%%%%%%%%%%%%%%%%%%%%%%%%%%%%%%%%%%%%%%%%%%%%%%%%%%%%%%%%%%%%%%%%%%%%%%%
\title{{\bf{Generalized Donaldson--Thomas theory \\
over fields $\K \neq \C$}}}
\author{\smallskip\\  
Vittoria Bussi \\
\small{The Mathematical Institute,}\\
\small{Andrew Wiles Building
Radcliffe Observatory Quarter,}\\
\small{Woodstock Road, Oxford, OX1 3LB, U.K.} \\
\small{E-mail: \tt bussi@maths.ox.ac.uk}\\
\smallskip}
\date{\small{}}
\maketitle

\begin{abstract}

\bigskip

{\it Generalized Donaldson--Thomas invariants\/} 
$\bar{DT}{}^\al(\tau)$ defined by Joyce and Song \cite{JoSo} are rational numbers which `count' both
$\tau$-stable and $\tau$-semistable coherent sheaves with Chern
character $\al$ on a Calabi--Yau 3-fold $X$, where $\tau$ denotes Gieseker
stability for some ample line bundle on $X$. 
The $\bar{DT}{}^\al(\tau)$ are defined for all classes $\al$, and are
equal to the classical $DT^\al(\tau)$ defined by Thomas \cite{Thom} when it is defined. 
They are unchanged under
deformations of $X$, and transform by a wall-crossing formula under
change of stability condition~$\tau$.
Joyce and Song use gauge theory and transcendental complex analytic
methods, so that their theory of generalized Donaldson--Thomas invariants is valid only in the complex case.
This also forces them to put constraints on the Calabi--Yau $3$-fold they can define 
generalized Donaldson--Thomas invariants for.

\smallskip

This paper will propose a new algebraic method 
extending the theory to algebraically closed fields $\K$ of characteristic zero,
and partly to triangulated categories and for non necessarily compact Calabi--Yau 
3-folds under some hypothesis. 

\smallskip

It will describe the local structure of the moduli stack
$\fM$ of (complexes of) coherent sheaves on $X,$ showing that an atlas for $\fM$ 
carries the structure of a $\GL(n,\K)$-invariant d-critical locus in the sense of \cite{
Joyc2}
and thus it 
may be written locally as the zero locus of a regular function 
defined on an \'etale neighborhood in the tangent space of $\fM$
and use this to deduce identities on the Behrend
function $\nu_\fM$. 

\smallskip

Moreover, when $\K = \C,$ \cite[Thm. 4.9]{JoSo} uses the integral Hodge conjecture result by Voisin for Calabi--Yau $3$-folds over $\C$ to show that the numerical Grothendieck group $K^{\rm num}(\coh(X))$ is unchanged under deformations of $X.$ This is important for the results that $\bar{DT}{}^\al(\tau)$ for $\al \in K^{\rm num}(\coh(X))$ are invariant under deformations of $X$, even to make sense. We will provide an algebraic proof of that result,  characterizing the numerical Grothendieck group of a Calabi--Yau $3$-fold in terms of a globally constant lattice described using the Picard scheme. 

\end{abstract}

\tableofcontents

\section*{Introduction} 
\markboth{Introduction}{Introduction}
\addcontentsline{toc}{section}{Introduction}
\label{dt1}

In the following we will summarize some background material on Donaldson--Thomas theory which permits to allocate our problem and state the main result. After that, we outline the contents of the sections. Expert readers can skip the first introductory part.

\bigskip

\begin{center} \textbf{Notations and conventions} \end{center}  

Let $\K$ be an algebraically closed field of
characteristic zero. A {\it Calabi--Yau\/
$3$-fold\/} is a smooth projective 3-fold $X$ over $\C$ or $\K$, with
trivial canonical bundle\index{canonical bundle} $K_X$ and $H^1(\cO_X)=0$. Fix a very
ample line bundle $\cO_X(1)$ on $X$, and let $\tau$ be
Gieseker stability on the abelian category of coherent sheaves $\coh(X)$ on $X$ with respect to  $\cO_X(1).$ 
If $E$ is a coherent sheaf on X then the class $[E] \in K^{\num}(\coh(X))$ 
is in effect the Chern character $\ch(E)$ of $E$ in the Chow ring $A^*(X)_{\Q}$ as in \cite{Fult}.
For a class $\al$ in the numerical Grothendieck group $K^\num(\coh(X))$, write $\M_\rss^\al(\tau),\M_\st^\al(\tau)$ for the coarse moduli
schemes of
$\tau$-(semi)stable sheaves $E$ with class $[E]=\al$. Then
$\M_\rss^\al(\tau)$ is a projective $\C$ or $\K$-scheme whose points correspond to S-equivalence classes of $\tau$-semistable sheaves, and $\M_\st^\al(\tau)$ is an open
subscheme of $\M_\rss^\al(\tau)$ whose points correspond to
isomorphism classes of $\tau$-stable sheaves.
Write $\fM$ for the moduli stack of coherent sheaves $E$ on $X$. It
is an Artin $\C$ or $\K$-stack, locally of finite
type and has affine
geometric stabilizers. For $\al\in
K^{\num}(\coh(X))$, write $\fM^\al$ for the open and closed substack of $E$
with $[E]=\al$ in $K^{\num}(\coh(X))$. Write $\fM_\rss^\al(\tau),
\fM_\st^\al(\tau)$ for the substacks of $\tau$-(semi)stable sheaves
$E$ in class $[E]=\al$, which are finite type open substacks of
$\fM^\al$. 

\index{Artin stack!affine geometric
stabilizers} 
\index{Artin stack!locally of finite type} 
\nomenclature[M al st]{$\M_\st^\al(\tau)$}{coarse moduli scheme of $\tau$-stable sheaves $E$ with class $[E]=\al$}
\nomenclature[M al rss]{$\M_\rss^\al(\tau)$}{coarse moduli scheme of $\tau$-semistable sheaves $E$ with class $[E]=\al$}
\index{Gieseker stability}
\nomenclature[Knum]{$K^\num(\coh(X))$}{numerical Grothendieck group of the abelian category $\coh(X)$} \index{Grothendieck group!numerical}

\bigskip

\begin{center} \textbf{Historical overview} \end{center}  

\index{Donaldson--Thomas 
invariants!original $DT^\al(\tau)$}\nomenclature[DTa]{$DT^\al(\tau)$}{original
Donaldson--Thomas invariants defined in \cite{Thom}} 
In 1998, Thomas \cite{Thom}, following his proposal with Donaldson \cite{DoTh}, 
motivates a {\it holomorphic Casson invariant} and defines the 
{\it Donaldson--Thomas invariants\/} $DT^\al(\tau)$
which are integers `counting' $\tau$-stable coherent sheaves with Chern character
$\al$ on a Calabi--Yau 3-fold $X$ over $\K,$ where $\tau$ 
 denotes Gieseker stability for some ample line bundle on $X$. 
Mathematically, and in `modern' terms, he found that $\M_\st^\al(\tau)$ is endowed with a symmetric obstruction theory
and defined
\begin{equation*}
DT^\al(\tau)\quad =\ts\displaystyle \int\limits_{\small{[\M_\st^\al(\tau)]^\vir}}\!\!\! 1
\end{equation*}
which is mathematical reflection of the heuristic that views $\M_\st^\al(\tau)$ as the critical locus of
the {\it holomorphic Chern-Simons functional} and the shadow of a more deeper `derived' geometry.
A crucial result is that the invariants are unchanged
under deformations of the underlying geometry of $X$. Finally we remark that 
the conventional definition of Thomas \cite{Thom} works only
for classes $\al$ containing no strictly $\tau$-semistable sheaves and this permits to work just with schemes rather than stacks as 
the stable moduli scheme itself already encodes all the information about the $\Ext$ groups, and thus about the tangent-obstruction complex of the moduli functor. 

\smallskip

In 2005, Behrend \cite{Behr} proved a {\it virtual Gauss--Bonnet theorem} 
which in particular yields that Donaldson--Thomas type invariants can be written as a weighted 
Euler characteristic $$DT^\al(\tau)=\chi\bigl(\M_\st^\al(\tau),
\nu_{\M_\st^\al(\tau)}\bigr)$$ of the stable moduli scheme
$\M_\st^\al(\tau)$ \nomenclature[M al tau]{$\M_\st^\al(\tau)$}{coarse moduli scheme of $\tau$-stable objects in class $\al$} by a constructible function $\nu_{\M_\st^\al(\tau)},$ as a consequence known in literature as the {\it Behrend function}. 
It depends only on the scheme structure of $\M_\st^\al(\tau),$ and it is convenient to think about it as a
multiplicity function. 
An important moral is that it is better to `count' points in a moduli
scheme by the weighted Euler characteristic rather than the unweighted one as
it often gives answers unchanged under
deformations of the underlying geometry.
It is worth to point out that this equation is local, and `motivic', and makes sense even for non-proper 
finite type $\K$-schemes. 
Anyway, using this formula to generalize the classical picture 
by defining the Donaldson--Thomas invariants as $\chi\bigl(\M_\rss^\al(\tau),
\nu_{\M_\rss^\al(\tau)}\bigr)$ when $\M_\rss^\al(\tau)\ne\M_\st^\al(\tau)$ is not a good idea as in the case there are  
strictly $\tau$-semistable sheaves, the moduli scheme $\M_\rss^\al(\tau)$ is no more a good model and suggest 
that schemes are no more `enough' to extend the theory. 

\smallskip

The crucial work by Behrend \cite{Behr} suggests that Donaldson--Thomas invariants 
can be written as motivic invariants, like those studied by Joyce in \cite{Joyc.1,Joyc.2,Joyc.3,Joyc.4,Joyc.5,Joyc.6}, and so it raises the possibility that one can extend the results of \cite{Joyc.1,Joyc.2,Joyc.3,Joyc.4,Joyc.5,Joyc.6} to Donaldson--Thomas invariants by including Behrend functions as weights.

\smallskip

\index{Donaldson--Thomas invariants!generalized
$\bar{DT}{}^\al(\tau)$} \nomenclature[DTb]{$\bar{DT}{}^\al(\tau)$}{generalized Donaldson--Thomas invariants defined in \cite{JoSo}}
Thus, in 2005, Joyce and Song \cite{JoSo}  proposed a theory of {\it generalized Donaldson--Thomas invariants\/}
$\bar{DT}{}^\al(\tau)$. They are rational numbers which `count' both
$\tau$-stable and $\tau$-semistable coherent sheaves with Chern
character $\al$ on a compact Calabi--Yau 3-fold $X$ over $\C$; strictly $\tau$-semistable sheaves must be
counted with complicated rational weights. The
$\bar{DT}{}^\al(\tau)$ are defined for all classes $\al$, and are
equal to $DT^\al(\tau)$ when it is defined. They are unchanged under
deformations of $X$, and transform by a wall-crossing formula under
change of stability condition~$\tau$. The theory is valid also for compactly supported coherent sheaves on {\it compactly embeddable} noncompact Calabi--Yau 3-folds in the complex analytic topology. 

\smallskip

To prove all this they study the local structure of the moduli stack \index{moduli stack!local structure}
$\fM$ of coherent sheaves on $X.$ They first show that $\fM$ is Zariski locally isomorphic to the moduli
stack $\fVect$ of algebraic vector bundles on $X$. Then they use {\it gauge
theory} on complex vector bundles and transcendental complex analytic
methods to show that an atlas for $\fM$ may be written locally in
the complex analytic topology as $\Crit(f)$ for $f:U\ra\C$ a \index{analytic topology}
holomorphic function on a complex manifold $U$. 
They use this to deduce identities on the Behrend function $\nu_\fM$ through the Milnor
fibre description of Behrend functions. These identities 
\bigskip
\begin{equation*}
\nu_{\fM}(E_1\op E_2)=(-1)^{\bar\chi([E_1],[E_2])}
\nu_{\fM}(E_1)\nu_{\fM}(E_2), 
\end{equation*}
\smallskip
\begin{equation*}
\displaystyle \int\limits_{\small{\begin{subarray}{l}  [\la]\in\mathbb{P}(\Ext^1(E_2,E_1)):\\
 \la\; \Leftrightarrow\; 0\ra E_1\ra F\ra E_2\ra
0\end{subarray}}}\!\!\!\!\! \!\!\!\! \!\!\!\! \!\!\!\! \nu_{\fM}(F)\,\rd\chi \quad - \!\!\!\! \!\!\!\! \!\!\!\! 
\displaystyle \int\limits_{\small{\begin{subarray}{l}[\mu]\in\mathbb{P}(\Ext^1(E_1,E_2)):\\
\mu\; \Leftrightarrow\; 0\ra E_2\ra D\ra E_1\ra
0\end{subarray}}}\!\!\!\!\! \!\!\!\! \!\!\!\! \!\!\!\! \nu_{\fM}(D)\,\rd\chi \;\;
= \;\; (e_{21}-e_{12})\;\;
\nu_{\fM}(E_1\op E_2),
\end{equation*}
\bigskip

where $e_{21}=\dim\Ext^1(E_2,E_1)$ and $e_{12}=\dim\Ext^1(E_1,E_2)$ for $E_1,E_2\in\coh(X),$ are crucial for the whole program of Joyce and Song, which is based on the idea that 
Behrend's approach should be integrated with Joyce's theory \cite{Joyc.1,Joyc.2,Joyc.3,Joyc.4,Joyc.5,Joyc.6}.
As the proof uses gauge theory and transcendental methods, it works only over~$\C$ and forces 
them to put constraints on the Calabi--Yau 3-fold they can define 
generalized Donaldson--Thomas invariants for.
Finally, in \cite[\S 4.5]{JoSo}, when $\K=\C,$ the Chern
character embeds $K^\num(\coh(X))$ in $H^{\rm
even}(X;\Q)$, and the Voisin Hodge conjecture result \cite{Vois} for Calabi--Yau over $\C$ 
completely characterize its image. 
They use this to show $K^\num(\coh(X))$ is
unchanged under deformations of $X$. This is important for the $\bar{DT}{}^\al(\tau)$ with $\al\in K^\num(\coh(X))$ to be invariant
under deformations of $X$ even to make sense. \index{deformation-invariance}\index{Chern character} 

\smallskip

In 2008 and 2010, with two subsequent papers \cite{KoSo1,KoSo}, Kontsevich and Soibelman also studied generalizations of Donaldson--Thomas invariants, both in the direction of motivic and categorified Donaldson--Thomas invariants.

\smallskip

In \cite{KoSo1}, they proposed a very general version of the theory, which, very roughly speaking, can be outlined saying that, 
supposing for the sake of simplicity that $\M_\st^\al(\tau)=\M_\rss^\al(\tau)$, their oversimplified idea is to define {\it motivic Donaldson--Thomas invariants} $$DT^\al_{\textrm{mot}}=\Upsilon(\M_\st^\al(\tau),\nu_{\textrm{mot}}),$$
where $\nu_{\textrm{mot}}$ is a complicated constructible function which we can refer to as the {\it motivic Behrend function} for a general motivic invariant $\Upsilon.$ 
Their construction is closely related to Joyce and Song's construction, even if they work in a more general context: they consider derived categories of coherent sheaves, Bridgeland stability conditions, and general motivic invariants, whereas Joyce and Song work with abelian categories of coherent sheaves, Gieseker stability, and the Euler characteristic. 
However, the price to work in a more general context is that most results depend on conjectures (motivic Behrend function identities, existence of orientation data, absence of poles).
In particular, Kontsevich and Soibelman's parallel passages of Joyce and Song's proof of the Behrend function identities 
\cite[\S 4.4 \& \S 6.3]{KoSo1} work over
a field $\K$ of characteristic zero, and say that the formal
completion $\hat\fM_{[E]}$ of $\fM$ at $[E]$ can be written in
terms of $\Crit(f)$ for $f$ a formal power series on
$\smash{\Ext^1(E,E)}$, with no convergence criteria. Their
analogue \cite[Conj.\,4]{KoSo1},
concerns the {\it motivic Milnor fibre} of the formal power series
$f$.
So the Behrend function identities are related to a conjecture of
Kontsevich and Soibelman \cite[Conj.\,4]{KoSo1} and its application
in \cite[\S 6.3]{KoSo1}, and could probably be deduced from it. 
Anyway, Joyce and Song's approach \cite{JoSo} is not wholly
algebro-geometric -- it uses gauge theory, and
transcendental complex analytic geometry methods. Therefore this
method will not suffice to prove the parallel conjectures in
Kontsevich and Soibelman \cite[Conj.\,4]{KoSo1}, which are supposed
to hold for general fields $\K$ as well as $\C$, and for general
motivic invariants of algebraic $\K$-schemes as well as for the \index{motivic invariant}
topological Euler characteristic. Recently, in 2012, Le Quy Thuong \cite{Thuong} provided a proof for 
this conjecture using some deep high technology results from motivic integration. 

\smallskip

In \cite{KoSo}, Kontsevich and Soibelman exposed the categorified version of Donaldson--Thomas theory. 
To fix ideas, suppose again $\M_\st^\al(\tau)=\M_\rss^\al(\tau).$ 
Following Thomas' argument \cite{Thom}, one can, heuristically, think of $\nu_{\M_\st^\al(\tau)}$ as the 
Euler characteristic of the {\it perverse sheaf of vanishing cycles} $\cal P$ of the holomorphic Chern-Simons functional.
Following this philosophy in which perverse sheaves are categorification of constructible functions, 
the hypercohomology 
$${\mathbb H}^*\bigl(\M_\st^\al(\tau);{\cal
P}\vert_{\M_\st^\al(\tau)}\bigr)$$ 
would be a natural cohomology 
group of $\M_\st^\al(\tau)$ whose Euler characteristic is the
Donaldson--Thomas invariant. Thus, the very basic idea in Kontsevich and Soibelman's paper 
is to define some kind of `generalized 
cohomology' for the moduli stack $\fM$ as a kind of Ringel--Hall algebra. 

\smallskip

In 2013, a sequence of five papers \cite{Joyc2,BBJ,BBDJS,BJM,BBBJ} developed a theory of d-critical loci, a new class of geometric objects by
Joyce, and uses this theory to apply powerful results of derived algebraic geometry as in \cite{Toen,Toen2,Toen3,Toen4,ToVe1,ToVe2,PTVV}
to Donaldson--Thomas theory. 
It is showed that the moduli stack of (complexes of) coherent sheaves on a Calabi--Yau 3-fold
carries the structure of an algebraic d-critical stack and it is given locally in the Zariski topology as the critical locus of a regular function. Moreover, using the notion of {\it orientation data}, they construct a natural perverse sheaf and a natural motive on the moduli stacks, thus answering a long-standing question in the problem of categorification. See \S\ref{ourpapers} for a detailed discussion.

\bigskip

\begin{center} \textbf{The main result and its implications} \end{center}  

Following Joyce and Song's proposal, the aim of this paper is to provide an 
extension of the theory of generalized Donaldson--Thomas invariants in \cite{JoSo}
to algebraically closed fields $\K$ of characteristic zero.
Our argument provides the algebraic analogue of \cite[Thm 5.5]{JoSo}, \cite[Thm 5.11]{JoSo} and \cite[Cor. 5.28]{JoSo} 
which are enough to extend \cite{JoSo} at least for compact Calabi--Yau 3-folds. Unfortunately, 
to extend the whole project to complexes of sheaves and to compactly 
supported sheaves on a noncompact quasi-projective Calabi--Yau 3-fold, we would need other results also from derived algebraic geometry which we do not have at the present. We hope to come back on this point in a future work. 

\smallskip

We will show that an atlas for $\fM$ near $[E]\in\fM(\K)$ may be written locally in
the \'etale topology as the zero locus $\d f^{-1}(0)$ for a $G$-invariant regular function $f$ 
defined on a \'etale neighborhood of $0\in U(\K)$ 
in the affine $\K$-space $\Ext^1(E,E),$  \index{\'etale topology}
where $G$ is a maximal torus of $\Aut(E).$ \index{reductive group!maximal}

\smallskip

Based on this picture, we give an algebraic proof of the Behrend function identities. 
We point out that our approach is actually valid much more generally for any stack which is locally a global quotient, and we actually do not use any particular properties of coherent sheaves on Calabi--Yau 3-folds.
In the past, the author tried
a picture in which the moduli stack of coherent sheaves was locally described as a zero locus of an algebraic almost closed 1-form in the sense of \cite{Behr}, 
which turned out later to be a wrong direction to follow.

\smallskip

Finally, we will study the deformation invariance properties of $\bar{DT}{}^\al(\tau)$ under changes of the underlying geometry of $X,$ 
characterizing a globally constant lattice containing the image through the Chern character of $K^{\num}(\coh(X))$ and
in which classes $\al$ vary. 

\smallskip

The implications are quite exciting and far-reaching.
Our algebraic method could lead to the extension of generalized Donaldson--Thomas theory to the derived categorical context. 
The plan to extend from abelian to derived categories the theory of Joyce and Song \cite{JoSo}
starts by reinterpreting the series of papers by Joyce
\cite{Joyc.1,Joyc.2,Joyc.3,Joyc.4,Joyc.5,Joyc.6,Joyc.7,Joyc.8} in this new general setup. In particular:

\smallskip

\begin{itemize}

\item[$(a)$]Defining configurations in triangulated categories $\mathcal{T}$ requires to replace the exact sequences by distinguished triangles.

\smallskip

\item[$(b)$]Constructing moduli stacks of objects and configurations in $\mathcal{T}$. Again, the theory of derived algebraic geometry \cite{Toen,Toen2,Toen3,Toen4,ToVe1,ToVe2,PTVV} can give us a satisfactory answer.

\smallskip

\item[$(c)$]Defining stability conditions on triangulated categories can be approached using  Bridgeland's results, and its extension by Gorodentscev et al., which combines Bridgeland's idea with Rudakov's definition for abelian categories \cite{Rud}. Since Joyce's stability conditions \cite{Joyc.3} are based on Rudakov, the modifications should be straightforward.

\end{itemize}

\begin{itemize}
\item[$(d)$]The `nonfunctoriality of the cone' in triangulated categories causes that the triangulated category versions of some operations on configurations are defined up to isomorphism, but not canonically, which yields that corresponding diagrams may be commutative, but not Cartesian as in the abelian case. In particular, one loses the associativity of the Ringel-Hall algebra of stack functions, which is a crucial object in Joyce and Song framework. 
We expect that derived Hall algebra approach of To\"en \cite{Toen3} resolve this issue. See also \cite{PL}.
\end{itemize}

\smallskip

We expect that a well-behaved theory of invariants counting $\tau$-semistable objects in triangulated categories in the style of Joyce's theory exists, and we hope to come back on it in a future work.

\bigskip

\begin{center} \textbf{Outstanding problems and recent research} \end{center}  

Donaldson--Thomas theory depicted in this picture is promising and
the literature based on the sketched milestones \cite{Thom,Behr,JoSo,KoSo1,KoSo} is vast. 
Although several interesting developments have been achieved, there are many outstanding problems and 
a whole final picture overcoming these problems and related conjectures is far reaching. 

\smallskip

In 2003, Maulik, Nekrasov, Okounkov and Pandharipande \cite{MNOP1,MNOP2} stated the celebrated {\it MNOP conjecture} in which Donaldson--Thomas invariants for sheaves of rank one have been conjectured to have deep connections with Gromov--Witten theory of Calabi--Yau 3-folds, but also with Gopakumar--Vafa invariants and Pandharipande--Thomas invariants \cite{PaTh}.
Even if there are some results on this conjectural equivalence of theories of curve counting invariants
(Bridgeland \cite{Brid2,Brid3}, Stoppa and Thomas \cite {StTh}, Toda \cite{Toda}), the MNOP conjecture is still unproved.
Moreover, very little is known about the `meaning'
of higher rank Donaldson--Thomas invariants.
In the same work, \cite[Conj.1]{MNOP1}, they formulated a conjecture on values of the virtual count of $\Hilb^dX$ 
(Donaldson--Thomas counting of dimension zero sheaves), that has now been established 
and different proofs are given
by Behrend and Fantechi \cite{BeFa2}, Levine and Pandharipande~\cite{LePa} and Li \cite{Li}.

\smallskip

In \cite[Questions 4.18,5.7,5.10,5.12,6.29]{JoSo} Joyce and Song pointed out some outstanding problems of their theory
and suggest new methods to deal with them. Some of those questions have been answered with new methods as in \cite{Joyc2,BBJ,BBDJS,BJM,BBBJ}.
However, the main limitation of Joyce and Song's approach is due to the fact that they work using gauge theory and transcendental complex analytic methods, which causes the theory is valid only over the complex numbers and puts restrictions on the Calabi--Yau which they can define the theory for, and they deal with abelian rather than triangulated categories. This limits the usefulness of their theory as, for many applications, especially to physics, one needs triangulated categories. 
 Moreover, in \cite[\S 6]{JoSo}, Joyce and Song, following Kontsevich and Soibelman \cite[\S 2.5
\& \S 7.1]{KoSo1}, and from ideas similar to Aspinwall--Morrison computation for a Calabi--Yau 3-fold,
defined the {\it BPS invariants} $\hat{DT}{}^\al(\tau),$ also 
generalizations of Donaldson--Thomas invariants, and 
conjectured to be integers for certain $\tau.$ There are some evidences on this fact \cite[\S 6]{JoSo}, but the problem is still open.
Finally, in \cite[\S 7]{JoSo}, they extended their generalized Donaldson--Thomas theory to abelian categories of representations of a quiver $Q$ with relations coming from a superpotential on Q, and connected their ideas with 
the already existing literature on noncommutative Donaldson--Thomas invariants and on
invariants counting quiver representations (just to cite some names: Bryan, Ginzburg, Hanany, Nagao, Nakajima, Reineke, Szendr\H oi, and Young). This is an active area of research in representation theory. 

\smallskip

There is a seething big area of research which aims to extend Donaldson--Thomas theory in the derived categorical 
framework. 
For a long time there was the problem to prove that the moduli space of complexes of sheaves 
can be given as a critical locus, similarly to the moduli space of sheaves. 
In 2006, Behrend and Getzler \cite{BeGe} announced a development in this direction, which various papers in literature refers to (e.g. Toda \cite{Toda,Toda2}), but the paper has not yet been published. 
It says that the formal potential function $f$ for the cyclic dg Lie algebra L coming from the Schur objects in the derived category of coherent sheaves on Calabi--Yau 3-folds can be made to be convergent over a local neighborhood of the origin.
In \cite[Conj.\,1.2]{Toda2}, Toda formulates the derived categorical analog of \cite[Thm.\,5.5]{JoSo} and then Hua announced in \cite{Hua}  a joint work with Behrend \cite{BeHua} 
about the construction of the derived moduli space of complexes of coherent sheaves. In \cite{Hua}, Hua gives a construction of the global Chern-Simons functions for toric Calabi--Yau stacks of dimension three using strong exceptional collections. The moduli spaces of sheaves on such stacks can be identified with critical loci of these functions.
Still in the direction of derived categorical context, Chang and Li \cite{ChangLi} defined recently a semi--perfect obstruction theory and used it to construct virtual cycles of moduli of derived objects on Calabi--Yau 3-folds. In an other paper with Kiem \cite{KiLi2}, Li studied stable objects in derived category using a `$\C^*$-intrinsic blowup' strategy.
Finally, in 2013, the author et al. in \cite{BBBJ} completely answered the issue of presenting the moduli stack as a critical locus, and this opens now the question about possibilities to extend the whole project in \cite{JoSo} to triangulated categories, the main difficult of which would be to provide a generalization of wall-crossing formulae from abelian to triangulated categories, in the style of Joyce \cite{Joyc.1,Joyc.2,Joyc.3,Joyc.4,Joyc.5,Joyc.6,Joyc.7,Joyc.8}.

\smallskip
 
This discussion enlightens the fact that beyond this theory there is some deeper `derived' geometry:
as the deformation theory of coherent sheaves concerns the $\Ext$ groups, one way to talk about different geometric structures on moduli spaces is to ask what information they store about the $\Ext$ groups. For instance, in Kontsevich and Soibelman's context, an interesting problem, among others, is finding what kind of geometric structure on moduli spaces of coherent sheaves on a Calabi--Yau 3-fold X would be the most appropriate for doing motivic and categorified Donaldson--Thomas theory. As a consequence, a natural question would be to ask if derived algebraic geometry has something again to say about a theory of Donaldson--Thomas invariants for Calabi--Yau m-folds for $m>3,$ and what is the most suitable geometric structure to develop the theory, see Corollary \ref{da5cor2}.

\smallskip

Finally, due to both many unproved conjectures and exciting results, Kontsevich and Soibelman's motivic and categorified theory brings to life a fervid area of research (just to cite some investigators: Behrend, Bryan, Davison, Dimca, Mozgovoy, Nagao and Szendr\H oi). 
In the present work, we will not discuss much more this area, but we will come back to Kontsevich and Soibelman's theory later.

\bigskip

\begin{center} \textbf{Outline of the paper} \end{center}  

The paper begins with a section of background material on obstruction theories
and conventional definition of Donaldson--Thomas theory, Behrend functions and Behrend's approach to 
Donaldson--Thomas theory and, finally, Joyce and Song's and Kontsevich and Soibelman's generalization of Donaldson--Thomas theory.
This mainly aims to provide 
a soft introduction to Donaldson--Thomas theory and more specifically to Joyce's theory
and the scenery in which the following sections take place. 

\medskip

Subsection \ref{dt2} will briefly recall material from \cite{BeFa}, \cite{LiTi} and then \cite{Thom}.
This should provide a general picture about obstruction theories and the classical Donaldson--Thomas invariants. 
To say that a scheme $X$ has an {\it obstruction theory} \index{obstruction theory}
means, very roughly speaking, that one is given locally on $\fX$ an
equivalence class of morphisms of vector bundles such that at each
point the kernel of the induced linear map of vector spaces is the
tangent space to $X$, and the cokernel is a space of
obstructions.
Following Donaldson and Thomas
\cite[\S 3]{DoTh}, Thomas \cite{Thom} motivates a holomorphic Casson invariant 
\index{holomorphic Casson invariant} counting bundles 
on a Calabi--Yau 3-fold. He develops the deformation theory necessary 
to obtain the virtual moduli cycles in moduli spaces of stable sheaves 
whose higher obstruction groups vanish, which allows him to define the holomorphic 
Casson invariant of a Calabi--Yau 3-fold X and prove it is deformation invariant.
Heuristically, the Donaldson--Thomas moduli space is the 
critical set of the holomorphic Chern--Simons functional \index{holomorphic Chern--Simons functional} and 
the Donaldson--Thomas invariant is a holomorphic analogue of the Casson invariant.

\medskip

Subsection \ref{dt3} provides a more eclectic presentation of the Behrend function\index{Behrend function}. 
The first part will review the microlocal approach
to defining it, with a discussion on the attempt to categorify Donaldson--Thomas theory. In particular the section describes
the bridge between perverse sheaves and vanishing cycles on one hand, 
and Milnor fibres and Behrend functions on the other. 
Thus, if $\fM$ is the Donaldson--Thomas moduli space of stable sheaves, 
one can, heuristically, think of $\nu_{\fM}$ as the Euler characteristic of the perverse sheaf of vanishing cycles 
of the holomorphic Chern-Simons functional.
Following this philosophy in which perverse sheaves are categorification of constructible functions, 
the section outline the categorification program for Donaldson--Thomas theory. 
Then, in the second part, 
the Euler characteristic weighted by the Behrend function is compared to the unweighted Euler characteristic, 
motivating the necessity to introduce the Behrend function as a multiplicity function.
Finally, some properties are listed, in particular the Behrend approach to the Donaldson--Thomas invariants 
as weighted Euler characteristics and the formula in the complex setting of the Behrend function through {\it linking numbers}, which guarantee 
a more useful expression also in the case it is not known if the scheme admitting a symmetric 
obstruction theory can locally be written as the critical locus of a regular function on a smooth scheme. 
This is done introducing the definition of almost closed $1$-forms. We point out that Pandharipande and Thomas \cite{PaTh} give examples which are zeroes
of almost closed 1-forms, but are not locally critical loci, and this is the main indication that almost closed 1-forms are not `enough' to develop our whole program.  

\medskip

Subsection \ref{dt4} combines some results of Joyce's series of papers
\cite{Joyc.1,Joyc.2,Joyc.3,Joyc.4,Joyc.5,Joyc.6} with Behrend's approach to Donaldson--Thomas theory and 
describes how Joyce and Song developed the theory of generalized Donaldson--Thomas invariants in \cite{JoSo}.
The idea behind the entire project is that one should insert the Behrend
function $\nu_{\fM}$ of the moduli stack $\fM$ of coherent sheaves as a weight in
the Joyce's program. A good introduction to the book is provided by Joyce in \cite{Joyc.8}.
Then, a concluding remark presents a sketch on Kontsevich and Soibelman's generalization of Donaldson--Thomas theory. 
As the present paper is mainly concentrated on Joyce and Song's approach, the remark focuses on analogies and differences 
between the two theories rather than going into a detailed explanation of Kontsevich and Soibelman's program, both because it is 
beyond the author's competence and it is not directly involved in the results presented here. 

\medskip

Sections \ref{dcr}--\ref{ourpapers} presents briefly the main application in Donaldson--Thomas theory
coming from the vast project developed in the series of papers \cite{Joyc2,BBJ,BBDJS,BJM,BBBJ}. 
We first summarize the theory of d-critical
schemes and stacks introduced by Joyce \cite{Joyc2}. They form a new class of spaces, which should be regarded as classical
truncations of the $-1$-shifted symplectic derived schemes of
\cite{PTVV}. They are simpler than their derived analogues.
In \cite{BBJ}, we prove a Darboux theorem for derived schemes with symplectic forms
of degree $k<0$, in the sense of Pantev, To\"en, Vaqui\'e and
Vezzosi \cite{PTVV}. More precisely, we show that a derived scheme
$\bX$ with symplectic form $\ti\om$ of degree $k$ is locally
equivalent to $(\bSpec A,\om)$ for $\bSpec A$ an affine derived
scheme in which the cdga $A$ has Darboux-like coordinates with
respect to which the symplectic form $\om$ is standard, and in which
the differential in $A$ is given by a Poisson bracket with a
Hamiltonian function $H$ of degree $k+1$.
When $k=-1$, this implies that a $-1$-shifted symplectic derived
scheme $(\bX,\ti\om)$ is Zariski locally equivalent to the derived
critical locus $\bs\Crit(H)$ of a regular function $H:U\ra\bA^1$ on
a smooth scheme $U$. We use this to show that the classical scheme
$X=t_0(\bX)$ has the structure of an {\it algebraic d-critical
locus}, in the sense of Joyce~\cite{Joyc2}.
In the sequels \cite{BBBJ,BBDJS,BJM} we extend these results to (derived) Artin stacks, and discuss applications to categorified and motivic Donaldson--Thomas theory of Calabi--Yau 3-folds.

\medskip

Section \ref{main.1} states our main results, including the description of the 
local structure of the moduli stack of coherent sheaves on a Calabi--Yau 3-fold, the Behrend function identities 
and the deformation invariance of the theory. 
The section explains why and where Joyce and Song use the restriction $\K=\C$ in \cite{JoSo} and how 
our results overcome this restriction: \S\ref{dt.1} provides algebraic 
analogues of \cite[Thm. 5.5]{JoSo} and \cite[Thm. 5.11]{JoSo}. 
Finally \S\ref{def} provides the analogue of \cite[Cor. 5.28]{JoSo} which yields the deformation invariance over $\K$ of the generalized Donaldson--Thomas invariants $\bar{DT}{}^\al(\tau)$ defined for classes $\al$ varying in a deformation invariant lattice $\Lambda_X$ described below in which the numerical Grothendieck group injects through the Chern character map.
The section culminates in Theorem \ref{mainthm} which summarizes all these ideas.

\medskip

Subsection \ref{dt.1} proves the Behrend function identities
\index{Behrend function!Behrend identities} above
using the existence of a $T$-equivariant $d$-critical chart in the sense of \cite{Joyc2}
for each given point $E$ of $\fM,$ where $T\subset G$ is a maximal torus in $G,$ a maximal torus of $\Aut(E).$ This gives us the local description of the stack as a critical locus for a $T$-invariant regular function $f$ defined on a smooth scheme $U\subset \Ext^1(E,E).$ This method is valid for every locally global quotient stack, and in particular it provides the required local description of the moduli stack (Theorem \ref{dt5thm2}).
 \index{moduli stack!local structure}
Note that we actually would not need the assumption of the local quotient structure if we wanted to restrict just to sheaves on Calabi--Yau 3-folds, as this would follow from the standard method for constructing coarse moduli schemes of semistable coherent sheaves as in Huybrechts and Lehn
\cite{HuLe2}. More precisely, one can find a `good' local atlas for $\fM$ which is a G-invariant, locally closed $\K$-subscheme in 
the Grothendieck's Quot Scheme\index{Quot scheme}
$\Quot_X\bigl(\K^{P(n)}\ot\cO_X(-n),P\bigr)$, explained in \cite[\S
2.2]{HuLe2}, which parametrizes quotients
$\K^{P(n)}\ot\cO_X(-n)\ab\twoheadrightarrow E'$, where $E'$ has
Hilbert polynomial $P,$ and which is acted on by the $\K$-group $\GL(P(n),\K).$
From \cite{JoSo} it turns out that the proof of the first Behrend identity is reduced to an identity between the Behrend function of the zero locus of $\d f$, which is a $\K^*$-scheme, and the Behrend function of the 
fixed part of this zero locus, that is
$$
\nu_{\d f^{-1}(0)}(p)=(-1)^{\dim(T_p \d f^{-1}(0))-\dim(T_p (\d f^{-1}(0))^T)}\nu_{(\d f^{-1}(0))^T}(p),
$$
where $p$ is a point in the $\K^*$-fixed point locus $(\d f^{-1}(0))^T.$ This relation is a generalization of
the result in \cite{BeFa2} to the case that $p$ is not necessarily an isolated fixed point of the action and $\K$ is a general algebraically closed field of characteristic zero.
This argument is a different approach from the one suggested in a work by Li--Qin \cite{LiQin}, where
there they use some properties of the Thom classes of vector bundles.
The first Behrend function identity over algebraically closed fields $\K$ of characteristic zero 
follows from a trick in the argument of the second Behrend function identity proof,
which is directly proved over $\K,$ and
is based on Theorem \ref{blowup}, which is the algebraic version of \cite[Thm 4.11]{JoSo}. 

\medskip

Subsection \ref{def} yields that
it is possible to extend \cite[Cor. 5.28]{JoSo} on the deformation invariance of the generalized Donaldson--Thomas invariants in the compact case to algebraically closed fields $\K$ of characteristic zero.
First of all, using existence results by Grothendieck and Artin, and smoothness and properness properties of the {\it relative Picard scheme} in a family of Calabi--Yau 3-folds, one proves that the Picard groups form a local system. 
Moreover, it is a local system with finite monodromy, so it can be made trivial after passing to a finite \'etale cover of the base scheme, as formulated in the Theorem which is the algebraic generalization of \cite[Thm. 4.21]{JoSo}, and which studies the monodromy of the Picard scheme in a family instead of the numerical Grothendieck group.
Finally, Theorem \ref{definv}, a substitute for \cite[Thm. 4.19]{JoSo}, which does not need 
the integral Hodge conjecture result by \cite{Vois} over Calabi--Yau 3-folds 
and which is valid over $\K,$ characterizes the numerical Grothendieck 
group of a Calabi--Yau 3-fold in terms of a globally constant lattice described using the Picard scheme:
\begin{equation*}
\La_X=\ts\bigl\{ (\la_0,\la_1,\la_2,\la_3) 
\textrm{ where } \la_0,\la_3\in\Q, \; \la_1\in\Pic(X)\ot_{\Z}\Q, \; \la_2\in \Hom(\Pic(X),\Q)
\textrm{ such that } 
\end{equation*}
\begin{equation*}
\la_0\in \Z,\;\>
\la_1\in \Pic(X)/ _{\textrm{torsion}}, \;
 \la_2-\ha\la_1^2\in  \Hom(\Pic(X),\Z),\;\> \la_3+\ts\frac{1}{12}\la_1
c_2(TX)\in \Z\bigr\},
\end{equation*}
where $\la_1^2$ is defined as the map $\al\in\Pic(X)\ra \frac{1}{2}c_1(\la_1)\cdot c_1(\la_1)\cdot c_1(\al)\in A^3(X)_{\Q}\cong \Q,$
and $\frac{1}{12}\la_1 c_2(TX)$ is defined as $\frac{1}{12}c_1(\la_1)\cdot c_2(TX)\in A^3(X)_{\Q}\cong\Q.$
Theorem \ref{definv} proves that $\La_X$ is deformation invariant and the Chern character gives an injective morphism
$\ch:K^\num(\coh(X))\!\hookra\!\La_X$.
Our $\bar{DT}{}^\al(\tau)$ will be defined for classes $\al\in\Lambda_X.$ 

\medskip

Section \ref{dt7} sketches some implications of the theory and proposes new ideas for further research, in particular in the direction of derived categorical framework 
trying to establish a theory of generalized Donaldson--Thomas invariants for objects in the derived category of coherent sheaves, and for non necessarily compact Calabi--Yau 3-folds.

\bigskip

\noindent{\bf Acknowledgements.} I would like to thank Tom Bridgeland, Daniel Huybrechts, 
Frances Kirwan, Jun Li, Balazs Szendr\H oi, Richard Thomas and Bertrand To\"en 
for useful discussions and especially my supervisor Dominic Joyce 
for his continuous support, for many enlightening suggestions and valuable discussions. 
This research is part of my D.Phil. project funded by an 
EPSRC Studentship.

\section{Donaldson--Thomas theory: background material}

This section should be conceived as background picture in which next new sections should be allocated.
The competent reader can skip directly to \S\ref{main.1}.

\subsection{Obstruction theories and Donaldson--Thomas type invariants}
\label{dt2}

This section will briefly recall material from \cite{BeFa}, \cite{LiTi} and then \cite{Thom}
which provide both important notions used in the sequel and a hopefully interesting picture of Donaldson--Thomas theory.

\subsubsection{Obstruction theories}
\label{dt2.1} \index{obstruction theory|(}

Suppose that $X$ is a subscheme of a smooth scheme $M,$ cut out by a section $s$ of a rank $r$
vector bundle $E\ra M.$ Then the  {\it expected dimension}, or virtual dimension, of $X$ is $n-r,$ the dimension it would have if the section $s$ was transverse. If it is not transverse, one wants to take a correct $(n-r)$-cycle on $X.$ As the section $s$ induces a cone in $E_{|_X},$ one may then intersect this cone with the zero section of $X$ inside $E$ to get a cycle of expected dimension
on $X.$ The key observation is that one works entirely on $X$ and not in the ambient scheme $M.$
The deformation theory of the moduli problem is often endowed with the infinitesimal version of 
$s:M\ra E$ on $X,$ namely the linearization of $s,$ yielding the following exact sequence:
\begin{equation*}
\xymatrix@C=20pt@R=10pt{0\ar[r] & TX \ar[r] & TM_{|_X}  \ar[r]^{\rd s} & E_{|_X} \ar[r] & Ob \ar[r] & 0, }
\end{equation*}
for some cokernel $Ob,$ which in the moduli problem becomes the {\it obstruction sheaf}.\index{obstruction sheaf}

\smallskip

Moduli spaces in algebraic geometry often have an expected dimension \index{virtual dimension}
at each point, which is a lower bound for the dimension at that
point. Sometimes it may not
coincide with the actual dimension of the moduli space and sometimes 
it is not possible to get a space of the expected dimension. 
When one has a moduli space $X$ one obtains 
{\it numerical invariants}\index{numerical invariants}
by integrating certain cohomology classes over the virtual moduli cycle,
a class of the expected dimension
in its Chow ring.

\smallskip

One example is the moduli space of torsion-free,
semi-stable vector bundles on a surface which yields the {\it Donaldson theory} \index{Donaldson invariants}  
and which provides a set of
differential invariants of 4-manifolds. Another one is 
the moduli space of stable maps from curves of genus $g$ to
a fixed projective variety which yields the {\it Gromov--Witten invariants}\index{Gromov--Witten invariants},
a kind of generalization of the classical enumerative invariant which counts
the number of algebraic curves with appropriate constraints in a variety. 
In both cases, these invariants are intersection theories on the moduli spaces, respectively, 
of vector bundles over the surfaces, and of stable maps from curves to a variety.  
The fundamental class is the core of an intersection theory. However, for Gromov--Witten invariants for example, one cannot take
the fundamental class of the whole moduli space directly. The virtual moduli cycle, roughly speaking, plays the role 
of the fundamental class in an appropriate  ``good'' intersection theory. 

\smallskip

\index{cycle!of correct dimension}\index{excess intersection theory}
A nice picture to start with is the following situation:  when the expected dimension does not
coincide with the actual dimension of the moduli space, one may view this as if the moduli space is a subspace of an `ambient' space cut out
by a set of `equations' whose vanishing loci do not meet
transversely. Such a situation is well understood in the following
setting described in the Introduction of \cite{LiTi}: let $X$, $Y$ and $W$ 
be smooth varieties, $X,Y\ra W$ and let $Z=X\times_W Y.$
Then $[X]\cdot [Y]$,
the intersection of the cycle $[X]$ and $[Y]$, is a
cycle in $A_* W$ of dimension $\dim X+\dim Y-\dim W$.
When $\dim Z=\dim X+\dim Y-\dim W$, then $[Z]=[X]\cdot [Y]$.
Otherwise, $[Z]$ may not be $ [X]\cdot[Y]$. The {\it excess intersection
theory} gives that one can find a cycle in $A_* Z$ 
so that it is $[X]\cdot[Y]$. One may view this cycle as the virtual
cycle of $Z$ representing $[X]\cdot[Y]$. 
Following Fulton--MacPherson's normal cone \index{normal cone!Fulton--MacPherson's construction}
construction (in \cite{Fult,FuMacP1,FuMacP2}), this cycle is the image of
the cycle of the normal cone to $Z$\ in $X$, denoted
by $C_{Z/X}$, under the Gysin homomorphism 
$s^*: A_*( C_{Y/W}\times
_YZ) \ra A_* Z$, where $s: Z\ra C_{Y/W}\times_YZ$\ is the
zero section. This theory does not apply directly to moduli
schemes, since, except for some isolated cases, it is impossible
to find pairs $X\ra W$\ and $Y\ra W$ for smooth $X,Y$ and $W$ so that $X\times_WY$\
is the moduli space and $[X]\cdot[Y]$ so defined is the
virtual moduli cycle one needs.

\smallskip

Behrend and Fantechi \cite{BeFa} and Li and Tian \cite{LiTi} give two different approaches to deal with this.  
Very briefly, the strategy to Li and Tian's approach in \cite{LiTi} is that rather than
trying to find an embedding of the moduli space
into some ambient space, they will
construct a cone in a vector bundle directly, say $C\subset V$,
over the moduli space and then define the virtual moduli 
cycle to be $s^*[C]$, where $s$ is the zero section of $V$.
The pair $C\subset V$ will be constructed
based on a choice of the {\it tangent-obstruction complex} \index{tangent-obstruction complex}
of the moduli functor. The construction commutes with Gysin maps and carries a good invariance property.

\smallskip

In \cite{BeFa} Behrend and Fantechi introduce the notion of {\it cone stacks} \index{cone!cone stack} 
over a scheme $X$ (or more generally for Deligne--Mumford stacks). These are \index{Deligne--Mumford stack}
Artin stacks which are locally the quotient of a cone by a vector
bundle acting on it. They call a cone {\it abelian} \index{cone!abelian} if it is defined as \nomenclature[SpecSym]{$\Spec\Sym \sF$}{abelian cone associated to a coherent sheaf $\sF$}
$\Spec\Sym \sF$, where $\sF$ is a coherent sheaf on $X$. Every cone is
contained as a closed subcone in a minimal abelian one, which is called
its {\it abelian hull}. \index{cone!abelian hull}The notions of being abelian and of abelian hull
generalize immediately to cone stacks.
Then, for a complex $E^\bu$ in the derived category $D(X)$ of quasicoherent sheaves on $X$
\nomenclature[DX]{$D(X)$}{derived category of quasicoherent sheaves on $X$}
which satisfies some suitable assumptions (denoted by
($*$), see Definition \ref{dt1def1}), there is an associated abelian cone
stack $h^1/h^0((E^\bu)^\vee)$. \nomenclature[hh]{$h^1/h^0((E^\bu)^\vee)$}{abelian cone stack associated to a complex $E^\bu$}
In particular the {\it cotangent complex} 
\index{cotangent complex} $L_X ^\bu$ of $X$  constructed by Illusie
\cite{Illu1} (a helpful review is given in Illusie \cite[\S
1]{Illu2})  satisfies condition ($*$), so one can define the
abelian cone stack ${\mathfrak N}_X:=h^1/h^0((L_X^\bu)^\vee)$, the {\it
intrinsic normal sheaf}. \index{intrinsic normal sheaf} 
\nomenclature[N X]{${\mathfrak N}_X$}{intrinsic normal sheaf over a scheme $X$}
More directly, ${\mathfrak N}_X$ is constructed as follows: \'etale locally on $X$, embed
an open set $U$ of $X$ in a smooth scheme $W$, and take the stack
quotient of the {\it normal sheaf} \index{normal sheaf} (viewed as abelian cone) $N_{U/W}$ by the
natural action of $TW_{|_{U}}$. One can glue these abelian cone stacks
together to get ${\mathfrak N}_X$.  The {\it intrinsic normal cone} \index{intrinsic normal cone} 
${\mathfrak C}_X$ is the closed \nomenclature[C X]{${\mathfrak C}_X$}{intrinsic normal cone associated to a scheme $X$}
subcone stack of ${\mathfrak N}_X$ defined by replacing
$N_{U/W}$ by the {\it normal cone} \index{normal cone} $C_{U/W}$ in the previous construction.
In particular, the intrinsic normal
sheaf ${\mathfrak N}_X$ of $X$ carries the obstructions for
deformations of affine $X$-schemes. With this motivation, they introduce the notion of
{\it obstruction theory} \index{obstruction theory!definition} for $X$. To say that $X$ has an obstruction theory
means, very roughly speaking, that one is given locally on $X$ an
equivalence class of morphisms of vector bundles such that at each
point the kernel of the induced linear map of vector spaces is the
tangent space to $X$, and the cokernel is a space of
obstructions. That is, this is an object $E^\bu$ in the
derived category together with a morphism $E^\bu\ra L_X^\bu$,
satisfying Condition ($*$) and such that the induced
map ${\mathfrak N}_X\ra h^1/h^0((E^\bu)^\vee)$ is a closed immersion.
One denotes the sheaf $h^1({E^\bu}^\vee)$ by $Ob,$ the obstruction sheaf of the obstruction theory. 
It contains the obstructions to the smoothness of $X.$ 
When an obstruction theory $E^\bu$ is {\it perfect}, \index{obstruction theory!perfect}
${\mathfrak E}=h^1/h^0((E^\bu)^\vee)$ is a vector bundle stack. 
Once an obstruction theory is given, with the additional technical assumption that it admits a global resolution, one can define a virtual fundamental class of the expected dimension:
one has a vector
bundle stack ${\mathfrak E}$ with a closed subcone stack ${\mathfrak C}_X$, and to define
the virtual fundamental class of $X$ with respect to $E^\bu$ one simply
intersects ${\mathfrak C}_\fX$ with the zero section of ${\mathfrak E}$. To get round of the problem of dealing with 
Chow groups for Artin stacks, Behrend and Fantechi choose to
assume that $E^\bu$ is globally given by a homomorphism of vector
bundles $F^{-1}\ra F^0$. Then ${\mathfrak C}_X$ gives rise to a cone $C$ in
$F_1={F^{-1}}^\vee$ and one intersects $C$ with the zero section of
$F_1$ (see \cite{Kre} for a statement without this assumption).
 
 \smallskip
 
So, recall the following definitions from Behrend and
Fantechi~\cite{Behr,BeFa,BeFa2}:

\begin{dfn} Let $Y$ be a $\K$-scheme, and $D(Y)$ the derived
category of quasicoherent sheaves on $Y$.
\begin{itemize}
\setlength{\itemsep}{0pt}
\setlength{\parsep}{0pt}
\item[{\bf(a)}]  A complex $E^\bu\in D(Y)$ is {\it perfect of perfect
amplitude contained in\/} $[a,b]$, if \'{e}tale locally on $Y$,
$E^\bu$ is quasi-isomorphic to a complex of locally free sheaves
of finite rank in degrees $a,a+1,\ldots,b$.
\item[{\bf(b)}] A complex $E^\bu\in D(Y)$ {\it satisfies
condition\/} $(*)$ if
\begin{itemize}
\setlength{\itemsep}{0pt}
\setlength{\parsep}{0pt}
\item[(i)] $h^i(E^\bu)=0$ for all $i>0$,
\item[(ii)] $h^i(E^\bu)$ is coherent for $i=0,-1$.
\end{itemize}
\item[{\bf(c)}] An {\it obstruction theory\/}\index{obstruction
theory!definition} for $Y$ is a morphism $\varphi:E^\bu\ra L_Y$ in
$D(Y)$, where $L_Y=L_{Y/\Spec\K}$ is the cotangent complex of
$Y$, and $E$ satisfies condition $(*)$, and $h^0(\varphi)$ is an
isomorphism, and $h^{-1}(\varphi)$ is an epimorphism.
\item[{\bf(d)}] An obstruction theory $\varphi:E^\bu\ra L_Y$ is called
{\it perfect\/}\index{obstruction theory!perfect} if $E^\bu$ is
perfect of perfect amplitude contained in $[-1,0]$.
\item[{\bf(e)}] A perfect obstruction theory $\varphi:E^\bu\ra L_Y$ on
$Y$ is called {\it symmetric\/}\index{obstruction
theory!symmetric}\index{symmetric obstruction theory!definition}
if there exists an isomorphism $\vartheta:E^\bu\ra E^{\bu\vee}[1]$,
such that $\vartheta^{\vee}[1]=\vartheta$. Here
$E^{\bu\vee}\!=\!R\SHom(E^\bu,\cO_Y)$ is the {\it dual\/} of
$E^\bu$, and $\vartheta^\vee$ the dual morphism of~$\vartheta$.
\item[{\bf(f)}] If moreover $Y$ is a scheme with  a $G$-action, where $G$ is an algebraic group, an {\it
equivariant }perfect obstruction theory \index{obstruction theory!equivariant} is a morphism $E^\bu\ra L_Y$ in \index{algebraic group}
the category $D(Y)^G$, which is a perfect obstruction theory as a 
\nomenclature[DG]{$D(X)^G$}{derived category of equivariant quasicoherent $\cO_X$-modules}
morphism in $D(Y)$ (this definition is originally due to
Graber--Pandharipande~\cite{GP}). Here $D(Y)^G$ denotes the derived category of the abelian
category of $G$-equivariant quasicoherent $\cO_Y$-modules. 
\item[{\bf(g)}] A {\it symmetric equivariant }obstruction theory (or an {\it
equivariant symmetric }obstruction theory) \index{obstruction theory!equivariant symmetric} is a pair $(E^\bu\ra L_Y,E^\bu\ra
E^{\bu\vee}[1])$ of morphisms in the category $D(Y)^G$, such that $E^\bu\ra
L_Y$ is an equivariant perfect obstruction theory and $\vartheta:E^\bu\ra
E^{\bu\vee}[1]$ is an isomorphism satisfying $\vartheta^\vee[1]=\vartheta$ in $D(Y)^G.$
Note that this is more than requiring that the obstruction theory be
equivariant and symmetric, separately, as said in \cite{BeFa2}.
\end{itemize}

If instead $Y\stackrel{\psi}{\longra}U$ is a morphism of
$\K$-schemes, so $Y$ is a $U$-scheme, we define {\it relative} \index{obstruction theory!relative}
perfect obstruction theories $\phi:E^\bu\ra L_{Y/U}$ in the
obvious way.
\label{dt1def1}
\end{dfn}

Behrend and Fantechi \cite[Th.~4.5]{BeFa} prove the following theorem, which both
explains the term obstruction theory and provides a criterion for
verification in practice:

\begin{thm} The
following two conditions are equivalent for $E^\bu \in D(Y)$
satisfying condition $(*)$.
\begin{itemize}
\setlength{\itemsep}{0pt}
\setlength{\parsep}{0pt}
\item[{\bf(a)}] The morphism $\phi: E^\bu \ra L_Y$ is an
obstruction theory.
\item[{\bf(b)}] Suppose that we are given a square-zero extension $\ov{T}$ of $T$
with ideal sheaf $J$, with $T,\ov T$ affine, and a morphism $g:T \ra
Y.$ The morphism $\phi$ induces an element $\phi^*(\om(g))\in
\Ext^1(g^*E^\bu, J)$ from $\om(g)\in\Ext^1(g^*L_Y, J)$ by
composition. Then $\phi^*(\om(g))$ vanishes if and only if there
exists an extension $\ov{g}$ of\/ $g$. If it vanishes, then the
set of extensions form a torsor under~$\Hom(g^*E^\bu,J)$.
\end{itemize}
\end{thm}

Some examples can be found in  \cite{BeFa2}: Lagrangian intersections, sheaves on Calabi--Yau $3$-folds, stable maps to Calabi--Yau $3$-folds. Next section will concentrate on Donaldson--Thomas obstruction theory as in  \cite{Thom}. \index{obstruction theory|)}

\subsubsection{Donaldson--Thomas invariants of Calabi--Yau 3-folds}
\label{dt2.2}\index{Calabi--Yau 3-fold|(}\index{Donaldson--Thomas
invariants!original $DT^\al(\tau)$|(}

{\it Donaldson--Thomas invariants\/} $DT^\al(\tau)$ were defined by
Richard Thomas \cite{Thom}, following a proposal of Donaldson and
Thomas~\cite[\S 3]{DoTh}.
They are the virtual counts of stable
sheaves on Calabi--Yau 3-folds $X.$
Starting from the formal picture in which a Calabi--Yau $n$-fold is
the complex analogue of an oriented real $n$-manifold, and a Fano with
a fixed smooth anticanonical divisor is the analogue of a manifold
with boundary, Thomas motivates a holomorphic Casson invariant counting
bundles on a Calabi--Yau 3-fold. He develops the deformation theory
necessary to obtain the virtual moduli cycles in moduli spaces of stable sheaves whose higher
obstruction groups vanish which allows to define the holomorphic Casson invariant of a
Calabi--Yau 3-fold $X$, prove it is deformation invariant, and
compute it explicitly in some examples. Thus, heuristically, the Donaldson--Thomas moduli space is the critical set
of the holomorphic Chern-Simons functional and the Donaldson--Thomas
invariant is a holomorphic analogue of the Casson invariant. 
\index{holomorphic Chern--Simons functional} \index{holomorphic Casson invariant}

\smallskip

Mathematically, Donaldson--Thomas invariants are constructed as follows.
Deformation theory gives rise to a perfect
obstruction theory \cite{BeFa} (or a tangent-obstruction complex \index{obstruction theory}
in the language of \cite{LiTi}) on the moduli space of stable sheaves
$\M_\st^\al(\tau).$ Recall that Thomas supposes $\M_\st^\al(\tau)=\M_\rss^\al(\tau),$ that is, there are no strictly semistable
sheaves\/ $E$ in class\/ $\al,$ which implies the properness of $\M_\st^\al(\tau).$
As Thomas points out in \cite{Thom}, the obstruction sheaf is
equal to $\Omega_{\M_\st^\al(\tau)}$, the sheaf of K\"ahler differentials, and hence
the tangents $T_{\M_\st^\al(\tau)}$ are dual to the obstructions. This expresses
a certain symmetry of the obstruction theory on $\M_\st^\al(\tau)$ and is a mathematical
reflection of the heuristic that views $\M_\st^\al(\tau)$ as the critical locus of
a holomorphic functional.
Associated to the perfect obstruction theory is the virtual
fundamental class\index{virtual moduli cycle}, an element of the Chow group $A_*(\M_\st^\al(\tau))$ of \index{Chow group}
algebraic cycles modulo rational equivalence on $\M_\st^\al(\tau).$  One implication \index{algebraic cycles}
of the symmetry of the obstruction theory is the fact that the virtual
fundamental class $[\M_\st^\al(\tau)]^\vir$ is of degree zero. It can hence be
integrated over the proper space of stable sheaves to an integer, the
Donaldson--Thomas invariant or `virtual count' of $\M_\st^\al(\tau)$
\e
DT^\al(\tau)\quad=\ts\displaystyle \int_{\small{[\M_\st^\al(\tau)]^\vir}}1.
\label{dt2eq1}
\e
In fact Thomas did not define invariants $DT^\al(\tau)$ counting
sheaves with fixed class $\al\in K^\num(\coh(X))$, but coarser
invariants $DT^P(\tau)$ \nomenclature[DTP]{$DT^P(\tau)$}{Donaldson--Thomas invariants counting sheaves with fixed Hilbert polynomial}
counting sheaves with fixed Hilbert
polynomial\index{Hilbert polynomial} $P(t)\in\Q[t]$. \nomenclature[PHilb]{$P(t)$}{Hilbert polynomial $\in \Q[t]$}
Thus
$$
\M_\rss^P(\tau)\;\; = \displaystyle\coprod_{\small{\al:P_\al=P}}\M_\rss^\al(\tau) \quad\leadsto\quad DT^P(\tau) \quad=  \!\!\!\!\!\!\!\!\!   \ts\displaystyle \sum_{\small{\al\in K^\num(\coh(X)):P_\al=P}}   \!\!\!\!\!\!\!\!\!  DT^\al(\tau),
$$ is the
relationship with Joyce and Song's version $DT^\al(\tau)$ reviewed in \S\ref{dt4}, where the r.h.s. has only finitely many nonzero terms in the sum.
Here, Thomas' main
result \cite[\S 3]{Thom}: \index{deformation-invariance}
\begin{thm} For each Hilbert polynomial $P(t),$ the invariant\/
$DT^P(\tau)$ is unchanged by continuous deformations of the
underlying Calabi--Yau $3$-fold~$X$ over $\K.$
\label{dt2thm1}
\end{thm}
The same proof shows that $DT^\al(\tau)$ for $\al\in
K^\num(\coh(X))$ is deformation-invariant, {\it provided\/} it is known
that the group $K^\num(\coh(X))$ is deformation-invariant, so that
this statement makes sense. 
This issue is discussed in
\cite[\S 4.5]{JoSo}. There, it is shown that when $\K=\C$ one can describe
$K^\num(\coh(X))$ in terms of cohomology groups \index{Grothendieck group!numerical}
$H^*(X;\Z),$ $H^*(X;\Q)$, so that $K^\num(\coh(X))$ is manifestly
deformation-invariant, and therefore $DT^\al(\tau)$ is also
deformation-invariant.  \index{deformation-invariance}
Theorem \cite[Thm. 4.19]{JoSo} crucially uses the integral Hodge conjecture result by \cite{Vois} for Calabi--Yau 3-folds over $\C.$
\index{Hodge conjecture}\index{Picard scheme}
In \cite[Rmk 4.20(e)]{JoSo}, Joyce and Song propose to extend that description
over an algebraically closed base field $\K$ of
characteristic zero by replacing $H^*(X;\Q)$ by the {\it algebraic
de Rham cohomology\/}\index{algebraic de Rham cohomology} $H^*_{\rm
dR}(X)$ \nomenclature[HalgDR]{$H^*_{\rm dR}(X)$}{algebraic de Rham cohomology of a smooth projective $\K$-scheme $X$} 
of Hartshorne \cite{Hart1}. For $X$ a smooth projective
$\K$-scheme, $H^*_{\rm dR}(X)$ is a finite-dimensional vector space
over $\K$. There is a Chern character map \index{Chern character}
$\ch:K^\num(\coh(X))\hookra H^{\rm even}_{\rm dR}(X)$. In \cite[\S
4]{Hart1}, Hartshorne considers how $H^*_{\rm dR}(X_t)$ varies in
families $X_t:t\in T$, and defines a {\it Gauss--Manin connection}, which \index{Gauss--Manin connection}
makes sense of $H^*_{\rm dR}(X_t)$ being locally constant in~$t$.
In \S\ref{def}  we will use another idea to characterize the numerical Grothendieck group of a Calabi--Yau 3-fold in terms of a globally constant lattice described using the Picard scheme. \index{Picard scheme}

\smallskip

Next section will introduce the Behrend function and the work done by Behrend in \cite{Behr}, 
which has been crucial for the development of Donaldson--Thomas theory.  \index{Calabi--Yau 3-fold|)}\index{Donaldson--Thomas
invariants!original $DT^\al(\tau)$|)}

\subsection{Microlocal geometry and the Behrend function}
\label{dt3} \index{microlocal geometry|(}\index{Behrend function|(}

This section briefly explains Behrend's approach \cite{Behr} to Donaldson--Thomas invariants as 
Euler characteristics of moduli schemes weighted by the Behrend function.
It was introduced by Behrend \cite{Behr} for finite type $\C$-schemes
$X$; in \cite[\S 4.1]{JoSo} it has been generalized to Artin $\K$-stacks. Behrend functions are also defined for complex
analytic spaces $X_\an$, and the Behrend function of a $\C$-scheme
$X$ coincides with that of the underlying complex analytic
space~$X_\an$. 
The theory is also valid for $\K$-schemes acted on by a reductive linear algebraic group.
A good reference for this section, other than the original paper by Behrend \cite{Behr}, are \cite[\S4]{JoSo} and \cite{O} for the equivariant 
version.

\subsubsection{Microlocal approach to the Behrend function}
\label{dt3.1}

In \cite{Behr}, Behrend suggests a microlocal approach to the problem. The first part of the discussion 
describes how the Behrend function is defined while the second part, although not detailed and 
not directly involved in the rest of the paper, aim to give a more complete picture.

\paragraph{The definition of the Behrend function.}
\label{dt3.1.1}

Let $\K$ be an algebraically closed field of
characteristic zero, and $X$ a finite type $\K$-scheme. Suppose $X\hookra M$ is an embedding of $X$ as a closed
subscheme of a smooth $\K$-scheme $M$. Then one has a commutative diagram 
\e
\begin{gathered}
\xymatrix@C=70pt@R=20pt{ Z_*(X)  \ar[dr]_{c_0^M}   \ar[r]_\cong^\Eu & \CF_\Z(X)  \ar[d]^{c_0^{SM}} \ar[r]_\cong^{\Ch} & \cL_X(M) \ar[dl]^{0^!}\\
 & A_0(X) & }
\label{dt3eq1}
\end{gathered}
\e\index{algebraic cycles}\nomenclature[CF]{$\CF_\Z(X)$}{group of $\Z$-valued constructible functions on $X$ as in \cite{Joyc.1}}\index{constructible function}\index{local Euler obstruction}\nomenclature[Eu]{$\Eu$}{the `local Euler
obstruction', an isomorphism $Z_*(X)\ra\CF_\Z(X)$} 
where the two horizontal arrows are isomorphisms. Here $Z_*(X)$ 
denotes the group of {\it algebraic cycles\/} on $X$, as in Fulton \cite{Fult}, and $\CF_\Z(X)$ the group of $\Z$-valued constructible 
functions on $X$ in the sense of \cite{Joyc.1}.  The {\it local Euler
obstruction\/} is a group isomorphism 
$\Eu:Z_*(X)\ra\CF_\Z(X)$.
The local Euler obstruction was first
defined by MacPherson  \cite{MacP}  to solve the problem of existence of covariantly functorial Chern classes,
answering thus a Deligne--Grothendieck conjecture when $\K=\C$, using complex \index{Deligne--Grothendieck conjecture}
analysis, but Gonzalez--Sprinberg \cite{GS} provides an alternative algebraic
definition which works over any algebraically closed field $\K$ of
characteristic zero. It is the obstruction to extending 
a certain section of the tautological bundle on the {\it Nash blowup}. \index{Nash blowup}
More precisely, if $V$ is a prime cycle on $X$, the
constructible function $\Eu(V)$ is given by
\begin{equation*}
\ts\Eu(V):x\longmapsto\displaystyle \int_{\mu^{-1}(x)}c(\ti T)\cap
s(\mu^{-1}(x),\ti V),
\end{equation*}
where $\mu:\ti V\ra V$ is the Nash blowup of $V$, $\ti T$ the dual
of the universal quotient bundle, $c$ the total Chern class and $s$
the Segre class of the normal cone to a closed immersion. Kennedy
\cite[Lem. 4]{Kenn} proves that $\Eu(V)$ is
constructible.\index{constructible function} 

\smallskip

As pointed out in the next section, it is worth observing that independently, 
at about the same time, 
Kashiwara proved an {\it index theorem} over $\C$ for a 
holonomic $\mathcal{D}$-module relating its local Euler characteristic and 
the local Euler obstruction with respect to an appropriate stratification (see \cite{Gin} for details). 
It coincides with the one defined above and this is equivalent to saying that the diagram \eq{dt3eq2} below commutes. 
\index{index theorem!microlocal index theorem}\index{Euler characteristic}

\smallskip

Observe that this part of the diagram exists
without the  embedding into $M$ and is sufficient to give the definition of the Behrend function as follow.
Let $C_{X/M}$ be the {\it normal
cone\/} of $X$ in $M$, as in \cite[p.73]{Fult}, and $\pi:C_{X/M}\ra X$
the projection. As in \cite[\S 1.1]{Behr}, define a cycle
${\mathfrak C}_{X/M}\in Z_*(X)$ by
$$
{\mathfrak C}_{X/M}=\ts\displaystyle\sum_{C'}(-1)^{\dim\pi(C')}{\rm
mult}(C')\pi(C'),
$$
where the sum is over all irreducible components $C'$ of $C_{X/M}$.
It turns out that ${\mathfrak C}_{X/M}$ depends only on $X$, and not
on the embedding $X\hookra M$. Behrend \cite[Prop. 1.1]{Behr} proves
that given a finite type $\K$-scheme $X$, there exists a unique
cycle ${\mathfrak C}_X\in Z_*(X)$, such that for any \'etale map
$\vp:U\ra X$ for a $\K$-scheme $U$ and any closed embedding
$U\hookra M$ into a smooth $\K$-scheme $M$, one has
$\vp^*({\mathfrak C}_X)={\mathfrak C}_{U/M}$ in $Z_*(U)$. If $X$ is
a subscheme of a smooth $M$ one takes $U=X$ and get ${\mathfrak
C}_X={\mathfrak C}_{X/M}$. Behrend calls ${\mathfrak C}_X$ the {\it
signed support of the intrinsic normal cone}, or the {\it
distinguished cycle} of~$X$. 
For each finite type
$\K$-scheme $X$, define the {\it Behrend function} $\nu_X$ in
$\CF_{\Z}(X)$ by $\nu_X=\Eu({\mathfrak C}_X)$, as in Behrend~\cite[\S
1.2]{Behr}.\index{Behrend function!definition}\index{intrinsic normal cone!signed support}\index{distinguished cycle}

\medskip\index{conormal bundle} \index{characteristic cycle map} \nomenclature[char]{$\Ch$}{the characteristic cycle map $\Ch:\CF_{\Z}(U)\ra \cL(U)$}

For completeness, the section now describes the other side of the diagram \eq{dt3eq1}, which yields another possible way to define the Behrend function. Write $\cL_X(M)$ for the free abelian \index{Lagrangian cycle!conical Lagrangian cycle}
\nomenclature[L]{$\cL_X(M)$}{free abelian group generated by closed, 
irreducible, reduced, conical Lagrangian, $\K$-subvariety in $T^*M$ lying over cycles contained in $X$}
group generated by closed, irreducible, reduced, conical Lagrangian, $\K$-subvariety in $\Omega_M$ lying over cycles contained in $X.$ 
The isomorphism $\Ch:\CF_{\Z}(X)\ra\cL_X(M)$ maps a constructible function to
its characteristic cycle, which is a conic Lagrangian cycle on \index{characteristic cycle map}
$\Omega_M$ supported inside $X$ defined in the following way. 
Consider the commutative diagram of group isomorphisms that fits in the diagram \eq{dt3eq1}:
\e
\xymatrix@R=10pt@C=50pt{ Z_*(M) \ar[r]^\Eu \ar@/_.7pc/[rr]_L & \CF_{\Z}(M) \ar[r]^\Ch  & \cL(M).}
\label{dt6eqq1}
\e
Here $L:Z_*(M)\ra \cL(M)$ is 
defined on any prime cycle $V$ by $L:V\ra (-1)^{\dim(V)}\ell(V),$ where $\ell(V)$ is the closure of 
the conormal bundle of any nonsingular dense open subset 
of $V.$ Then $\Eu,$ $L$ are isomorphisms, 
and the {\it characteristic cycle map}
$\Ch:\CF_{\Z}(M)\ra \cL(M)\subset Z_{\dim M}(\Om_M)$ is defined to be
the unique isomorphism making (\ref{dt6eqq1}) commute. 
In the complex case Ginsburg \cite{Gin} describes the inverse of this 
map as {\it intersection multiplicity} between two conical Lagrangian cycles. 
This formula is crucial in \cite[\S 4.3]{Behr}, where Behrend gives an expression for the Behrend function 
in terms of linking numbers, which \index{linking number}
has a validity also in the case it is not known if a scheme admitting a symmetric obstruction theory can locally
be written as the critical locus of a regular function on a smooth scheme (Theorem \ref{dt6prop1}). See also \cite[Ex. 19.2.4]{Fult}.\index{characteristic cycle map!inverse}\index{intersection multiplicity}

\medskip\index{Chern-Mather class} \index{Schwartz-MacPherson Chern class}
\nomenclature[cm]{$c^M(V)$}{Mather class of an algebraic cycle $V$} \index{algebraic cycles}

The maps to $A_0(X)$ are the degree
zero  {\it Chern-Mather class}, the degree zero 
{\it Schwartz-MacPherson Chern class}, and the intersection with the zero
section, respectively. 
The Mather class is a homomorphism $c^M:Z_*(X)\ra A_*(X),$ whose definition is
a globalization of the construction of the local Euler obstruction.
One has
$c^M(V)=\mu_*\big(c(\widetilde T)\cap[\widetilde V]\big)\,,$ for a prime cycle $V$ of degree $p$ on $X$
with the same notation as above. For a the expression in terms of normal cones, see for example \cite[\S 1]{Sabb}.
Applying $c^M$ to the cycle ${\mathfrak C}_X$, one obtains the 
{\it Aluffi class} \index{Aluffi class} 
$\alpha_X=c^M({\mathfrak C}_X)\in A_*(X)$ defined in \cite{Alu}. 
\nomenclature[11]{$\al_X$}{Aluffi class of $X$}
If $X$ is smooth, its Aluffi class equals 
$\alpha_X=c(\Omega_X)\cap[X]\,.$

\medskip

Now given a symmetric obstruction theory on $X$, the {\it cone of curvilinear 
obstructions} $cv\hookrightarrow ob=\Omega_X$, pulls back to a cone in \index{cone!of curvilinear obstructions}
\nomenclature[cv]{cv}{cone of curvilinear obstructions}
$\Omega_{M_{|_{X}}}$ via the epimorphism $\Omega_{M_{|_{X}}}\ra \Omega_X.$  Via the
embedding $\Omega_{M_{|_{X}}}\hookrightarrow \Omega_M$ one obtains a conic
subscheme $C\hookrightarrow \Omega_M$, the {\it obstruction cone }for \index{cone!obstruction cone}
the embedding $X\hookrightarrow M$.  Behrend proves that the virtual fundamental class is
$[X]^\vir=0^![C]$. \nomenclature[csm]{$c^{SM}(f)$}{Schwartz-MacPherson Chern class of a constructible function $f$}
The key fact is that $C$ is {\it Lagrangian}. Because of this, there
exists a unique constructible function $\nu_X$ on $X$ such that
$\Ch(\nu_X)=[C]$ and $c_0^{SM}(\nu_X)=[X]^\vir$.  Then 
Theorem \ref{dt3thm4} below follows as an application of MacPherson's
theorem~\cite{MacP} (or equivalently from the microlocal index theorem of \index{index theorem!microlocal index theorem}
Kashiwara~\cite{KaSc}), which one can think of as a kind of generalization of
the {\it Gauss--Bonnet theorem} to singular schemes. \index{Gauss--Bonnet theorem}
See Theorem \ref{dt3thm4} below for its validity over $\K.$
The cycle ${\mathfrak C}_X$ such that $\Eu({\mathfrak C}_X)=\nu_X$ is as defined above, the (signed) support of the intrinsic
normal cone of $X$.
The Aluffi class $\alpha_X=c^M({\mathfrak C}_X)=c^{SM}(\nu_X)$ has thus the property that its degree zero
component is the virtual fundamental class of any symmetric
obstruction theory on $X.$ \index{virtual moduli cycle} \index{obstruction theory!symmetric}

\medskip\index{Zariski topology}\nomenclature[Xan]{$X_\an$}{complex analytic space $X_\an$
underlying $X$}
\index{Behrend function!algebraic} \index{Behrend function!analytic}

In the case $\K=\C$, using MacPherson's complex analytic definition
of the local Euler obstruction \cite{MacP}, the definition of
$\nu_X$ makes sense in the framework of complex analytic geometry,
and so Behrend functions can be defined for complex analytic
spaces $X_\an$.\index{complex analytic space} Thus, as in \cite[Prop. 4.2]{JoSo} one has
that if $X$ is a finite type $\K$-scheme, then the Behrend function $\nu_X$ is a well-defined\/ $\Z$-valued
constructible function on $X,$ in the Zariski topology. 
If $Y$ is a complex analytic space then the Behrend function $\nu_Y$ is a well-defined\/ $\Z$-valued locally
constructible function on $Y,$ in the analytic topology. \index{analytic topology}
Finally, if $X$ is a finite type $\C$-scheme, with
underlying complex analytic space $X_\an,$ then the algebraic
Behrend function $\nu_X$ and the analytic Behrend
function $\nu_{\smash{X_\an}}$ coincide. In particular,
$\nu_X$ depends only on the complex analytic space $X_\an$
underlying $X,$ locally in the analytic topology. 
Finally, the definition of Behrend functions is valid over $\K$-schemes, 
algebraic $\K$-spaces and Artin $\K$-stacks, locally of finite type (see \cite[Prop. 4.4]{JoSo}).
\index{Behrend function!of Artin stack}

\paragraph{Categorifying the theory.}
\label{dt3.1.2} \index{categorification|(} 

What follows will not be needed to understand the rest of the paper. We include this material both for completeness, as it underlies the theory of Behrend functions, and also 
because it is one of the main application of the whole program \cite{Joyc2,BBJ,BBDJS,BJM,BBBJ} explained in \S\ref{ourpapers}.

\medskip

For this paragraph, restrict to $\K=\C$ for simplicity. 
There exists a sophisticated modern theory of linear partial differential equations
on a smooth complex algebraic variety $X,$ sometimes called {\it microlocal analysis}, because it involves analysis
on the cotangent bundle $T^* X; $ this yields a theory which is invariant
with respect to the action of the whole group of canonical transformation 
of $T^* X$ while the usual theory is only invariant
under the subgroup induced by 
diffeomorphism of $X.$ It is sometimes called $\mathcal{D}$-{\it module theory,} 
\index{sheaf of rings of holomorphic linear partial differential operators of finite order} 
because it involves sheaves of modules
$\mathcal{M}$ over the sheaf of rings of holomorphic linear partial differential 
operators of finite order $\mathcal{D}=\mathcal{D}_X;$
\nomenclature[D]{$\mathcal{D}_X$}{sheaf of rings of holomorphic linear partial differential operators of finite order}
these rings are noncommutative, left and right Noetherian, and have finite global homological dimension.    
It is also sometimes called {\it algebraic analysis} because it involves such algebraic constructions
as $\Ext^i_\mathcal{D}(\mathcal{M},\mathcal{N}).$ The theory as it is known today 
grew out of the work done in the 1960s by the school of 
Mikio Sato in Japan. During the 1970's, one of the central themes 
in $\mathcal{D}$-module theory was David Hilbert's twenty-first problem, 
now called the \index{Riemann-Hilbert problem}
{\it Riemann-Hilbert problem.} A generalization of it may be stated as the 
problem to solve the {\it Riemann-Hilbert correspondence,} which, roughly speaking, 
\index{Riemann-Hilbert correspondence}
describes the nature of the correspondence between a system 
of differential equations and its solutions. A comprehensive reference is the 
book of Kashiwara and Shapira \cite{KaSc}, while an interesting eclectic vision 
on the subject is provided by Ginsburg  \cite{Gin}.  
One has the following commutative diagram:  
\e
\begin{gathered}
\xymatrix@C=60pt@R=25pt{ \textrm{(perverse) constructible sheaves}  \ar[d]_{\chi}    &  \ar[l]_\sim^{DR} \textrm{(regular) holonomic modules} \ar[d]^{SS}\\
\textrm{constructible functions}   \ar[r]_\sim^{\Ch}  &  \textrm{Lagrangian cycles in }T^* X.}
\label{dt3eq2}
\end{gathered}
\e
\index{constructible sheaf!perverse} \index{holonomic modules!regular} \index{characteristic cycle map}
\nomenclature[SS]{$SS$}{characteristic cycle map} \index{characteristic cycle} \index{characteristic variety}
Recall that here $SS$ denotes the {\it characteristic cycle map} which to a $\mathcal{D}$-module $\mathcal{M}$  
associates its  {\it characteristic cycle.} It is the formal linear combination
of irreducible components of the {\it characteristic variety} 
(the support of the graded sheaf gr$\mathcal{M}$ associated to $\mathcal{M}$)
counted with their multiplicities. It looks like $$SS(\mathcal{M})=\sum m_\al(\mathcal{M})\cdot \overline{T^*_{X_\al}X}$$ 
for a stratification $\{X_\al \}$ of $X,$ where $m_\al(\mathcal{M})$ are 
positive integers and $ \overline{T^*_{X_\al}X}$ is the 
closure of the conormal bundle $T^*_{X_\al}X.$
Each component of the characteristic variety has dimension at least $\dim(X).$ A  $\mathcal{D}$-module $\mathcal{M}$  is called 
{\it holonomic} if its characteristic variety is pure of dimension $\dim(X).$ To have also {\it regular singularities} 
means, very roughly speaking, that the system is determined by its principal symbol. 

\smallskip

So, to a holonomic system it has been associated an object of microlocal nature, the characteristic cycle. 
On the other side, the Riemann-Hilbert correspondence associates to an holonomic system $\mathcal{M}$ 
its {\it De Rham complex,} \index{De Rham complex} 
\nomenclature[DRM]{$\textrm{DR}\mathcal{M}$}{De Rham complex of an holonomic system}
$$
\xymatrix@C=40pt@R=10pt{\textrm{DR}(\mathcal{M}): \; 0 \ar[r] & \Om^0(\mathcal{M}) \ar[r]^{d\;} &  \Om^1(\mathcal{M}) \ar[r]^{\;d} & \ldots \ar[r]^{d\;\;\;\;\;\;\;\;} & \Om^{\dim(X)}(\mathcal{M}) \ar[r]^{\;\;\;\;\;\;\; d} & 0,}
$$
where $ \Om^p(\mathcal{M})$ is the sheaf of $\mathcal{M}$-valued $p$-forms on $X$ and $d$ is the differential defined by Cartan formula. 
As an object in the derived category it can be expressed as $\textrm{DR}(\mathcal{M})= \textrm{R}\SHom_{\mathcal{D}_X}(\cO_X,\mathcal{M})[\dim(X)].$
If $\mathcal{M}$ is holonomic, $\textrm{DR}(\mathcal{M})$ is constructible and determines $\mathcal{M}$ provided
that the latter has regular singularities. Recall the following definition (see also \cite[\S 4]{JoSo}):

\begin{dfn} Let $X$ be a complex analytic space. Consider sheaves of
$\Q$-modules $\cal C$ on $X$. Note that these are {\it not\/}
coherent sheaves, which are sheaves of $\cO_X$-modules. A sheaf
$\cal C$ is called {\it
constructible\/}\index{sheaf!constructible}\index{constructible
sheaf} if there is a locally finite stratification $X=\bigcup_{j\in
J}X_j$ of $X$ in the complex analytic topology, such that ${\cal 
C}\vert_{X_j}$ is a $\Q$-local system for all $j\in J$, and all the \index{analytic topology}
stalks ${\cal C}_x$ for $x\in X$ are finite-dimensional $\Q$-vector
spaces. A complex ${\cal C}^\bu$ of sheaves of $\Q$-modules on $X$
is called {\it constructible\/} if all its cohomology sheaves
$H^i({\cal C}^\bu)$ for $i\in\Z$ are constructible. \index{constructible complex}
Write $D^b_\Con(X)$\nomenclature[DbCon(X)]{$D^b_\Con(X)$}{bounded
derived category of constructible complexes on $X$} for the bounded
derived category of constructible complexes on $X$. It is a
triangulated category. By \cite[Thm. 4.1.5]{Dimc}, $D^b_\Con(X)$ is
closed under Grothendieck's ``six operations on
sheaves''\index{sheaf!Grothendieck's six operations}
$R\vp_*,R\vp_!,\vp^*,\vp^!,{\cal
RH}om,\smash{\mathop{\otimes}\limits^{\sst L}}$. The {\it perverse
sheaves\/} on $X$ are a particular abelian subcategory
$\Per(X)$\nomenclature[Per(X)]{$\Per(X)$}{abelian category of
perverse sheaves on $X$} in $D^b_\Con(X)$, which is the heart of a
t-structure on $D^b_\Con(X)$. So perverse sheaves are actually
complexes of sheaves, not sheaves, on $X$. The category $\Per(X)$ is
noetherian\index{noetherian}\index{abelian category!noetherian} and
locally artinian, and is artinian\index{artinian}\index{abelian
category!artinian} if $X$ is of finite type, so every perverse sheaf  \index{perverse sheaf}
has (locally) a unique filtration whose quotients are {\it simple} \index{perverse sheaf!simple}
perverse sheaves; and the simple perverse sheaves can be described
completely in terms of irreducible local systems on irreducible
subvarieties in~$X$.
\label{dt3def1}
\end{dfn}

Now, given a constructible sheaf ${\cal C}^\bu$ there is associated a constructible function on $X$: 
define a map \nomenclature[1wchX]{$\chi_X$}{constructible function on $X$ associated to a constructible sheaf}
$\chi_X: \Obj(D^b_\Con(X))\ra\CF_\Z^\an(X)$ by taking Euler
characteristics of the cohomology of stalks of complexes, given by
\begin{equation*}
\chi_X({\cal C}^\bu):x\longmapsto \ts\displaystyle \sum_{k\in\Z}(-1)^k\dim{\cal
H}^k({\cal C}^\bu)_x.
\end{equation*}
Since distinguished triangles in $D^b_\Con(X)$ give long exact
sequences on cohomology of stalks ${\cal H}^k(-)_x$, this $\chi_X$
is additive over distinguished triangles, and so descends to a group
morphism $\chi_X:K_0(D^b_\Con(X))\ra \CF_\Z^\an(X)$.
These maps $\chi_X:\Obj(D^b_\Con(X))\ra\CF_\Z^\an(X)$ and $\chi_X:
K_0(D^b_\Con(X))\ra \CF_\Z^\an(X)$ are surjective, since
$\CF_\Z^\an(X)$ is spanned by the characteristic functions of closed
analytic cycles $Y$ in $X$, and each such $Y$ lifts to a perverse
sheaf in $D^b_\Con(X)$. In category-theoretic terms, $X\mapsto
D^b_\Con(X)$ is a functor $D^b_\Con$ from complex analytic spaces to
triangulated categories, and $X\mapsto \CF_\Z^\an(X)$ is a functor
$\CF_\Z^\an$ from complex analytic spaces to abelian groups, and
$X\mapsto\chi_X$ is a natural transformation $\chi$ from $D^b_\Con$ \index{natural transformation}
to~$\CF_\Z^\an$.

\smallskip

For a holonomic $\mathcal{D}$-module
$\mathcal{M}$ one sets $\chi(x,\mathcal{M})=\chi(x,\textrm{DR}(\mathcal{M})).$  
Thus, if $\mathcal{M}$ is a regular holonomic $\mathcal{D}$-module on $X\subset M,$ with $M$ smooth, 
whose characteristic cycle is $[C_{X/M}]$, then
$$\nu_X(P)=\sum_i (-1)^i \dim_\C H^i_{\{P\}}(X,\mathcal{M}_{DR})\,,$$ for any point $P\in M$.  Here $H^i_{\{P\}}$ denotes cohomology with
supports in the subscheme $\{P\}\hookrightarrow M$ and
$\mathcal{M}_{DR}$ denotes the perverse sheaf associated to
\nomenclature[MDR]{$\mathcal{M}_{DR}$}{perverse sheaf associated to a 
regular holonomic $\mathcal{D}$-module $\mathcal{M}$ via the Riemann-Hilbert correspondence}
$\mathcal{M}$ via the Riemann-Hilbert correspondence, as incarnated,
for example, by the De~Rham complex $\textrm{DR}(\mathcal{M}).$ 
At the moment, Kai Behrend is attempting to give explicit constructions in some cases (see \cite{GeBaViLa}).

\smallskip

In the case $X$ is the critical scheme of a regular function $f$ on a smooth scheme $M,$ Behrend 
\cite{Behr} gives the following expression for the Behrend function due to Parusi\'nski and Pragacz \cite{PaPr}.
This formula has been crucial in  \cite{JoSo}. For the definition of the {\it Milnor fibres\/} 
\index{Milnor fibre} for holomorphic functions on complex \index{vanishing cycle}
analytic spaces and the a review on  {\it vanishing cycles\/} a survey paper on the subject is Massey \cite{Mass}, and
three books are Kashiwara and Schapira \cite{KaSc}, Dimca
\cite{Dimc}, and Sch\"urmann \cite{Schur}. Over the field $\C$,
Saito's theory of {\it mixed Hodge modules\/}\index{mixed Hodge module}
\cite{Sait} provides a generalization of the theory of perverse
sheaves with more structure, which may also be a context in which to
generalize Donaldson--Thomas theory. 

\begin{thm} Let\/ $U$ be a complex manifold of dimension $n,$ and\/ $f:U\ra\C$ a
holomorphic function, and define $X$ to be the complex analytic
space $\Crit(f)$ contained in
$U_0=f^{-1}(\{0\}).$ Then the
Behrend function $\nu_X$ of\/ $X$ is given by
\e
\nu_X(x)=(-1)^{\dim U}\bigl(1-\chi(MF_f(x))\bigr) \qquad\text{for
$x\in X$.}
\label{dt3eq3}
\e
Moreover, 
the perverse sheaf of vanishing
cycles\index{vanishing cycle!perverse sheaf}\index{perverse sheaf!of
vanishing cycles} $\phi_f(\underline{\Q}[n-1])$ on $U_0$ is
supported on $X,$ and
\nomenclature[1vphif]{$\phi_f$}{vanishing cycle functor on derived category of constructible sheaves}
\nomenclature[MF]{$MF_f(x)$}{Milnor fibre of a holomorphic function $f$ at point $x$}
\e
\chi_{U_0}\bigl(\phi_f(\underline{\Q}[n-1])\bigr)(x)=\begin{cases}
\nu_X(x), & x\in X, \\ 0, & x\in U_0\sm X, \end{cases}
\label{dt3eq4}
\e
where $\nu_X$ is the Behrend function of the complex analytic
space~$X$.
\label{dt3thm1}
\end{thm}

Thus, if $X$ is the Donaldson--Thomas moduli space of stable sheaves,
one can, heuristically, think of $\nu_X$ as the {\it Euler characteristic
of the perverse sheaf of vanishing cycles of the holomorphic
Chern-Simons functional.}

\medskip

In \cite[Question 4.18, 5.7]{JoSo}, Joyce and Song ask the following question. 
\begin{quest}{\bf(a)} Let\/ $X$ be a Calabi--Yau\/ $3$-fold over\/
$\C,$ and write\/ $\M_\rsi$ 
\nomenclature[M al si]{$\M_\rsi$}{coarse moduli space of simple coherent sheaves on $X$}
\index{coherent sheaf!simple}
for the coarse moduli space of simple
coherent sheaves on\/ $X$. Does there exist a natural perverse \index{perverse sheaf}
sheaf\/ $\cal P$ on $\M_\rsi,$ with\/ $\chi_{\M_\rsi}({\cal
P})=\nu_{\M_\rsi},$ which is locally isomorphic to
$\phi_f(\underline{\Q}[\dim U-1])$ for $f,U$ as in  \cite[Thm. 5.4]{JoSo}?
\smallskip

\noindent{\bf(b)} Is there also some Artin stack version of\/ $\cal
P$ in\/ {\bf(a)} for the moduli stack\/ $\fM,$ locally isomorphic to 
$\phi_f(\underline{\Q}[\dim U-1])$ for $f,U$ as in Theorem
{\rm\ref{dt5thm1}} below?

\noindent{\bf(c)} Let $M$ be a complex manifold, $\om$ an almost closed holomorphic $(1,0)$-form on $M$ as defined below, and 
$X = \om^{-1}(0)$ as a complex analytic subspace of $M.$
Can one define a natural perverse sheaf $\cal P$ supported on $X$, with $\chi_X(\cal P)$ $= \nu_X$ , such that 
$\cal P\cong$ $\phi_f(\underline{\Q}[\dim U-1])$ when $\om=\rd f$ for $f:M\ra \C$ holomorphic? Are there generalizations to the algebraic setting?

\label{dt5quest1}
\end{quest}

One can also ask Question \ref{dt5quest1} for Saito's mixed
Hodge modules~\cite{Sait}.\index{mixed Hodge module}
If the answer to Question \ref{dt5quest1}(a) is
yes, it would provide a way of {\it categorifying\/} (conventional)
Donaldson--Thomas invariants $DT^\al(\tau)$. 
That is
${\mathbb H}^*\bigl(\M_\st^\al(\tau);{\cal
P}\vert_{\M_\st^\al(\tau)}\bigr)$ would be a natural cohomology \index{perverse sheaf!hypercohomology}
\nomenclature[HPM]{${\mathbb H}^*\bigl(\cal M; \cal P\bigr)$}{hypercohomology of a perverse sheaf $\cal P$ on a scheme $\cal M$}
group of the stable moduli scheme $\M_\st^\al(\tau)$ whose Euler characteristic is the
Donaldson--Thomas invariant. 
This question is also crucial for the programme of Kontsevich--Soibelman \cite{KoSo1} to extend
Donaldson--Thomas invariants of Calabi--Yau 3-folds to other motivic
invariants.\index{motivic invariant} 
as discussed in \cite[Remark 5.8]{JoSo}.\index{categorification|)} 
We will explain in \S\ref{ourpapers} how this question has been resolved.

\subsubsection{The Behrend function and its characterization}
\label{dt3.2}

Here we will point out some important remarks and properties of the Behrend function. 

\paragraph{Behrend function as a multiplicity function in the weighted Euler characteristic.}
\label{dt3.2.1} \index{Behrend function!multiplicity function|(} \index{Euler characteristic} 
It is worth to report here \cite[ \S 1.2]{JoSo} which provides a good way to think of Behrend functions as {\it multiplicity
functions}. If $X$ is a finite type $\C$-scheme then the Euler
characteristic $\chi(X)$ `counts' points without multiplicity, so
that each point of $X(\C)$ contributes 1 to $\chi(X)$. 
\nomenclature[xred]{$X^{\rm red}$}{underlying reduced $\C$-scheme of the $\C$-scheme $X$}
If $X^{\rm red}$ is the underlying reduced $\C$-scheme then $X^{\rm
red}(\C)=X(\C)$, so $\chi(X^{\rm red})=\chi(X)$, and $\chi(X)$ does
not see non-reduced behaviour in $X$. However, the {\it weighted Euler
characteristic} \index{Euler
characteristic!weighted} $\chi(X,\nu_X)$ `counts' each $x\in X(\C)$ weighted
by its multiplicity $\nu_X(x)$. The Behrend function $\nu_X$ detects
non-reduced behaviour, so in general $\chi(X,\nu_X)\ne\chi(X^{\rm
red},\nu_{X^{\rm red}})$. For example, let $X$ be the $k$-fold point
$\Spec\bigl(\C[z]/(z^k)\bigr)$ for $k\ge 1$. Then $X(\C)$ is a
single point $x$ with $\nu_X(x)=k$, so $\chi(X)=1=\chi(X^{\rm
red},\nu_{X^{\rm red}})$, but~$\chi(X,\nu_X)=k$.

\smallskip

An important moral of \cite{Behr} is that (at least in moduli
problems with symmetric obstruction theories, such as
Donaldson--Thomas theory) it is better to `count' points in a moduli
scheme $\M$ by the weighted Euler characteristic $\chi(\M,\nu_\M)$
than by the unweighted Euler characteristic $\chi(\M)$. One reason
is that $\chi(\M,\nu_\M)$ often gives answers unchanged under
deformations of the underlying geometry, but $\chi(\M)$ does not.
For example, consider the family of $\C$-schemes
$X_t=\Spec\bigl(\C[z]/(z^2-t^2)\bigr)$ for $t\in\C$. Then $X_t$ is
two reduced points $\pm t$ for $t\ne 0$, and a double point when
$t=0$. So as above we find that $\chi(X_t,\nu_{X_t})=2$ for all $t$,
which is deformation-invariant, but $\chi(X_t)$ is 2 for $t\ne 0$
and 1 for $t=0$, which is not deformation-invariant. \index{deformation-invariance}
\index{Behrend function!multiplicity function|)}

\paragraph{Properties of the Behrend function.}
\label{dt3.2.2}

Here are some important properties of Behrend functions. They are
proved by Behrend \cite[\S 1.2 \& Prop. 1.5]{Behr} when $\K=\C$, but
his proof is valid for general~$\K$.
\begin{thm} Let $X,Y$ be Artin $\K$-stacks locally of finite type. Then:
\begin{itemize}
\setlength{\itemsep}{0pt}
\setlength{\parsep}{0pt}
\item[{\rm(i)}] If\/ $X$ is smooth of dimension\/ $n$
then~$\nu_X\equiv(-1)^n$.
\item[{\rm(ii)}] If\/ $\vp:X\!\ra\! Y$ is smooth with
relative dimension $n$ then\/~$\nu_X\!\equiv\!(-1)^n\vp^*(\nu_Y)$.
\item[{\rm(iii)}] $\nu_{X\times Y}\equiv\nu_X\boxdot\nu_Y,$
where $(\nu_X\boxdot\nu_Y)(x,y)=\nu_X(x)\nu_Y(y)$.
\end{itemize}
\label{dt3thm3}
\end{thm}

Let us recall \cite[Thm 4.11]{JoSo}. It is stated using the Milnor fibre, but its proof works
algebraically over $\K$.

\begin{thm} Let\/ $U$ be a smooth $\K$-variety, $f:U\ra \bA^1_{\K}$ a regular function over $U,$ and
$V$ a smooth $\K$-subvariety of $U,$
and\/ $v\in V\cap\Crit(f)$. Define $\ti U$ to be the
blowup of\/ $U$ along\/ $V,$ with blowup map $\pi:\ti U\ra U,$ and
set\/ $\ti f=f\ci\pi:\ti U\ra \bA^1_\K$. Then $\pi^{-1}(v)={\mathbb
P}(T_vU/T_vV)$ is contained in $\Crit(\ti f),$ and
$$
\nu_{\Crit(f)}(v)\quad =\,\displaystyle \int\limits_{w\in {\mathbb
P}(T_vU/T_vV)}\nu_{\Crit(\ti f)}(w)\,\rd\chi\quad+\quad (-1)^{\dim U-\dim V}\bigl(1-\dim U+\dim V\bigr)\nu_{\Crit(f\vert_V)}(v),
$$
where $w\mapsto\nu_{\Crit(f)}(w)$ is a constructible
function\index{constructible function} on ${\mathbb P}(T_vU/T_vV),$ and
the integral is the Euler characteristic of\/
${\mathbb P}(T_vU/T_vV)$ weighted by this.
\label{blowup}
\end{thm}

One can see the next result as a kind of {\it virtual Gauss--Bonnet formula}. It is crucial for
Donaldson--Thomas theory. It is proved by Behrend
\cite[Th. 4.18]{Behr} when $\K=\C$, but his proof is valid for
general~$\K$. It depends crucially on \cite[Prop.\,1.12]{Behr}
which again depend on an application of MacPherson's theorem \cite{MacP} over $\C$
but valid over general $\K$ thanks to Kennedy \cite{Kenn} and the definition of the Euler characteristic 
over algebraically closed field $\K$ of characteristic zero given by Joyce \cite{Joyc.1}. 
See also an independent construction of the Schwartz--MacPherson Chern class given by Aluffi \cite{Alu1}.

\begin{thm} Let $X$ a proper $\K$-scheme with a symmetric
obstruction theory, and\/ $[X]^\vir\in A_0(X)$ the corresponding
virtual class. Then
\begin{equation*}
\ts\displaystyle \int_{[X]^\vir}1=\chi(X,\nu_X)\in\Z,
\end{equation*}
where $\chi(X,\nu_X)=\int_{X(\K)}\nu_X\,\rd\chi$ is the Euler
characteristic of\/ $X$ weighted by the Behrend function $\nu_X$
of\/ $X$. In particular, $ \int_{[X]^\vir}1$ depends only on the\/
$\K$-scheme structure of\/ $X,$ not on the choice of symmetric
obstruction theory.
\label{dt3thm4}
\end{thm}

Theorem \ref{dt3thm4} implies that $DT^\al(\tau)$ in \eq{dt2eq1} is
given by \index{Donaldson--Thomas invariants!original $DT^\al(\tau)$}
\e
DT^\al(\tau)=\chi\bigl(\M_\st^\al(\tau),\nu_{\M_\st^\al(\tau)}\bigr).
\label{dt3eq5}
\e
There is a big difference between the two equations \eq{dt2eq1} and
\eq{dt3eq5} defining Donaldson--Thomas invariants. Equation
\eq{dt2eq1} is non-local, and non-motivic, and makes sense only if
$\M_\st^\al(\tau)$ is a proper $\K$-scheme. But \eq{dt3eq5} is
local, and (in a sense) motivic, and makes sense for arbitrary
finite type $\K$-schemes $\M_\st^\al(\tau)$. In fact, one could take
\eq{dt3eq5} to be the definition of Donaldson--Thomas invariants
even when $\M_\rss^\al(\tau)\ne\M_\st^\al(\tau)$, but 
in \cite [\S 6.5]{JoSo} Joyce and Song argued that this is not a good idea, as then $DT^\al(\tau)$
would not be unchanged under deformations of~$X$.
In \cite[\S 6.5]{JoSo} Joyce and Song say: 
\begin{quotation}
`Equation \eq{dt3eq5} was the inspiration for this book. 
It shows that Donaldson--Thomas invariants $DT^\al(\tau)$ can be written as
{\it motivic\/} invariants,\index{motivic invariant} like those studied
in \cite{Joyc.3,Joyc.4,Joyc.5,Joyc.6,Joyc.7}, and so it raises the
possibility that we can extend the results of
\cite{Joyc.3,Joyc.4,Joyc.5,Joyc.6,Joyc.7} to Donaldson--Thomas invariants
by including Behrend functions as weights.'
\end{quotation}

\paragraph{Almost closed 1-forms.}
\label{dt3.2.3} \index{almost closed $1$-form|(}

In \cite{PaTh} Pandharipande and Thomas give a counterexample to the idea that every scheme admitting a symmetric obstruction
theory can locally be written as the critical locus of a regular
function on a smooth scheme. This limits the usefulness of the above
formula for $\nu_X(x)$ in terms of the Milnor fibre.
Here is the more general approach due to Behrend \cite{Behr}, which
the author tried to use 
to give a strictly algebraic proof on the Behrend function identities,
but later this proof turned out to be not completely correct. 

\begin{dfn} Let $\K$ be an algebraically closed field, and $M$ a
smooth $\K$-scheme. Let $\om$ be an algebraic 1-form on $M$, that is, $\om\in
H^0(T^*M)$. Call $\om$ {\it almost closed\/}
if $\rd\om$ is a section of $I_\om\cdot\La^2T^*M$, where
$I_\om$ is the ideal sheaf of the zero locus $\om^{-1}(0)$ of $\om$.
Equivalently, $\rd\om\vert_{\om^{-1}(0)}$ is zero as a section of
$\La^2T^*M\vert_{\om^{-1}(0)}$. In (\'etale) local coordinates
$(z_1,\ldots,z_n)$ on $M$, if $$\om=f_1\rd z_1+\cdots+f_n\rd z_n,$$
then $\om$ is almost closed provided
\begin{equation*}
\frac{\pd f_j}{\pd z_k}\equiv\frac{\pd f_k}{\pd z_j} \;\>\mod
(f_1,\ldots,f_n).
\end{equation*}
\label{dt3def2}
\end{dfn}
Let $M$ be a smooth Deligne--Mumford stack and $\omega$ an almost \index{Deligne--Mumford stack}
closed 1-form on $M$ with zero locus $X=Z(\omega)$.
It is a general principle, that a section of a vector bundle defines a
perfect obstruction theory for the zero locus of the section. 
This obstruction theory is given by
\e
\begin{gathered}
\xymatrix@C=30pt@R=30pt{ [T_{M_{|_{X}}} \ar[r]^{d\circ\omega^\vee}\ar[d]_{\omega^\vee}
& \Omega_{M_{|_{X}}}]\ar[d]^1\\
 [I/I^2\ar[r]^{d} & \Omega_{M_{|_{X}}}]}
\label{dt3example}
\end{gathered}
\e

This obstruction theory is symmetric, in a canonical way, because
\index{obstruction theory!symmetric}
under the assumption that $\omega$ is almost closed one has that
$d\circ\omega^\vee$ is self-dual, as a homomorphism of vector bundles
over $X$.  

\smallskip

Behrend \cite[Prop. 3.14]{Behr} proves a kind of converse of that, by a proof
valid for general~$\K$, which says that, at least locally, every symmetric obstruction
theory is given in this way by an almost closed $1$-form. 

\begin{prop} Let\/ $\K$ be an algebraically closed field, and\/ $X$
a $\K$-scheme with a symmetric obstruction theory. Then $X$ may be
covered by Zariski open sets $Y\subseteq X$ such that there exists a
smooth\/ $\K$-scheme $M,$ an almost closed\/ $1$-form $\om$ on $M,$
and an isomorphism of\/ $\K$-schemes\/~$Y\cong\om^{-1}(0)$.
\label{dt3prop5}
\end{prop}
Restricting to $\K=\C$, Behrend \cite[Prop. 4.22]{Behr} gives an
expression for the Behrend function of the zero locus of an almost
closed 1-form as a {\it linking number}\index{linking number}. It is possible to use it to give an algebraic 
proof of the first Behrend identity over $\C.$

\begin{prop} Let\/ $M$ be a smooth scheme and\/ $\om$ an almost
closed $1$-form on $M,$ and let\/ $Y=\om^{-1}(0)$ 
be the scheme-theoretic zero locus of\/ $\om$. Fix\/ $p$ a closed point in $Y$, choose
\'etale coordinates $(x_1,\ldots,x_n)$ on $M$ around $p$ with\/ $(x_1,\ldots,x_n,p_1,\ldots,
p_n)$ the associated canonical coordinates for $T^*M.$ Write $\om=\displaystyle\sum_{i=1}^{n}f_i \rd x_i$
in these coordinates. One can identify
$T^*M$ near $p$ with\/~$\C^{2n}$.
Then for all\/ $\eta\in\C$ and\/ $\ep\in\R$ with\/
$0<\md{\eta}\ll\ep\ll 1$ one has
\e
\nu_Y(p)=L_{{\cal S}_\ep}\bigl(\Ga_{\eta^{-1}\om}\cap{\cal S}_\ep,
\De\cap {\cal S}_\ep\bigr),
\label{dt6eq8}
\e
where \nomenclature[sphere]{${\cal S}_\ep$}{sphere of radius $\ep$ in $\C^{2n}$}
\nomenclature[1b]{$\Ga_{\eta^{-1}\om}$}{graph of $\eta^{-1}\om$ for $\om$ almost closed $1$-form and $\eta\in\C$}
\nomenclature[1c]{$\De$}{graph of the section given by the square of the distance function}
\nomenclature[LS]{$L_{{\cal S}_\ep}(\,,\,)$}{linking number of two disjoint, closed, oriented $(n-1)$-submanifolds in ${\cal S}_\ep$}
\begin{itemize}
\setlength{\itemsep}{0pt}
\setlength{\parsep}{0pt}
\item ${\cal S}_\ep\!=\!\bigl\{(x_1,\ldots,p_n)\!\in\!\C^{2n}:
\ms{x_1}\!+\!\cdots\!+\!\ms{p_n}\!=\!\ep^2\bigr\}$ is the sphere of
radius $\ep$ in $\C^{2n},$
\item  $\Ga_{\eta^{-1}\om}$ is the graph of\/
$\eta^{-1}\om$ regarded locally as a complex submanifold of\/
$\C^{2n}$ of real dimension $2n$ oriented so that $M\longra \Omega_M$ is orientation preserving 
and defined by the equations $\{\eta p_i=f_i(x)\},$
\item $\De=\bigl\{(x_1,\ldots,p_n)\!\in\!\C^{2n}:p_j\!=\!\bar x_j,$
$j\!=\!1,\ldots,n\bigr\},$ i.e. the image of the smooth map $M\longra\Omega_M$ 
given by the section $\rd\varrho$ of $\Omega_M,$ with 
$$\varrho=\displaystyle \sum_{i}  x_{i} \bar x_{i}+\displaystyle \sum_{i}  p_{i} \bar p_{i}$$ the square of the distance function
defined on $\Omega_M$ by the choice of coordinates of real dimension $2n,$ 
\item  $L_{{\cal S}_\ep}(\,,\,)$ is the
linking number of two disjoint, closed, oriented\/
$(n\!-\!1)$-submanifolds in~${\cal S}_\ep$.
\end{itemize}
\label{dt6prop1}
\end{prop} 

We remark here that $\Delta$ is not a complex submanifold, but only a real submanifold. Thus, there 
are no good generalizations of $\Delta$ to other fields $\K.$

\index{almost closed $1$-form|)}

\index{microlocal geometry|)}\index{Behrend function|)}

\subsection{Generalizations of Donaldson--Thomas theory}
\label{dt4}

Next it will be briefly reviewed how the theory of generalized Donaldson--Thomas invariants has been developed,
starting from the series of papers  \cite{Joyc.1,Joyc.2,Joyc.3,Joyc.4,Joyc.5,Joyc.6,Joyc.7} about constructible functions,
stack functions, Ringel--Hall algebras, counting invariants for
Calabi--Yau 3-folds, and wall-crossing and then summarizing the main results in \cite{JoSo} 
including the definition of generalized Donaldson--Thomas invariants
$\bar{DT}{}^\al(\tau) \in\Q$, their deformation-invariance, and
wall-crossing formulae under change of stability condition~$\tau$. 
In the sequel, there are two paragraphs on 
statements and a sketch of proofs of
the theorems \cite[Thm 5.5]{JoSo} and \cite[Thm 5.11]{JoSo} on which this paper is concentrated.  
We conclude with a brief and rough remark on Kontsevich and Soibelman's parallel approach to Donaldson--Thomas theory 
\cite{KoSo1}, focusing more on analogies and differences with Joyce and Song's construction \cite{JoSo} rather than going into a detailed exposition. This choice is due to the fact that for the present paper we do not need it.
\index{constructible function}
\index{stack function}
\index{Ringel--Hall algebra}
\index{counting invariants for Calabi--Yau 3-folds}
\index{wall crossing}
\index{deformation-invariance}
\index{configurations}

\subsubsection[Brief sketch of background from
$\text{\cite{Joyc.1,Joyc.2,Joyc.3,Joyc.4,Joyc.5,Joyc.6,Joyc.7}}$]{Brief
sketch of background from \cite{Joyc.1,Joyc.2,Joyc.3,Joyc.4,Joyc.5,Joyc.6,Joyc.7}}
\label{dt4.1}

Here it will be recalled a few important ideas from \cite{Joyc.1,Joyc.2,Joyc.3,Joyc.4,
Joyc.5,Joyc.6,Joyc.7}. They deal with {\it Artin stacks} rather than coarse moduli \index{Artin stack}
schemes,\index{coarse moduli scheme}\index{moduli scheme!coarse} as in
\cite{Thom}. Let $X$ be a Calabi--Yau
3-fold over $\C$, and write $\fM$ for the moduli stack of all
coherent sheaves $E$ on $X$. It is an Artin $\C$-stack.

\smallskip

\index{stack function}
\nomenclature[SFfF]{$\SF(\fF)$}{vector space of `stack functions' on an Artin stack $\fF$, defined using representable 1-morphisms}
The ring of {\it stack functions} $\SF(\fM)$ in \cite{Joyc.2} is 
basically the Grothendieck group $K_0(\mathop{\rm Sta}_\fM)$ of the \index{2-category} \index{Grothendieck group}
2-category $\mathop{\rm Sta}_\fM$ of stacks over $\fM$. That is,
\nomenclature[Sta]{$\mathop{\rm Sta}_\fM$}{2-category of stacks over $\fM$}
$\SF(\fM)$ is generated by isomorphism classes $[(\fR,\rho)]$ of
representable 1-morphisms $\rho:\fR\ra\fM$ for $\fR$ a finite type
Artin $\C$-stack, with the relation \index{1-morphism!representable} \index{Artin stack!of finite type}
\begin{equation*}
[(\fR,\rho)]=[(\fS,\rho\vert_\fS)]+[(\fR\sm\fS,\rho\vert_{\fR\sm\fS})]
\end{equation*}
when $\fS$ is a closed $\C$-substack of $\fR$. In \cite{Joyc.2} Joyce studies different kinds
of stack function spaces with other choices of generators and
relations, and operations on these spaces. These include projections 
\nomenclature[1pivi]{$\Pi^\vi_n$}{projection to stack functions of `virtual rank $n$'}
$\Pi^\vi_n:\SF(\fM)\ra\SF(\fM)$ to stack functions of {\it virtual rank} 
$n$, which act on $[(\fR,\rho)]$ by modifying $\fR$ depending on
its stabilizer groups. \index{stabilizer group} \index{virtual rank}\index{stack function!virtual rank}
\nomenclature[SFai]{$\SFai(\fM)$}{Lie subalgebra of $\SFa(\fM)$ of stack functions `supported on virtual indecomposables'}
\index{virtual indecomposable}\index{stack
function!supported on virtual indecomposables} 

In \cite[\S 5.2]{Joyc.4} he defines a {\it Ringel--Hall} type
algebra\index{Ringel--Hall algebra} $\SFa(\fM)$ of stack
\nomenclature[SFal]{$\SFa(\fM)$}{Ringel--Hall algebra of stack functions with with algebra stabilizers}
functions\index{stack function!with algebra stabilizers} {\it with algebra
stabilizers} on $\fM$, with an associative, non-commutative
multiplication $*$ and in \cite[\S 5.2]{Joyc.4} he defines a Lie
subalgebra $\SFai(\fM)$ of stack functions {\it supported on virtual indecomposables}.
In \cite[\S
6.5]{Joyc.4} he defines an explicit Lie algebra $L(X)$ to be the
\nomenclature[LX]{$L(X)$}{Lie algebra depending on a Calabi--Yau $3$-fold $X$}
\nomenclature[1laal]{$\la^\al$}{basis element of Lie algebra $L(X)$}
$\Q$-vector space with basis of symbols $\la^\al$ for $\al\in
K^\num(\coh(X))$, with Lie bracket
\e
[\la^\al,\la^\be]=\bar\chi(\al,\be)\la^{\al+\be},
\label{dt4eq1}
\e
for $\al,\be\in K^\num(\coh(X))$, where $\bar\chi(\,,\,)$ is the
\nomenclature[1wchbar]{$\bar\chi(\,,\,)$}{Euler form on $K^\num(\coh(X))$}
{\it Euler form} \index{Euler form} on $K^\num(\coh(X))$ defined as follows: 
\e
\bar{\chi}([E],[F])=\displaystyle\sum_{i\geq 0} (-1)^i \dim\Ext^i(E,F)
\label{eu}
\e
for all $E,F\in \coh(X).$
As $X$ is a
Calabi--Yau 3-fold, $\bar\chi$ is antisymmetric, so \eq{dt4eq1}
satisfies the Jacobi identity and makes $L(X)$ into an
infinite-dimensional Lie algebra over~$\Q$.

Then in \cite[\S 6.6]{Joyc.4} Joyce defines a 
\nomenclature[1vPsSfL]{$\Psi$}{Lie algebra morphism $\Psi:\SFai(\fM)\ra L(X)$}
{\it Lie algebra morphism\/} $\Psi:\SFai(\fM)\ra L(X)$, which, roughly speaking, is of the form
\e
\Psi(f)\;\;\;\;=\ts\displaystyle \sum_{\al\in K^\num(\coh(X))}\chi^{\rm stk}
\bigl(f\vert_{\fM^\al}\bigr)\la^{\al},
\label{dt4eq2}
\e
where $f=\sum_{i=1}^mc_i[(\fR_i,\rho_i)]$ is a stack function on
$M$, and $\fM^\al$ is the substack in $\fM$ of sheaves $E$ with
\nomenclature[Mal]{$\fM^\al$}{the substack in $\fM$ of sheaves $E$ with class $\al$}
class $\al$, and $\chi^{\rm stk}$ is a kind of stack-theoretic Euler
\nomenclature[1wchistk]{$\chi^{\rm stk}$}{stack-theoretic Euler characteristic} \index{Euler characteristic!stack-theoretic}
characteristic. But in fact the definition of $\Psi$, and the proof
that $\Psi$ is a Lie algebra morphism, are highly nontrivial, and
use many ideas from \cite{Joyc.1,Joyc.2,Joyc.4}, including those of
`virtual rank' and `virtual indecomposable'.\index{virtual
indecomposable} The problem is that the obvious definition of
$\chi^{\rm stk}$ usually involves dividing by zero, so defining
\eq{dt4eq2} in a way that makes sense is quite subtle. The proof
that $\Psi$ is a Lie algebra morphism uses {\it Serre duality} and the \index{Serre duality}
assumption that $X$ is a Calabi--Yau 3-fold.

\smallskip

\nomenclature[M bal ssalt]{$\fM_\rss^\al(\tau)$}{open, finite type substack  
in $\fM$ of $\tau$-semistable sheaves $E$ in class $\al$, for all $\al\in K^\num(\coh(X))$}
\nomenclature[M bal stalt]{$\fM_\st^\al(\tau)$}{open, finite type substack  
in $\fM$ of $\tau$-stable sheaves $E$ in class $\al$, for all $\al\in K^\num(\coh(X))$}
Now let $\tau$ be a stability condition on $\coh(X)$, such as
Gieseker stability. Then one has open, finite type substacks \index{Gieseker stability}
$\fM_\rss^\al(\tau),\fM_\st^\al(\tau)$ in $\fM$ of
$\tau$-(semi)stable sheaves $E$ in class $\al$, for all $\al\in
K^\num(\coh(X))$. Write $\bar\de_\rss^\al(\tau)$ for the \index{stack function!characteristic function}
\nomenclature[1cc]{$\bar\de_\rss^\al(\tau)$}{element of the Ringel--Hall Lie algebra $\SFa(\fM)$ that ÔcountsÕ $\tau$-
semistable objects in class $\al$}
characteristic function of $\fM_\rss^\al(\tau)$, in the sense of
stack functions \cite{Joyc.2}. Then
$\bar\de_\rss^\al(\tau)\in\SFa(\fM)$. In \cite[\S 8]{Joyc.5}, Joyce
defines elements $\bar\ep^\al(\tau)$ in $\SFa(\fM)$ by
\nomenclature[1ccc]{$\bar\ep^\al(\tau)$}{element of the Ringel--Hall Lie algebra $\SFai(\fM)$ that ÔcountsÕ $\tau$-
semistable objects in class $\al$}
\e
\bar\ep^\al(\tau)\quad= \!\!\!\!\!\!\!
\sum_{\begin{subarray}{l}\small{n\ge 1,\;\al_1,\ldots,\al_n\in
K^\num(\coh(X)):}\\
\small{\al_1+\cdots+\al_n=\al,\; \tau(\al_i)=\tau(\al),\text{ all
$i$}}\end{subarray}} \!\!\!
\frac{(-1)^{n-1}}{n}\,\,\bar\de_\rss^{\al_1}(\tau)*\bar
\de_\rss^{\al_2}(\tau)* \cdots*\bar\de_\rss^{\al_n}(\tau),
\label{dt4eq3}
\e
where $*$ is the Ringel--Hall multiplication in $\SFa(\fM)$. Then
\cite[Thm. 8.7]{Joyc.5} shows that $\bar\ep^\al(\tau)$ lies in the Lie
subalgebra $\SFai(\fM)$, a nontrivial result.
Thus one can apply the Lie algebra morphism $\Psi$ to
$\bar\ep^\al(\tau)$. In \cite[\S 6.6]{Joyc.6} he defines invariants
$J^\al(\tau)\in\Q$ for all $\al\in K^\num(\coh(X))$ by
\nomenclature[Jalt]{$J^\al(\tau)$}{invariant counting 
$\tau$-semistable sheaves in class $\al$ on a CalabiÐYau $3$-fold, introduced in \cite{Joyc.7}}
\e
\Psi\bigl(\bar\ep^\al(\tau)\bigr)=J^\al(\tau)\la^\al.
\label{dt4eq4}
\e

\smallskip

These $J^\al(\tau)$ are rational numbers `counting' \index{invariant $J^\al(\tau)$}
$\tau$-semistable sheaves $E$ in class $\al$. When
$\M_\rss^\al(\tau)=\M_\st^\al(\tau)$ then
$J^\al(\tau)=\chi(\M_\st^\al(\tau))$, that is, $J^\al(\tau)$ is the
na\"\i ve Euler characteristic of the moduli space \index{Euler characteristic!na\"\i ve}
$\M_\st^\al(\tau)$. This is {\it not\/} weighted by the Behrend
function $\nu_{\M_\st^\al(\tau)}$, and so in general does not
coincide with the Donaldson--Thomas invariant $DT^\al(\tau)$
in~\eq{dt4eq1}. \index{Donaldson--Thomas invariants!original $DT^\al(\tau)$}
\nomenclature[1wchinaive]{$\chi^{\textrm{na}}(C)$}{na\"\i ve Euler characteristic of a constructible set $C$ in a stack as in \cite{Joyc.1}}
As the $J^\al(\tau)$ do not include Behrend functions, they do not
count semistable sheaves with multiplicity, and so they will not in
general be unchanged under deformations of the underlying
Calabi--Yau 3-fold, as Donaldson--Thomas invariants are. However,
the $J^\al(\tau)$ do have very good properties under change of
stability condition. In \cite{Joyc.6} Joyce shows that if $\tau,\ti\tau$
are two stability conditions on $\coh(X)$, then it is possible to write
$\bar\ep^\al(\ti\tau)$ in terms of a (complicated) explicit formula
involving the $\bar\ep^\be(\tau)$ for $\be\in K^\num(\coh(X))$ and
the Lie bracket in~$\SFai(\fM)$.
Applying the Lie algebra morphism $\Psi$ shows that
$J^\al(\ti\tau)\la^\al$ may be written in terms of the
$J^\be(\tau)\la^\be$ and the Lie bracket in $L(X)$, and hence
\cite[Thm. 6.28]{Joyc.6} yields an explicit transformation law for the
$J^\al(\tau)$ under change of stability condition. In \cite{Joyc.7}
he shows how to encode invariants $J^\al(\tau)$ satisfying a
transformation law in generating functions on a complex manifold of
stability conditions, which are both holomorphic and continuous,
despite the discontinuous wall-crossing behaviour of the
$J^\al(\tau)$. 

\subsubsection[Summary of the main results from $\text{\cite{JoSo}}$]{Summary
of the main results from \cite{JoSo}} \index{Donaldson--Thomas invariants!generalized $\bar{DT}{}^\al(\tau)$|(}
\label{dt4.2}

The basic idea behind the project developed in  \cite{JoSo} is that the
Behrend function $\nu_\fM$ of the moduli stack $\fM$ of coherent \index{Behrend function}
sheaves in $X$ should be inserted  as a weight in the programme of
\cite{Joyc.1,Joyc.2,Joyc.3,Joyc.4,Joyc.5,Joyc.6,Joyc.7} summarized in
\S\ref{dt4.1}. Thus one will obtain weighted versions $\ti\Psi$ of
the Lie algebra morphism $\Psi$ of \eq{dt4eq2}, and
$\bar{DT}{}^\al(\tau)$ of the counting invariant $J^\al(\tau)\in\Q$
in \eq{dt4eq4}. Here is how this is worked out in  \cite{JoSo}.

Joyce and Song define a modification $\ti L(X)$ of the Lie algebra $L(X)$ above,
\nomenclature[LXt]{$\ti L(X)$}{Lie algebra depending on a Calabi--Yau $3$-fold $X$, variant of $L(X)$}
the $\Q$-vector space with basis of symbols $\ti \la^\al$ for
\nomenclature[1lati]{$\ti \la^\al$}{basis element of Lie algebra $\ti L(X)$}
$\al\in K^\num(\coh(X))$, with Lie bracket
\begin{equation*}
[\ti\la^\al,\ti\la^\be]=(-1)^{\bar\chi(\al,\be)}
\bar\chi(\al,\be)\ti \la^{\al+\be},
\end{equation*}
which is \eq{dt4eq2} with a sign change. Then they define a {\it Lie
algebra morphism\/} $\ti\Psi:\SFai(\fM)\ra\ti L(X)$. Roughly
speaking this is of the form
\nomenclature[1vPsiti]{$\ti\Psi$}{Lie algebra morphism $\ti\Psi:\SFai(\fM)\ra\ti L(X)$}
\e
\ti\Psi(f)\;\;\;\;=\ts\displaystyle \sum_{\al\in K^\num(\coh(X))}\chi^{\rm stk}
\bigl(f\vert_{\fM^\al},\nu_{\fM}\bigr)\ti \la^{\al},
\label{dt4eq5}
\e
that is, in \eq{dt4eq2} we replace the stack-theoretic Euler
characteristic $\chi^{\rm stk}$ with a stack-theoretic Euler  
\index{Euler characteristic!stack-theoretic} \index{Euler characteristic!weighted}
characteristic weighted by the Behrend function $\nu_{\fM}$.
The proof that $\ti\Psi$ is a Lie algebra morphism combines the
proof in \cite{Joyc.4} that $\Psi$ is a Lie algebra morphism with the
two {\it Behrend function identities} \index{Behrend function!Behrend identities} \eq{dt6eq1}--\eq{dt6eq2}
proved in \cite[thm. 5.11]{JoSo} and reported below.
Proving \eq{dt6eq1}--\eq{dt6eq2} requires a deep understanding of
the local structure of the moduli stack $\fM$, which is of interest \index{moduli stack!local structure}
in itself. First they show using a composition of
{\it Seidel--Thomas twists} by $\cO_X(-n)$ for $n\gg 0$ that $\fM$ is \index{Seidel--Thomas twist}
locally 1-isomorphic to the moduli stack $\fVect$ of vector bundles \index{vector bundle!moduli stack}
on $X$. Then they prove that near $[E]\in\fVect(\C)$,
an atlas for $\fVect$ can be written locally in the complex analytic \index{analytic topology}
topology in the form $\Crit(f)$ for $f:U\ra\C$ a holomorphic
function on an open set $U$ in $\Ext^1(E,E)$. These $U,f$ are {\it
not algebraic}, they are constructed using gauge theory on the \index{gauge theory}
complex vector bundle $E$ over $X$ and transcendental methods. 
Finally, they deduce \eq{dt6eq1}--\eq{dt6eq2} using
the Milnor fibre expression \eq{dt3eq3} for Behrend functions \index{Milnor fibre}
applied to these~$U,f$.

\smallskip

Before going on with the review of Joyce and Song's program, it is worth to stop 
for a while on some details about \cite[Thm 5.5]{JoSo} and \cite[Thm 5.11]{JoSo}, the statements of the theorems and how they prove it. 
This will be useful later on in \S\ref{dt5}.

\paragraph[Gauge theory and transcendental complex analysis from
$\text{\cite{JoSo}}$]{Gauge theory and transcendental complex analytic geometry from \cite{JoSo}.}
\label{dt5.1} \index{gauge theory}
In \cite[Thm. 5.5]{JoSo} Joyce and Song give a local characterization of an atlas\index{Artin stack!atlas}
for the moduli stack $\fM$ as the critical points of a holomorphic
function on a complex manifold. The statement and a sketch of its proof are reported below. 
Some background references are Kobayashi \cite[\S
VII.3]{Koba}, L\"ubke and Teleman \cite[\S 4.1 \& \S 4.3]{LuTe},
Friedman and Morgan \cite[\S 4.1--\S 4.2]{FrMo} and Miyajima \cite{Miya}.

\index{moduli stack}
\begin{thm} Let\/ $X$ be a Calabi--Yau $3$-fold over\/ $\C,$ and\/
$\fM$ the moduli stack of coherent sheaves on\/ $X$. Suppose\/ $E$
is a coherent sheaf on\/ $X,$ so that\/ $[E]\in\fM(\C)$. Let\/ $G$
be a maximal reductive subgroup\index{reductive group!maximal} in
$\Aut(E),$ and\/ $G^{\sst\C}$ its complexification. Then\/
$G^{\sst\C}$ is an algebraic $\C$-subgroup of\/ $\Aut(E),$ a maximal
reductive subgroup,\index{reductive group!maximal} and\/
$G^{\sst\C}=\Aut(E)$ if and only if\/ $\Aut(E)$ is reductive.
There exists a quasiprojective $\C$-scheme $S,$ an action of\/
$G^{\sst\C}$ on $S,$ a point\/ $s\in S(\C)$ fixed by $G^{\sst\C},$
and a $1$-morphism of Artin $\C$-stacks $\Phi:[S/G^{\sst\C}]\ra\fM,$ \index{1-morphism}
which is smooth of relative dimension $\dim\Aut(E)-\dim G^{\sst\C},$ \index{quotient stack}
where $[S/G^{\sst\C}]$ is the quotient stack, such that\/
$\Phi(s\,G^{\sst\C})=[E],$ the induced morphism on stabilizer groups \index{stabilizer group}
$\Phi_*:\Iso_{[S/G^{\sst\C}]}(s\,G^{\sst\C})\ra\Iso_{\fM}([E])$ is
the natural morphism $G^{\sst\C}\hookra\Aut(E)\cong\Iso_{\fM}([E]),$
and\/ $\rd\Phi\vert_{s\,G^{\sst\C}}:T_sS\cong T_{s\,G^{\sst\C}}
[S/G^{\sst\C}]\ra T_{[E]}\fM\cong \Ext^1(E,E)$ is an isomorphism. \index{versal family}
Furthermore, $S$ parametrizes a formally versal family $(S,{\cal
D})$ of coherent sheaves on $X,$ equivariant under the action of\/
$G^{\sst\C}$ on $S,$ with fibre\/ ${\cal D}_s\cong E$ at\/ $s$. If\/
$\Aut(E)$ is reductive then $\Phi$ is \'etale.

Write $S_\an$ for the complex analytic space \index{complex analytic space} underlying the
$\C$-scheme $S$. Then there exists an open neighbourhood\/ $U$ of\/
$0$ in\/ $\Ext^1(E,E)$ in the analytic topology, a holomorphic \index{analytic topology}
function $f:U\ra\C$ with\/ $f(0)=\rd f\vert_0=0,$ an open
neighbourhood\/ $V$ of\/ $s$ in $S_\an,$ and an isomorphism of
complex analytic spaces $\Xi:\Crit(f)\ra V,$ such that\/ $\Xi(0)=s$
and\/ $\rd\Xi\vert_0:T_0\Crit(f)\ra T_sV$ is the inverse of\/
$\rd\Phi\vert_{s\,G^{\sst\C}}:T_sS\ra\Ext^1(E,E)$. Moreover we can
choose $U,f,V$ to be $G^{\sst\C}$-invariant, and\/ $\Xi$ to be
$G^{\sst\C}$-equivariant.
\label{dt5thm1}
\end{thm}

In  \cite{JoSo}, Theorem \ref{dt5thm1} gives Joyce and Song the possibility to use the Milnor fibre\index{Milnor fibre}  
formula \eq{dt3eq3} for the Behrend function of $\Crit(f)$ to study the Behrend
function $\nu_\fM$, crucially used in proving Behrend identities. \index{Behrend function!Behrend identities}
The proof of Theorem \ref{dt5thm1} comes in two parts. First it is shown
in \cite[\S 8]{JoSo} that $\fM$ near $[E]$ is locally isomorphic, as
an Artin $\C$-stack, to the moduli stack $\fVect$ of  {\it algebraic vector
bundles\/}\index{vector bundle!algebraic}  on $X$ near $[E']$ for some vector bundle $E'\ra X$.
The proof uses algebraic geometry, and is valid for $X$ a
Calabi--Yau $m$-fold for any $m>0$ over any algebraically closed
field $\K$. The local morphism $\fM\ra\fVect$ is the composition of
shifts and $m$  {\it Seidel--Thomas
twists\/}\index{Seidel--Thomas twist} by $\cO_X(-n)$ for~$n\gg 0$.
Thus, it is enough to prove Theorem \ref{dt5thm1} with $\fVect$ in
place of $\fM$. This is done in \cite[\S 9]{JoSo} using gauge theory
on vector bundles over $X$. An interesting motivation for this 
approach could be found in  \cite[\S 3]{DoTh} and \cite[\S 2]{Thom}. 
Let $E\ra X$ be a fixed complex (not holomorphic)
vector bundle over $X$. Write $\sA$ for the infinite-dimensional \index{semiconnection}
\nomenclature[As]{$\sA$}{affine space of smooth semiconnections on a vector bundle}
affine space of smooth {\it semiconnections} ($\db$-operators) on $E$, and
$\sG$ for the infinite-dimensional Lie group of  \index{gauge theory!gauge group}
\nomenclature[G]{$\sG$}{gauge group of smooth gauge transformations of a vector bundle}
{\it smooth gauge transformations} of $E$. Then $\sG$ acts on $\sA$, and $\sB=\sA/\sG$
\nomenclature[B]{$\sB$}{space of gauge-equivalence classes of semiconnections on a vector bundle}
is the space of gauge-equivalence classes of semiconnections on~$E$.
Fix $\db_E$ in $\sA$ coming from a holomorphic vector bundle
structure on $E$. Then points in $\sA$ are of the form $\db_E+A$ for
$A\in C^\iy\bigl(\End(E)\ot_\C \La^{0,1}T^*X\bigr)$, and $\db_E+A$
makes $E$ into a holomorphic vector bundle if $F_A^{0,2}=\db_EA+A\w
A$ is zero in $\smash{C^\iy\bigl(\End(E)\ot_\C
\La^{0,2}T^*X\bigr)}$. Thus, the moduli space (stack) of holomorphic
vector bundle structures on $E$ is isomorphic to $\{\db_E+A\in\sA:
F_A^{0,2}=0\}/\sG$. In  \cite{Thom}, it is observed that when $X$ is a Calabi--Yau
3-fold, there is a natural holomorphic function $CS:\sA\ra\C$ called
the {\it holomorphic Chern--Simons functional}, invariant under
\nomenclature[CS]{$CS$}{holomorphic Chern--Simons functional}
$\sG$ up to addition of constants, such that
$\{\db_E+A\in\sA:F_A^{0,2}=0\}$ is the critical locus of $CS$. Thus,
$\fVect$ is (informally) locally the critical points of a
holomorphic function $CS$ on an infinite-dimensional complex stack
$\sB=\sA/\sG$. To prove Theorem \ref{dt5thm1} Joyce and Song show that one can
find a finite-dimensional complex submanifold $U$ in $\sA$ and a
finite-dimensional complex Lie subgroup $G^{\sst\C}$ in $\sG$
preserving $U$ such that the theorem holds with~$f=CS\vert_U$.
These $U,f$ are {\it not algebraic}, they are constructed using gauge theory on the
complex vector bundle $E$ over $X$ and transcendental methods.  

\paragraph[The Behrend function identities from
$\text{\cite{JoSo}}$]{The Behrend function identities from \cite{JoSo}.}
\label{dtBehid} \index{Behrend function!Behrend identities|(}
In \cite[Thm. 5.11]{JoSo} Behrend
function identities are proven: they are the crucial step to 
define the Lie algebra morphism $\ti\Psi$ below and then the generalized Donaldson--Thomas invariants:

\begin{thm} Let\/ $X$ be a Calabi--Yau $3$-fold over\/ $\C,$ and\/
$\fM$ the moduli stack of coherent sheaves on\/ $X$. The
Behrend function $\nu_{\fM}:
\fM(\C)\ra\Z$ is a natural locally constructible function on $\fM$.
For all\/ $E_1,E_2\in\coh(X),$ it satisfies:
\bigskip
\begin{equation}
\nu_{\fM}(E_1\op E_2)=(-1)^{\bar\chi([E_1],[E_2])}
\nu_{\fM}(E_1)\nu_{\fM}(E_2), 
\label{dt6eq1}
\end{equation}
\smallskip
\begin{equation}
\displaystyle \int\limits_{\small{\begin{subarray}{l}  [\la]\in\mathbb{P}(\Ext^1(E_2,E_1)):\\
 \la\; \Leftrightarrow\; 0\ra E_1\ra F\ra E_2\ra
0\end{subarray}}}\!\!\!\!\! \!\!\!\! \!\!\!\! \!\!\!\! \nu_{\fM}(F)\,\rd\chi \quad - \!\!\!\! \!\!\!\! \!\!\!\! 
\displaystyle \int\limits_{\small{\begin{subarray}{l}[\mu]\in\mathbb{P}(\Ext^1(E_1,E_2)):\\
\mu\; \Leftrightarrow\; 0\ra E_2\ra D\ra E_1\ra
0\end{subarray}}}\!\!\!\!\! \!\!\!\! \!\!\!\! \!\!\!\! \nu_{\fM}(D)\,\rd\chi \;\;
= \;\; (e_{21}-e_{12})\;\;
\nu_{\fM}(E_1\op E_2),
\label{dt6eq2}
\end{equation}
\bigskip

where $e_{21}=\dim\Ext^1(E_2,E_1)$ and $e_{12}=\dim\Ext^1(E_1,E_2)$ for $E_1,E_2\in\coh(X).$
Here\/ $\bar\chi([E_1],[E_2])$ in \eq{dt6eq1} is the Euler form as in \eq{eu}, \index{Euler form}
and in \eq{dt6eq2} the correspondence between\/
$[\la]\in\mathbb{P}(\Ext^1(E_2,E_1))$ and\/ $F\in\coh(X)$ is that\/
$[\la]\in\mathbb{P}(\Ext^1(E_2,E_1))$ lifts to some\/
$0\ne\la\in\Ext^1(E_2,E_1),$ which corresponds to a short exact
sequence\/ $0\ra E_1\ra F\ra E_2\ra 0$ in\/ $\coh(X)$ in the usual
way. The function $[\la]\mapsto\nu_{\fM}(F)$ is a constructible
function\/ $\mathbb{P}(\Ext^1(E_2,E_1))\ra\Z,$ and the integrals in
\eq{dt6eq2} are integrals of constructible functions using the Euler
characteristic as measure. \index{constructible function} \index{Euler characteristic}
\label{dt6thm1}
\end{thm}

Joyce and Song prove Theorem \ref{dt6thm1} using Theorem \ref{dt5thm1} and the
Milnor fibre description of Behrend functions from ~\S\ref{dt4}. 
They apply Theorem \ref{dt5thm1} to $E=E_1\op E_2$, 
and take the maximal reductive subgroup $G$ of $\Aut(E)$ to contain \index{reductive group!maximal}
the subgroup $\bigl\{\id_{E_1}+\la\id_{E_2}: \la\in\U(1)\bigr\}$, so
that $G^{\sst\C}$ contains $\bigl\{\id_{E_1}+\la\id_{E_2}:\la\in
\bG_m\bigr\}$. Equations \eq{dt6eq1} and \eq{dt6eq2} are proved by a
kind of localization using this $\bG_m$-action on~$\Ext^1(E_1\op
E_2,E_1\op E_2)$.
More precisely, Theorem \ref{dt5thm1} gives an atlas for $\fM$ near $E$ as $\Crit(f)$ near $0$, \index{Milnor fibre}
where $f$ is a holomorphic function defined near $0$ on ~$\Ext^1(E_1\op
E_2,E_1\op E_2)$ and $f$ is invariant under the action of $T=\bigl\{\id_{E_1}+\la\id_{E_2}: \la\in\U(1)\bigr\}$ on
$\Ext^1(E_1\op E_2,E_1\op E_2)$ by conjugation. The fixed points of $T$ on $\Ext^1(E_1\op
E_2,E_1\op E_2)$ are $\Ext^1(E_1,E_1)\op\Ext^1(E_2,E_2)$ and heuristically one can says that
the restriction of $f$ to these fixed points is $f_1 + f_2,$
where $f_j$ is defined near $0$ in $\Ext^1(E_j,E_j)$ and $\Crit(f_j)$ is an atlas for $\fM$ near $E_j$. \index{moduli stack!atlas}
The Milnor fibre $MF_f(0)$ is invariant under $T$, so by localization one has 
$$\chi (MF_f(0))=\chi(MF_f(0)^T)=\chi(MF_{f_1 + f_2}(0)).$$
A product property of Behrend functions, which may be seen as a 
kind of {\it Thom-Sebastiani theorem}, gives \index{Thom-Sebastiani theorem}
$$1-\chi(MF_{f_1 + f_2}(0))=(1-\chi(MF_{f_1}(0)))(1-\chi(MF_{f_2}(0))).$$
Then the identity \eq{dt6eq1} follows from Theorem \ref{dt3thm1}:
$$\nu_\fM (E)=(-1)^{	\dim\Ext^1(E,E)-\dim\Hom(E,E)}(1-\chi (MF_f(0))),$$
and the analogues for $E_1$ and $E_2$. 
Equation \eq{dt6eq2} uses a more involved argument to do with the Milnor fibres of $f$ 
at non-fixed points of the $U(1)$-action.
The proof of Theorem 
\ref{dt6thm1} uses gauge theory, and transcendental complex analytic \index{gauge theory}
geometry methods, and is valid only over~$\K=\C$.
However, as pointed out in \cite[Question 5.12]{JoSo}, Theorem \ref{dt6thm1} makes sense as a statement in
algebraic geometry, for Calabi--Yau 3-folds over \index{field $\K$} $\K$. \index{constructible function} \index{conical Lagrangian cycle}
\index{Behrend function!Behrend identities|)}

\medskip

In \cite[\S 5]{JoSo}, Joyce and Song 
then define {\it generalized Donaldson--Thomas invariants\/}
$\bar{DT}{}^\al(\tau)\in\Q$ by\index{Donaldson--Thomas
invariants!generalized $\bar{DT}{}^\al(\tau)$}
\e
\ti\Psi\bigl(\bar\ep^\al(\tau)\bigr)=-\bar{DT}{}^\al(\tau)\ti
\la^\al,
\label{dt4eq8}
\e
as in \eq{dt4eq4}. When $\M_\rss^\al(\tau)=\M_\st^\al(\tau)$ then
$\bar\ep^\al(\tau)=\bar\de_\rss^\al(\tau)$, and \eq{dt4eq5} gives
\e
\ti\Psi\bigl(\bar\ep^\al(\tau)\bigr)=\chi^{\rm stk}\bigl(
\fM_\st^\al (\tau),\nu_{\fM_\st^\al(\tau)}\bigr)\ti\la^\al.
\label{dt4eq9}
\e
The projection $\pi:\fM_\st^\al(\tau)\ra\M_\st^\al(\tau)$ from the
moduli stack to the coarse moduli scheme\index{coarse moduli
scheme}\index{moduli scheme!coarse} is smooth of dimension $-1$, so
$\nu_{\fM_\st^\al(\tau)}=-\pi^*(\nu_{\M_\st^\al(\tau)})$ by (ii) in
\S\ref{dt3.2.2}, and comparing \eq{dt3eq5}, \eq{dt4eq8}, \eq{dt4eq9}
shows that $\bar{DT}{}^\al(\tau)=DT^\al(\tau)$. But the new
invariants $\bar{DT}{}^\al(\tau)$ are also defined for $\al$ with
$\M_\rss^\al(\tau)\ne\M_\st^\al(\tau)$, when conventional
Donaldson--Thomas invariants $DT^\al(\tau)$ are not defined. 
\index{Donaldson--Thomas invariants!original $DT^\al(\tau)$}

\smallskip

Thanks to Theorem \ref{dt5thm1} and Theorem \ref{dt6thm1},  
$\ti\Psi$ is a Lie algebra morphism \cite[\S 5.3]{JoSo}, thus the change of stability
condition formula for the $\bar\ep^\al(\tau)$ in \cite{Joyc.6}
implies a formula for the elements $-\bar{DT}{}^\al(\tau)\ti
\la^\al$ in $\ti L(X)$, and thus a transformation law for the
invariants $\bar{DT}{}^\al(\tau)$, \index{wall crossing}
using combinatorial coefficients. \index{Euler form}
\nomenclature[V]{$V(I,\Ga,\ka;\tau,\ti\tau)$}{combinatorial coefficient used in wall-crossing formulae}

\smallskip\nomenclature[PI]{$PI^{\al,n}(\tau')$}{invariants counting stable pairs $s:\cO_X(-n)\ra E$}
\nomenclature[PT]{$PT_{n,\beta}$}{PandharipandeÐThomas invariants}\index{Pandharipande--Thomas invariants}\index{stable pair}
\nomenclature[M al stp]{$\M_\stp^{\al,n}(\tau')$}{the moduli space of stable pairs $s:\cO_X(-n)\ra X$ with $[E]=\al$}

To study the new invariants $\bar{DT}{}^\al(\tau)$, it is
helpful to introduce another family of invariants
$PI^{\al,n}(\tau')$,\index{stable pair invariants
$PI^{\al,n}(\tau')$} similar to Pandharipande--Thomas invariants
\cite{PaTh}. Let $n\gg 0$ be
fixed. A {\it stable pair} is a nonzero
morphism $s:\cO_{X}(-n)\ra E$ in $\coh(X)$ such that $E$ is
$\tau$-semistable, and if $\Im s\subset E'\subset E$ with $E'\ne E$
then $\tau([E'])<\tau([E])$. For $\al\in K^\num(\coh(X))$ and $n\gg
0$, the moduli space $\M_\stp^{\al,n}(\tau')$ of stable pairs
$s:\cO_X(-n)\ra X$ with $[E]=\al$ is a fine moduli
scheme,\index{fine moduli scheme}\index{moduli scheme!fine} which is
proper and has a symmetric obstruction theory.\index{symmetric
obstruction theory}\index{obstruction theory!symmetric} Joyce and Song  define
\e
\ts PI^{\al,n}(\tau')\;\;\;\;=\displaystyle \int_{\small{[\M_\stp^{\al,n}(\tau')]^\vir}}1\;\;\;\;=\;\;\;\;
\chi\bigl( \M_\stp^{\al,n}(\tau'),\nu_{\M_\stp^{\al,n}(\tau')}
\bigr)\in\Z,
\label{dt4eq11}
\e
where the second equality follows from Theorem \ref{dt3thm4}. By a similar
proof to that for Donaldson--Thomas invariants in \cite{Thom}, Joyce and Song
find that $PI^{\al,n}(\tau')$ is unchanged under deformations of the
underlying Calabi--Yau 3-fold~$X$. \index{deformation-invariance}
By a wall-crossing proof similar to that
for $\bar{DT}{}^\al(\tau)$, they show that $PI^{\al,n}(\tau')$ can be written in
terms of the $\bar{DT}{}^\be(\tau).$  As $PI^{\al,n}(\tau')$ is
deformation-invariant, one deduces from this relation by induction on
$\rank\al$ with $\dim\al$ fixed that $\bar{DT}{}^\al(\tau)$ is also
deformation-invariant. \index{deformation-invariance}

\smallskip

The pair invariants $PI^{\al,n}(\tau')$ are a useful tool for
computing the $\bar{DT}{}^\al(\tau)$ in examples in \cite[\S 6]{JoSo}. The
method is to describe the moduli spaces $\M_\stp^{\al,n}(\tau')$
explicitly, and then use \eq{dt4eq11} to compute
$PI^{\al,n}(\tau')$, and their relation with $\bar{DT}{}^\al(\tau)$ to deduce the values of
$\bar{DT}{}^\al(\tau)$. Their point of view is that the
$\bar{DT}{}^\al(\tau)$ are of primary interest, and the
$PI^{\al,n}(\tau')$ are secondary invariants, of less interest in
themselves.

\paragraph[Motivic Donaldson--Thomas invariants: Kontsevich and Soibelman's approach from $\text{\cite{KoSo1}}$]{Motivic Donaldson--Thomas invariants: Kontsevich and Soibelman's approach from \cite{KoSo1}.} \index{Donaldson--Thomas invariants!motivic|(}
\label{KoSo}

Kontsevich and Soibelman in \cite{KoSo1} also studied generalizations of Donaldson--Thomas invariants.
They work in a more general context but their results are in great part based on conjectures. They consider derived categories of coherent sheaves, Bridgeland stability conditions \cite{Brid1}, and general motivic invariants, whereas Joyce and Song work with abelian categories of coherent sheaves, Gieseker stability, and the Euler characteristic.
Kontsevich and Soibelman's motivic functions in the equivariant setting \cite[\S 4.2]{KoSo1}, motivic Hall algebra \cite[\S 6.1]{KoSo1},
motivic quantum torus \cite[\S 6.2]{KoSo1} and their algebra morphism to define Donaldson--Thomas invariants \cite[Thm.\,8]{KoSo1} all
have an analogue in Joyce and Song's program. 

\medskip

It is worth to note here some points (see \cite[\S 1.6]{JoSo} for the entire discussion).
\begin{itemize}
\setlength{\itemsep}{0pt}
\setlength{\parsep}{0pt}
\item[{\bf(a)}] Joyce was probably the first to approach Donaldson--Thomas type invariants in an abstract categorical setting. He developed the technique of motivic stack functions and understood the relevance of motives to the counting problem \cite{Joyc.1,Joyc.2,Joyc.3,Joyc.4,Joyc.5,Joyc.6}. The main limitation of his approach was due to the fact that he worked with abelian rather than triangulated categories. For many applications, especially to physics, one needs triangulated categories. The more recent theory of Joyce and Song \cite{JoSo} fixes some of these gaps and fits well with the general philosophy of \cite{KoSo1} (and actually Joyce and Song use some ideas from Kontsevich and Soibelman). 
They deal with concrete examples of categories (e.g. the category of coherent sheaves) and construct numerical invariants via Behrend approach. It is difficult to prove that they are in fact invariants of triangulated categories which is manifest in \cite{KoSo1}.
\item[{\bf(b)}] Kontsevich and Soibelman write their wall-crossing formulae in terms of products in a pro-nilpotent Lie group
while Joyce and Song's formulae are written in terms of combinatorial coefficients. 
\item[{\bf(c)}] Equations \eq{dt6eq1}--\eq{dt6eq2} are related to a conjecture of
Kontsevich and Soibelman \cite[Conj.\,4]{KoSo1} and its application
in \cite[\S 6.3]{KoSo1}, and could probably be deduced from it. Joyce and Song got the idea of proving \eq{dt6eq1}--\eq{dt6eq2} by
localization using the $\bG_m$-action on $\Ext^1(E_1\op E_2, E_1\op
E_2)$ from \cite{KoSo1}. However, Kontsevich and Soibelman approach
\cite[Conj.\,4]{KoSo1} via formal power series and non-Archimedean 
geometry. Their analogue concerns the `motivic Milnor fibre' \index{motivic Milnor fibre} of the formal power series $f$. 
Instead, in Theorem \ref{dt5thm1} Joyce and Song in effect first prove
that they can choose the formal power series to be convergent, and
then use ordinary differential geometry and Milnor fibres.
\item[{\bf(d)}] While Joyce's series of papers \cite{Joyc.1,Joyc.2,Joyc.3,Joyc.4,Joyc.5,Joyc.6} develops the difficult idea of `virtual rank' 
and `virtual indecomposables', Kontsevich and Soibelman have no analogue of these. They come up against the problem (specialization from virtual Poincar\'e polynomial to Euler characteristic)
this technology was designed to solve in the `absence of poles conjecture' \cite[\S7]{KoSo1}.  
\end{itemize}\index{non-Archimedean geometry}\index{formal power series}

Section \ref{dt7} proposes new ideas for further research also in the direction of Kontsevich and Soibelman's paper \cite{KoSo1}.  

\index{Donaldson--Thomas invariants!motivic|)}
\index{Donaldson--Thomas invariants!generalized $\bar{DT}{}^\al(\tau)$|)}

\section{D-critical loci}
\label{dcr}

We summarizes the theory of d-critical
schemes and stacks introduced by Joyce \cite{Joyc.2}.
There are two versions of the theory, complex analytic and algebraic
d-critical loci, sometimes we give results for both
the versions simultaneously, otherwise just
briefly indicate the differences between the two, referring to \cite{Joyc.2}
for details.

\subsection{D-critical schemes}
\label{dcr.1}

Let $X$ be a complex analytic space or a $\K$-scheme. Then \cite[Th.~2.1 \& Prop.~2.3]{Joyc.2}
associates a natural sheaf $\cS_X$ to $X$, such that, very
briefly, sections of $\cS_X$ parametrize different ways of writing
$X$ as $\Crit(f)$ for $U$ a complex manifold or smooth $\K$-scheme
and $f:U\ra\C$ holomorphic or $f:U\ra\bA^1$ regular. 
We refer to \cite[Th.~2.1 \& Prop.~2.3]{Joyc.2} for details. 
Just to give a bit more clear idea, we point out the following:

\begin{rem} Suppose we have $U$ a complex manifold, $f:U\ra\C$ an
holomorphic, and $X=\Crit(f)$, as a closed complex analytic subspace
of $U$. Write $i:X\hookra U$ for the inclusion, and
$I_{X,U}\subseteq i^{-1}(\O_U)$ for the sheaf of ideals vanishing on
$X\subseteq U$. Then
we obtain a natural section $s\in H^0(\cS_X)$. Essentially
$s=f+I_{\d f}^2$, where $I_{\d f}\subseteq\O_U$ is the ideal
generated by $\d f$. Note that $f\vert_X=f+I_{\d f}$, so $s$
determines $f\vert_X$. Basically, $s$ remembers all of the
information about $f$ which makes sense intrinsically on $X$, rather
than on the ambient space~$U$.
\label{dc2ex1}
\end{rem}

Following \cite[Def.~2.5]{Joyc.2} we define algebraic d-critical
loci:

\begin{dfn} An ({\it algebraic\/}) {\it d-critical locus\/} over a
field $\K$ is a pair $(X,s)$, where $X$ is a $\K$-scheme and $s\in
H^0(\cSz_X)$, such that for each $x\in X$, there exists a Zariski
open neighbourhood $R$ of $x$ in $X$, a smooth $\K$-scheme $U$, a
regular function $f:U\ra\bA^1=\K$, and a closed embedding
$i:R\hookra U$, such that $i(R)=\Crit(f)$ as $\K$-subschemes of $U$,
and $\io_{R,U}(s\vert_R)=i^{-1}(f)+I_{R,U}^2$. We call the quadruple
$(R,U,f,i)$ a {\it critical chart\/} on~$(X,s)$.
If $U'\subseteq U$ is a Zariski open, and
$R'=i^{-1}(U')\subseteq R$, $i'=i\vert_{R'}:R'\hookra U'$, and
$f'=f\vert_{U'}$, then $(R',U',f',i')$ is a critical chart on
$(X,s)$, and we call it a {\it subchart\/} of $(R,U,f,i)$, and we write~$(R',U',f',i')\subseteq (R,U,f,i)$.

\smallskip

Let $(R,U,f,i),(S,V,g,j)$ be critical charts on $(X,s)$, with
$R\subseteq S\subseteq X$. An {\it embedding\/} of $(R,U,f,i)$ in
$(S,V,g,j)$ is a locally closed embedding $\Phi:U\hookra V$ such
that $\Phi\ci i=j\vert_R$ and $f=g\ci\Phi$. As a shorthand we write
$\Phi: (R,U,f,i)\hookra(S,V,g,j)$. If $\Phi:(R,U,f,i)\hookra
(S,V,g,j)$ and $\Psi:(S,V,g,j)\hookra(T,W,h,k)$ are embeddings, then
$\Psi\ci\Phi:(R,U,i,e)\hookra(T,W,h,k)$ is also an embedding.

\smallskip

A {\it morphism\/} $\phi:(X,s)\ra (Y,t)$ of d-critical loci
$(X,s),(Y,t)$ is a $\K$-scheme morphism $\phi:X\ra Y$ with
$\phi^\star(t)=s$. This makes d-critical loci into a category.

\label{sa3def1}
\end{dfn}

\begin{rem} {\bf(a)} For $(X,s)$ to be a (complex analytic or
algebraic) d-critical locus places strong local restrictions on the
singularities of $X$. For example, Behrend \cite{Behr} notes that if
$X$ has reduced local complete intersection singularities then
locally it cannot be the zeroes of an almost closed 1-form on a
smooth space, and hence not locally a critical locus, and
Pandharipande and Thomas \cite{PaTh} give examples which are zeroes
of almost closed 1-forms, but are not locally critical loci.
\smallskip

\noindent{\bf(b)} If
$X=\Crit(f)$ for holomorphic $f:U\ra\C$, then $f\vert_{X^\red}$ is
locally constant, and we can write $f=f^0+c$ uniquely near $X$ in
$U$ for $f^0:U\ra\C$ holomorphic with $\Crit(f^0)=X=\Crit(f)$,
$f^0\vert_{X^\red}=0$, and $c:U\ra\C$ locally constant
with~$c\vert_{X^\red}=f\vert_{X^\red}$.
Defining d-critical loci using $s\in H^0(\cSz_X)$
corresponds to remembering only the function $f^0$ near $X$ in $U$,
and forgetting the locally constant function
$f\vert_{X^\red}:X^\red\ra\C$. 
\smallskip

\noindent{\bf(c)} In \cite[ex. 2.16]{Joyc.2}, Joyce shows a case in which the
algebraic d-critical locus remembers more information, locally, than
the symmetric obstruction theory. In \cite[ex. 2.17]{Joyc.2}, Joyce 
shows that the (symmetric) obstruction theory remembers global, non-local
information which is forgotten by the algebraic d-critical locus.
\smallskip

\noindent{\bf(e)} One could think about critical charts as Kuranishi neighbourhoods on a topological space, and embeddings as analogous to coordinate changes between Kuranishi neighbourhoods.
\end{rem}

Here are~\cite[Prop.s 2.8, 2.30, Th.s 2.20, 2.28,
Def.~2.31,  Rem 2.32 \& Cor.~2.33]{Joyc.2}:

\begin{prop} Let\/ $\phi:X\ra Y$ be a smooth morphism of\/
$\K$-schemes. Suppose $t\in H^0(\cSz_Y),$ and set\/
$s:=\phi^\star(t)\in H^0(\cSz_X)$. If\/ $(Y,t)$ is a d-critical
locus, then\/ $(X,s)$ is a d-critical locus, and\/
$\phi:(X,s)\ra(Y,t)$ is a morphism of d-critical loci. Conversely,
if also $\phi:X\ra Y$ is surjective, then $(X,s)$ a d-critical locus
implies $(Y,t)$ is a d-critical locus.
\label{sa3prop1}
\end{prop}

\begin{thm} Suppose\/ $(X,s)$ is an algebraic d-critical locus, and\/
$(R,U,f,i),\ab(S,V,g,j)$ are critical charts on $(X,s)$. Then for each\/
$x\in R\cap S\subseteq X$ there exist subcharts
$(R',U',f',i')\subseteq(R,U,f,i),$ $(S',V',g',j')\subseteq
(S,V,g,j)$ with\/ $x\in R'\cap S'\subseteq X,$ a critical chart\/
$(T,W,h,k)$ on $(X,s),$ and embeddings $\Phi:(R',U',f',i')\hookra
(T,W,h,k),$ $\Psi:(S',V',g',j')\hookra(T,W,h,k)$.
\label{sa3thm2}
\end{thm}

\begin{thm} Let\/ $(X,s)$ be an algebraic d-critical locus, and\/
$X^\red\subseteq X$ the associated reduced\/ $\K$-subscheme. Then
there exists a line bundle $K_{X,s}$ on $X^\red$ which we call the
\begin{bfseries}canonical bundle\end{bfseries} of\/ $(X,s),$ which
is natural up to canonical isomorphism, and is characterized by the
following properties:
\begin{itemize}
\setlength{\itemsep}{0pt}
\setlength{\parsep}{0pt}
\item[{\bf(a)}] For each $x\in X^\red,$ there is a canonical
isomorphism
\e
\ka_x:K_{X,s}\vert_x\,{\buildrel\cong\over\longra}\,
\bigl(\La^{\rm top}T_x^*X\bigr){}^{\ot^2},
\label{sa3eq3}
\e
where $T_xX$ is the Zariski tangent space of\/ $X$ at\/~$x$.
\item[{\bf(b)}] If\/ $(R,U,f,i)$ is a critical chart on
$(X,s),$ there is a natural isomorphism
\e
\io_{R,U,f,i}:K_{X,s}\vert_{R^\red}\longra
i^*\bigl(K_U^{\ot^2}\bigr)\vert_{R^\red},
\label{sa3eq4}
\e
where $K_U=\La^{\dim U}T^*U$ is the canonical bundle of\/ $U$ in
the usual sense.

\item[{\bf(c)}] In the situation of\/ {\bf(b)\rm,} let\/ $x\in R$.
Then we have an exact sequence
\e
\begin{gathered}
{}\!\!\!\!\xymatrix@C=22pt@R=15pt{ 0 \ar[r] & T_xX
\ar[r]^(0.4){\d i\vert_x} & T_{i(x)}U
\ar[rr]^(0.53){\Hess_{i(x)}f} && T_{i(x)}^*U \ar[r]^{\d
i\vert_x^*}  & T_x^*X  \ar[r] & 0, }
\end{gathered}
\label{sa3eq5}
\e
and the following diagram commutes:
\begin{equation*}
\xymatrix@C=150pt@R=13pt{ *+[r]{K_{X,s}\vert_x}
\ar[dr]_{\io_{R,U,f,i}\vert_x} \ar[r]_(0.55){\ka_x} &
*+[l]{\bigl(\La^{\rm top}T_x^*X\bigr){}^{\ot^2}}
\ar[d]_(0.45){\al_{x,R,U,f,i}} \\
& *+[l]{K_U\vert_{i(x)}^{\ot^2},\!\!\!} }
\end{equation*}
where $\al_{x,R,U,f,i}$ is induced by taking top exterior powers
in\/~\eq{sa3eq5}.
\end{itemize}
\label{sa3thm3}
\end{thm}

\begin{prop} Suppose $\phi:(X,s)\ra(Y,t)$ is a morphism of
d-critical loci with\/ $\phi:X\ra Y$ smooth, as in Proposition\/
{\rm\ref{sa3prop1}}. The \begin{bfseries}relative cotangent
bundle\end{bfseries} $T^*_{X/Y}$ is a vector bundle of mixed rank on
$X$ in the exact sequence of coherent sheaves on $X\!:$
\e
\xymatrix@C=35pt{0 \ar[r] & \phi^*(T^*Y) \ar[r]^(0.55){\d\phi^*} &
T^*X \ar[r] & T^*_{X/Y} \ar[r] & 0. }
\label{sa3eq6}
\e
There is a natural isomorphism of line bundles on $X^\red\!:$
\e
\Up_\phi:\phi\vert_{X^\red}^* (K_{Y,t})\ot\bigl(\La^{\rm
top}T^*_{X/Y}\bigr)\big\vert_{X^\red}^{\ot^2}
\,{\buildrel\cong\over\longra}\,K_{X,s},
\label{sa3eq7}
\e
such that for each\/ $x\in X^\red$ the following diagram of
isomorphisms commutes:
\e
\begin{gathered}
\xymatrix@C=160pt@R=17pt{ *+[r]{K_{Y,t}
\vert_{\phi(x)}\ot\bigl(\La^{\rm top}T^*_{X/Y}\vert_x\bigr)^{\ot^2}}
\ar[r]_(0.7){\Up_\phi\vert_x} \ar[d]^{\ka_{\phi(x)}\ot\id} &
*+[l]{K_{X,s}\vert_x}
\ar[d]_{\ka_x} \\
*+[r]{\bigl(\La^{\rm top}T_{\phi(x)}^*Y\bigr)^{\ot^2}\ot
\bigl(\La^{\rm top}T^*_{X/Y}\vert_x\bigr)^{\ot^2}}
\ar[r]^(0.7){\up_x^{\ot^2}} & *+[l]{\bigl(\La^{\rm
top}T_x^*X\bigr)^{\ot^2},\!\!{}} }
\end{gathered}
\label{sa3eq8}
\e
where $\ka_x,\ka_{\phi(x)}$ are as in {\rm\eq{sa3eq3},} and\/
$\up_x:\La^{\rm top}T_{\phi(x)}^*Y\ot \La^{\rm top}T^*_{X/Y}
\vert_x\ra\La^{\rm top}T_x^*X$ is obtained by restricting
\eq{sa3eq6} to $x$ and taking top exterior powers.
\label{sa3prop2}
\end{prop}

\begin{dfn} Let $(X,s)$ be an algebraic d-critical locus, and
$K_{X,s}$ its canonical bundle from Theorem \ref{sa3thm3}. An {\it
orientation\/} on $(X,s)$ is a choice of square root line bundle
$K_{X,s}^{1/2}$ for $K_{X,s}$ on $X^\red$. That is, an orientation
is a line bundle $L$ on $X^\red$, together with an isomorphism
$L^{\ot^2}=L\ot L\cong K_{X,s}$. A d-critical locus with an
orientation will be called an {\it oriented d-critical locus}.
\label{sa3def2}
\end{dfn}

\begin{rem} In view of equation \eq{sa3eq3}, one might hope to
define a canonical orientation $K_{X,s}^{1/2}$ for a d-critical
locus $(X,s)$ by $K_{X,s}^{1/2}\big\vert_x=\La^{\rm top}T_x^*X$ for
$x\in X^\red$. However, {\it this does not work}, as the spaces
$\La^{\rm top}T_x^*X$ do not vary continuously with $x\in X^\red$ if
$X$ is not smooth. An example in \cite[Ex.~2.39]{Joyc.2} shows that
d-critical loci need not admit orientations.
\label{sa3rem1}
\end{rem}

In the situation of Proposition \ref{sa3prop2}, the factor
$(\La^{\rm top}T^*_{X/Y})\vert_{X^\red}^{\ot^2}$ in \eq{sa3eq7} has
a natural square root $(\La^{\rm top}T^*_{X/Y})\vert_{X^\red}$. Thus
we deduce:

\begin{cor} Let\/ $\phi:(X,s)\ra(Y,t)$ be a morphism of
d-critical loci with\/ $\phi:X\ra Y$ smooth. Then each orientation
$K_{Y,t}^{1/2}$ for\/ $(Y,t)$ lifts to a natural orientation
$K_{X,s}^{1/2}=\phi\vert_{X^\red}^*(K_{Y,t}^{1/2})\ot(\La^{\rm
top}T^*_{X/Y}) \vert_{X^\red}$ for~$(X,s)$.
\label{sa3cor1}
\end{cor}

\subsection{D-critical stacks}
\label{dcr.2}

In \cite[\S 2.7--\S 2.8]{Joyc.2} Joyce extends the material of
\S\ref{dcr.1} from $\K$-schemes to Artin $\K$-stacks. We work in the
context of the theory of {\it sheaves on Artin stacks} by Laumon and
Moret-Bailly \cite{LaMo}.

\begin{prop}[Laumon and Moret-Bailly \cite{LaMo}] Let\/ $X$ be an
Artin\/ $\K$-stack. The category of sheaves of sets on $X$ in the
lisse-\'etale topology is equivalent to the category $\Sh(X)$
defined as follows:
\smallskip

\noindent{\bf(A)} Objects $\cA$ of\/ $\Sh(X)$ comprise the following
data:
\begin{itemize}
\setlength{\itemsep}{0pt}
\setlength{\parsep}{0pt}
\item[{\bf(a)}] For each\/ $\K$-scheme $T$ and smooth\/ $1$-morphism
$t:T\ra X$ in $\Art_\K,$ we are given a sheaf of sets $\cA(T,t)$
on $T,$ in the \'etale topology.
\item[{\bf(b)}] For each\/ $2$-commutative diagram in
$\Art_\K\!:$
\e
\begin{gathered}
\xymatrix@C=50pt@R=6pt{ & U \ar[ddr]^u \\
\rrtwocell_{}\omit^{}\omit{^{\eta}} && \\
T  \ar[uur]^{\phi} \ar[rr]_t && X, }
\end{gathered}
\label{sa3eq9}
\e
where $T,U$ are schemes and\/ $t: T\ra X,$ $u:U\ra X$ are
smooth\/ $1$-morphisms in $\Art_\K,$ we are given a morphism
$\cA(\phi,\eta):\phi^{-1} (\cA(U,u)) \ra\cA(T,t)$ of \'etale
sheaves of sets on $T$.
\end{itemize}
This data must satisfy the following conditions:
\begin{itemize}
\setlength{\itemsep}{0pt}
\setlength{\parsep}{0pt}
\item[{\bf(i)}] If\/ $\phi:T\ra U$ in {\bf(b)} is \'etale, then
$\cA(\phi,\eta)$ is an isomorphism.
\item[{\bf(ii)}] For each\/ $2$-commutative diagram in $\Art_\K\!:$
\begin{equation*}
\xymatrix@C=70pt@R=6pt{ & V \ar[ddr]^v \\
\rrtwocell_{}\omit^{}\omit{^{\ze}} && \\
U  \ar[uur]^{\psi} \ar[rr]_(0.3)u && X, \\
\urrtwocell_{}\omit^{}\omit{^{\eta}} && \\
T \ar[uu]_{\phi} \ar@/_/[uurr]_t }
\end{equation*}
with $T,U,V$ schemes and\/ $t,u,v$ smooth, we must have
\begin{align*}
\cA\bigl(\psi\ci\phi,(\ze*\id_{\phi})\od\eta\bigr)
&=\cA(\phi,\eta)\ci\phi^{-1}(\cA(\psi,\ze))\quad\text{as
morphisms}\\
(\psi\ci\phi)^{-1}(\cA(V,v))&=\phi^{-1}\ci
\psi^{-1}(\cA(V,v))\longra\cA(T,t).
\end{align*}
\end{itemize}

\noindent{\bf(B)} Morphisms $\al:\cA\ra\cB$ of\/ $\Sh(X)$ comprise a
morphism $\al(T,t):\cA(T,t)\ra\cB(T,t)$ of \'etale sheaves of sets
on a scheme $T$ for all smooth\/ $1$-morphisms $t:T\ra X,$ such that
for each diagram \eq{sa3eq9} in {\bf(b)} the following commutes:
\begin{equation*}
\xymatrix@C=120pt@R=25pt{*+[r]{\phi^{-1}(\cA(U,u))}
\ar[d]^{\phi^{-1}(\al(U,u))} \ar[r]_(0.55){\cA(\phi,\eta)} &
*+[l]{\cA(T,t)} \ar[d]_{\al(T,t)} \\
*+[r]{\phi^{-1}(\cB(U,u))}
\ar[r]^(0.55){\cB(\phi,\eta)} & *+[l]{\cB(T,t).\!{}} }
\end{equation*}

\noindent{\bf(C)} Composition of morphisms $\cA\,{\buildrel\al
\over\longra}\,\cB\,{\buildrel\be\over\longra}\,\cC$ in $\Sh(X)$ is
$(\be\ci\al)(T,t)=\ab\be(T,t)\ab\ci\ab\al(T,t)$. Identity morphisms
$\id_\cA:\cA\ra\cA$ are $\id_\cA(T,t)=\id_{\cA(T,t)}$.

\smallskip

The analogue of all the above also holds for (\'etale) sheaves of\/
$\K$-vector spaces, sheaves of\/ $\K$-algebras, and so on, in place
of (\'etale) sheaves of sets.
Furthermore, the analogue of all the above holds for quasi-coherent
sheaves, (or coherent sheaves, or vector bundles, or line bundles)
on $X,$ where in {\bf(a)} $\cA(T,t)$ becomes a quasi-coherent sheaf
(or coherent sheaf, or vector bundle, or line bundle) on $T,$ in
{\bf(b)} we replace $\phi^{-1}(\cA(U,u))$ by the pullback\/
$\phi^*(\cA(U,u))$ of quasi-coherent sheaves (etc.), and\/
$\cA(\phi,\eta),\ab\al(T,t)$ become morphisms of quasi-coherent
sheaves (etc.) on\/~$T$.

\smallskip

We can also describe \begin{bfseries}global sections\end{bfseries}
of sheaves on Artin $\K$-stacks in the above framework: a global
section $s\in H^0(\cA)$ of\/ $\cA$ in part\/ {\bf(A)} assigns a
global section $s(T,t)\in H^0(\cA(T,t))$ of\/ $\cA(T,t)$ on\/ $T$
for all smooth\/ $t:T\ra X$ from a scheme $T,$ such that\/
$\cA(\phi,\eta)^*(s(U,u))=s(T,t)$ in $H^0(\cA(T,t))$ for all\/
$2$-commutative diagrams \eq{sa3eq9} with\/ $t,u$ smooth.
\label{sa3prop3}
\end{prop}

In \cite[Cor.~2.52]{Joyc.2} Joyce generalizes the sheaves $\cS_X,\cSz_X$ in
\S\ref{dcr.1} to Artin $\K$-stacks:

\begin{prop} Let\/ $X$ be an Artin $\K$-stack, and write\/
$\Sh(X)_\Kalg$ and\/ $\Sh(X)_\Kvect$ for the categories of sheaves
of\/ $\K$-algebras and\/ $\K$-vector spaces on $X$ defined in
Proposition\/ {\rm\ref{sa3prop3}}. Then:
\begin{itemize}
\setlength{\itemsep}{0pt}
\setlength{\parsep}{0pt}
\item[{\bf(a)}] We may define canonical objects\/ $\cS_X$ in
both\/ $\Sh(X)_\Kalg$ and\/ $\Sh(X)_\Kvect$ by
$\cS_X(T,t):=\cS_T$ for all smooth morphisms $t:T\ra X$ for
$T\in\Sch_\K,$ for $\cS_T$ as in\/ {\rm\S\ref{dcr.1}} taken to be
a sheaf of\/ $\K$-algebras (or $\K$-vector spaces) on $T$ in the
\'etale topology, and\/ $\cS_X(\phi,\eta)
:=\phi^\star:\phi^{-1}(\cS_X(U,u))=\phi^{-1}(\cS_U)\ra\cS_T
=\cS_X(T,t)$ for all\/ $2$-commutative diagrams \eq{sa3eq9} in
$\Art_\K$ with\/ $t,u$ smooth, where $\phi^\star$ is as in
{\rm\S\ref{dcr.1}}.
\item[{\bf(b)}] There is a natural decomposition
$\cS_X\!=\!\K_X\!\op\!\cSz_X$ in $\Sh(X)_\Kvect$ induced by the
splitting $\cS_X(T,t)\!=\!\cS_T\!=\!\K_T\op\cSz_T$ in
{\rm\S\ref{dcr.1},} where $\K_X$ is a sheaf of\/ $\K$-subalgebras
of\/ $\cS_X$ in $\Sh(X)_\Kalg,$ and\/ $\cSz_X$ a sheaf of
ideals~in\/~$\cS_X$.
\end{itemize}
\label{sa3prop4}
\end{prop}
Here \cite[Def. 2.53]{Joyc.2} is the generalization of Definition
\ref{sa3def1} to Artin stacks.

\begin{dfn} A {\it d-critical stack\/} $(X,s)$ is an Artin
$\K$-stack $X$ and a global section $s\in H^0(\cSz_X)$, where
$\cSz_X$ is as in Proposition \ref{sa3prop4}, such that
$\bigl(T,s(T,t)\bigr)$ is an algebraic d-critical locus in the sense
of Definition \ref{sa3def1} for all smooth morphisms $t:T\ra X$
with~$T\in\Sch_\K$.
\label{sa3def3}
\end{dfn}

Here is a convenient way to understand
d-critical stacks $(X,s)$ in terms of d-critical structures on an
atlas $t:T\ra X$ for~$X$ from \cite[Prop.~2.54]{Joyc.2}.

\begin{prop} Let\/ $X$ be an Artin $\K$-stack, and\/ $t:
T\ra X$ a smooth atlas for $X$. Then $T\t_{t,X,t}T$ is equivalent to
a $\K$-scheme $U$ as $t$ is representable and\/ $T$ a scheme, so we
have a $2$-Cartesian diagram
\e
\begin{gathered}
\xymatrix@C=90pt@R=21pt{ *+[r]{U} \ar[d]^{\pi_1} \ar[r]_(0.3){\pi_2}
\drtwocell_{}\omit^{}\omit{^{\eta}} &
*+[l]{
T} \ar[d]_t \\
*+[r]{T} \ar[r]^(0.7)t & *+[l]{X} }
\end{gathered}
\label{sa3eq10}
\e
in $\Art_\K,$ with\/ $\pi_1,\pi_2:U\ra T$ smooth morphisms in
$\Sch_\K$. Also $T,U,\pi_1,\pi_2$ can be naturally completed to a
smooth groupoid in $\Sch_\K,$ and\/ $X$ is equivalent in $\Art_\K$
to the associated groupoid stack\/ $[U\rra T]$.
\begin{itemize}
\setlength{\itemsep}{0pt}
\setlength{\parsep}{0pt}
\item[{\bf(i)}] Let\/ $\cS_X$ be as in Proposition\/
{\rm\ref{sa3prop4},} and\/ $\cS_T,\cS_U$ be as in
{\rm\S\ref{dcr.1},} regarded as sheaves on $T,U$ in the \'etale
topology, and define $\pi_i^\star:\pi_i^{-1}(\cS_T)\ra \cS_U$ as
in {\rm\S\ref{dcr.1}} for $i=1,2$. Consider the map
$t^*:H^0(\cS_X)\ra H^0(\cS_T)$ mapping $t^*:s\mapsto s(T,t)$.
This is injective, and induces a bijection
\e
t^*:H^0(\cS_X)\,{\buildrel\cong\over\longra}\,\bigl\{s'\in
H^0(\cS_T):\text{$\pi_1^\star(s')= \pi_2^\star(s')$ in
$H^0(\cS_U)$}\bigr\}.
\label{sa3eq11}
\e
The analogue holds for $\cSz_X,\cSz_T,\cSz_U$.
\item[{\bf(ii)}] Suppose $s\in H^0(\cSz_X),$ so that\/ $t^*(s)\in
H^0(\cSz_T)$ with\/ $\pi_1^\star\ci t^*(s)=\pi_2^\star\ci
t^*(s)$. Then $(X,s)$ is a d-critical stack if and only if\/
$\bigl(T,t^*(s)\bigr)$ is an algebraic d-critical locus, and
then\/ $\bigl(U,\pi_1^\star\ci t^*(s)\bigr)$ is also an
algebraic d-critical locus.
\end{itemize}
\label{sa3prop5}
\end{prop}

In \cite[Ex.~2.55]{Joyc.2} we consider quotient stacks $X=[T/G]$.

\begin{ex} Suppose an algebraic $\K$-group $G$ acts on a $\K$-scheme
$T$ with action $\mu:G\t T\ra T$, and write $X$ for the quotient
Artin $\K$-stack $[T/G]$. Then as in \eq{sa3eq10} there is a natural
2-Cartesian diagram
\begin{equation*}
\xymatrix@C=110pt@R=21pt{ *+[r]{G\t T} \ar[d]^{\pi_T}
\ar[r]_(0.3){\mu} \drtwocell_{}\omit^{}\omit{^{\eta}} &
*+[l]{T} \ar[d]_t \\
*+[r]{T} \ar[r]^(0.6)t & *+[l]{X=[T/G],\!{}} }
\end{equation*}
where $t:T\ra X$ is a smooth atlas for $X$. If $s'\in H^0(\cSz_T)$
then $\pi_1^\star(s')=\pi_2^\star(s')$ in \eq{sa3eq11} becomes
$\pi_T^\star(s')=\mu^\star(s')$ on $G\t T$, that is, $s'$ is
$G$-invariant. Hence, Proposition \ref{sa3prop5} shows that
d-critical structures $s$ on $X=[T/G]$ are in 1-1 correspondence
with $G$-invariant d-critical structures $s'$ on~$T$.
\label{sa3ex}
\end{ex}

Here \cite[Th.~2.56]{Joyc.2} is an analogue of Theorem~\ref{sa3thm3}.

\begin{thm} Let\/ $(X,s)$ be a d-critical stack. Using the
description of quasi-coherent sheaves on $X^\red$ in Proposition
{\rm\ref{sa3prop3}} there is a line bundle $K_{X,s}$ on the
reduced\/ $\K$-substack\/ $X^\red$ of\/ $X$ called the
\begin{bfseries}canonical bundle\end{bfseries} of\/ $(X,s),$
unique up to canonical isomorphism, such that:
\begin{itemize}
\setlength{\itemsep}{0pt}
\setlength{\parsep}{0pt}
\item[{\bf(a)}] For each point\/ $x\in X^\red\subseteq X$ we
have a canonical isomorphism
\e
\ka_x:K_{X,s}\vert_x\,{\buildrel\cong\over\longra}\,
\bigl(\La^{\rm top}T_x^*X\bigr)^{\ot^2} \ot\bigl(\La^{\rm
top}\fIso_x(X)\bigr)^{\ot^2},
\label{sa3eq12}
\e
where $T_x^*X$ is the Zariski cotangent space of\/ $X$ at\/ $x,$
and\/ $\fIso_x(X)$ the Lie algebra of the isotropy group
(stabilizer group) $\Iso_x(X)$ of\/ $X$ at\/~$x$.
\item[{\bf(b)}] If\/ $T$ is a $\K$-scheme and\/ $t:
T\ra X$ a smooth $1$-morphism, so that\/ $t^\red:T^\red\ra
X^\red$ is also smooth, then there is a natural isomorphism of
line bundles on $T^\red\!:$
\e
\Ga_{T,t}:K_{X,s}(T^\red,t^\red)\,{\buildrel\cong\over\longra}\,
K_{T,s(T,t)}\ot \bigl(\La^{\rm top}T^*_{
T/X}\bigr)\big\vert_{T^\red}^{\ot^{-2}}.
\label{sa3eq13}
\e
Here $\bigl(T,s(T,t)\bigr)$ is an algebraic d-critical locus by
Definition\/ {\rm\ref{sa3def3},} and\/ $K_{T,s(T,t)}\ra T^\red$
is its canonical bundle from Theorem\/~{\rm\ref{sa3thm3}}.
\item[{\bf(c)}] If\/ $t:T\ra X$ is a smooth
$1$-morphism, we have a distinguished triangle
in~$D_{\qcoh}(T)\!:$
\e
\xymatrix@C=40pt{ t^*(\bL_X) \ar[r]^(0.55){\bL_t} & \bL_T \ar[r]
& T^*_{T/X} \ar[r] & t^*(\bL_X)[1], }
\label{sa3eq14}
\e
where $\bL_T,\bL_X$ are the cotangent complexes of\/ $T,X,$
and\/ $T^*_{T/X}$ the relative cotangent bundle of\/ $t: T\ra
X,$ a vector bundle of mixed rank on\/ $T$. Let\/ $p\in
T^\red\subseteq T,$ so that\/ $t(p):=t\ci p\in X$. Taking the
long exact cohomology sequence of\/ \eq{sa3eq14} and restricting
to $p\in T$ gives an exact sequence
\e
0 \longra T^*_{t(p)}X \longra T^*_pT \longra T^*_{ T/X}\vert_p
\longra \fIso_{t(p)}(X)^* \longra 0.
\label{sa3eq15}
\e
Then the following diagram commutes:
\begin{equation*}
\xymatrix@!0@C=108pt@R=50pt{*+[r]{K_{X,s}\vert_{t(p)}}
\ar[d]^{\ka_{t(p)}} \ar@{=}[r] & K_{X,s}(T^\red,t^\red)\vert_p
\ar[rr]_(0.3){\Ga_{T,t}\vert_p} && *+[l]{K_{T,s(T,t)}\vert_p\ot
\bigl(\La^{\rm top}T^*_{\smash{
T/X}}\bigr)\big\vert_p^{\ot^{-2}}} \ar[d]_{\ka_p\ot\id} \\
*+[r]{\bigl(\La^{\rm top}T_{t(p)}^*X\bigr)^{\ot^2}\!\!\ot\!\bigl(\La^{\rm
top}\fIso_{t(p)}(X)\bigr)^{\ot^2}} \ar[rrr]^(0.54){\al_p^2} &&&
*+[l]{\bigl(\La^{\rm top}T^*_pT\bigr)^{\ot^2}\!\!\ot\!
\bigl(\La^{\rm top}T^*_{T/X}\bigr) \big\vert_p^{\ot^{-2}},} }
\end{equation*}
where $\ka_p,\ka_{t(p)},\Ga_{T,t}$ are as in {\rm\eq{sa3eq3},
\eq{sa3eq12}} and\/ {\rm\eq{sa3eq13},} respectively, and\/
$$\al_p:\La^{\rm top}T_{t(p)}^*X\ot\La^{\rm
top}\fIso_{t(p)}(X)\,{\buildrel\cong\over\longra}\,\La^{\rm
top}T^*_pT\ot\La^{\rm top}T^*_{T/X}\vert^{-1}_p$$ is induced by
taking top exterior powers in\/~\eq{sa3eq15}.
\end{itemize}
\label{sa3thm5}
\end{thm}

Here \cite[Def.~2.57]{Joyc.2} is the analogue of
Definition~\ref{sa3def2}:

\begin{dfn} Let $(X,s)$ be a d-critical stack, and $K_{X,s}$
its canonical bundle from Theorem \ref{sa3thm5}. An {\it
orientation\/} on $(X,s)$ is a choice of square root line bundle
$K_{X,s}^{1/2}$ for $K_{X,s}$ on $X^\red$. That is, an orientation
is a line bundle $L$ on $X^\red$, together with an isomorphism
$L^{\ot^2}=L\ot L\cong K_{X,s}$. A d-critical stack with an
orientation will be called an {\it oriented d-critical stack}.
\label{sa3def4}
\end{dfn}

Let $(X,s)$ be an oriented d-critical stack. Then for each smooth
$t:T\ra X$ we have a square root $K_{X,s}^{1/2} (T^\red,t^\red)$.
Thus by \eq{sa3eq13}, $K_{X,s}^{1/2}(T^\red,t^\red)\ot (\La^{\rm
top}\bL_{\smash{T/X}})\vert_{T^\red}$ is a square root for
$K_{T,s(T,t)}$. This proves~\cite[Lem.~2.58]{Joyc.2}:

\begin{lem} Let\/ $(X,s)$ be a d-critical stack. Then an
orientation $K_{X,s}^{1/2}$ for $(X,s)$ determines a canonical
orientation\/ $K_{T,s(T,t)}^{1/2}$ for the algebraic
d-critical locus $\bigl(T,s(T,t)\bigr),$ for all smooth\/ $t: T\ra
X$ with\/ $T$ a $\K$-scheme.
\label{sa3lem1}
\end{lem}

\subsection{Equivariant d-critical loci}
\label{eqdcr}

Here we summarizes some results about group actions on algebraic d-critical loci from \cite{Joyc2}.

\begin{dfn} Let $(X,s)$ be an algebraic d-critical locus over $\K$, and $\mu:G\t X\ra X$ an action of an algebraic $\K$-group $G$ on the $\K$-scheme $X$. We also write the action as $\mu(\ga):X\ra X$ for $\ga\in G$. We say that $(X,s)$ is $G$-{\it invariant\/} if $\mu(\ga)^\star(s)=s$ for all $\ga\in G$, or equivalently, if $\mu^\star(s)=\pi_X^\star(s)$ in $H^0(\cSz_{G\t X})$, where $\pi_X:G\t X\ra X$ is the projection.

\smallskip

Let $\chi:G\ra\bG_m$ be a morphism of algebraic $\K$-groups, that is, a character of $G$, where $\bG_m=\K\sm\{0\}$ is the multiplicative group. We say that $(X,s)$ is $G$-{\it equivariant, with character\/} $\chi,$ if $\mu(\ga)^\star(s)=\chi(\ga)\cdot s$ for all $\ga\in G$, or equivalently, if $\mu^\star(s)=(\chi\ci\pi_G)\cdot(\pi_X^\star(s))$ in $H^0(\cSz_{G\t X})$, where $H^0(\O_G)\ni\chi$ acts on $H^0(\cSz_{G\t X})$ by multiplication, as $G$ is a smooth $\K$-scheme.

\smallskip

Suppose $(X,s)$ is $G$-invariant or $G$-equivariant, with $\chi=1$ in the $G$-invariant case. We call a critical chart $(R,U,f,i)$ on $(X,s)$ with a $G$-action $\rho:G\t U\ra U$ a $G$-{\it equivariant critical chart\/} if $R\subseteq X$ is a $G$-invariant open subscheme, and $i:R\hookra U$, $f:U\ra\bA^1$ are equivariant with respect to the actions $\mu\vert_{G\t R},\rho,\chi$ of $G$ on $R,U,\bA^1$, respectively.

\smallskip

We call a subchart $(R',U',f',i')\subseteq(R,U,f,i)$ a $G$-{\it equivariant subchart\/} if $R'\subseteq R$ and $U'\subseteq U$ are $G$-invariant open subschemes. Then $(R',U',f',i'),\rho'$ is a $G$-equivariant critical chart, where~$\rho'=\rho\vert_{G\t U'}$.
\label{dc2def7}
\end{dfn}

Note that $X$ may not be covered by $G$-equivariant critical charts without extra assumptions on~$X,G$.
We will restrict to the case when $G$ is a torus, with a `good' action on~$X$:

\begin{dfn} Let $X$ be a $\K$-scheme, $G$ an algebraic $\K$-torus, and $\mu:G\t X\ra X$ an action of $G$ on $X$. We call $\mu$ a {\it good action\/} if $X$ admits a Zariski open cover by $G$-invariant affine open $\K$-subschemes $U\subseteq X$.
\label{dc2def8}
\end{dfn}

A torus-equivariant d-critical locus $(X,s)$ admits an open cover by equivariant critical charts if and only if the torus action is good:

\begin{prop} Let\/ $(X,s)$ be an algebraic d-critical locus which is invariant or equivariant under the action $\mu:G\t X\ra X$ of an algebraic torus $G$.
\smallskip

\noindent{\bf(a)} If\/ $\mu$ is good then for all\/ $x\in X$ there exists a $G$-equivariant critical chart\/ $(R,U,f,i),\rho$ on $(X,s)$ with\/ $x\in R,$ and we may take $\dim U=\dim T_xX$.
\smallskip

\noindent{\bf(b)} Conversely, if for all\/ $x\in X$ there exists a $G$-equivariant critical chart\/ $(R,U,f,i),\rho$ on $(X,s)$ with\/ $x\in R,$ then $\mu$ is good.
\label{dc2prop14}
\end{prop}

\section{Derived symplectic structures in Donaldson--Thomas theory}
\label{ourpapers}

We are now going to use derived algebraic geometry from \cite{PTVV} and summarize the main results from the sequel \cite{BBJ,BBDJS,BJM,BBBJ}
and their consequences in Donaldson--Thomas theory. Some of them will not be used 
to prove our main results stated in \S\ref{main.1}, but we will expose them as they contribute 
to a whole picture of the theory.

\subsection{Symplectic derived schemes and critical loci}

Here we summarizes the main results from \cite{BBJ}. The following is \cite[Thm. 5.18]{BBJ}.

\begin{thm} Let\/ $\bX$ be a derived\/ $\K$-scheme with\/
$k$-shifted symplectic form $\ti\om$ for $k<0,$ and\/ $x\in\bX$.
Then there exists a standard form cdga $A$ over $\K$ which is
minimal at\/ $p\in\Spec H^0(A)$ in the sense of \cite[\S 4]{BBJ}, a $k$-shifted symplectic form\/
$\om$ on $\bSpec A,$ and a morphism $\bs f:\bU=\bSpec A\ra\bX$
with\/ $\bs f(p)=x$ and\/ $\bs f^*(\ti\om)\sim\om,$ such that
if\/ $k$ is odd or divisible by $4,$ then $\bs f$ is
a Zariski open inclusion, and\/ $A,\om$ are in Darboux form, 
and if\/ $k\equiv 2\mod 4,$ then $\bs f$ is
\'etale, and\/ $A,\om$ are in strong Darboux form, as 
in \cite[\S 5]{BBJ}.
\label{sa2thm3}
\end{thm}

Let $Y$ be a Calabi--Yau $m$-fold over $\K$,
that is, a smooth projective $\K$-scheme with $H^i(\O_Y)=\K$ for
$i=0,m$ and $H^i(\O_Y)=0$ for $0<i<m$. Suppose $\cM$ is a classical
moduli $\K$-scheme of simple coherent sheaves in $\coh(Y)$, where we
call $F\in\coh(Y)$ {\it simple\/} if $\Hom(F,F)=\K$. More generally,
suppose $\cM$ is a moduli $\K$-scheme of simple complexes of
coherent sheaves in $D^b\coh(Y)$, where we call $F^\bu\in
D^b\coh(Y)$ {\it simple\/} if $\Hom(F^\bu,F^\bu)=\K$ and
$\mathop{\rm Ext}^{<0}(F^\bu,F^\bu)=0$. Such moduli spaces $\cM$ are
only known to be algebraic $\K$-spaces in general, but we assume
$\cM$ is a $\K$-scheme.
Then $\cM=t_0(\bs\cM)$, for $\bs\cM$ the corresponding derived
moduli $\K$-scheme. To make $\cM,\bs\cM$ into schemes rather than
stacks, we consider moduli of sheaves or complexes with fixed
determinant. Then Pantev et al.\ \cite[\S 2.1]{PTVV} prove $\bs\cM$
has a $(2-m)$-shifted symplectic structure $\om$, so Theorem
\ref{sa2thm3} shows that $(\bs\cM,\om)$ is Zariski locally modelled
on $(\bSpec A,\om)$, and $\cM$ is Zariski locally modelled on $\Spec
H^0(A)$. In the case $m=3$, so that $k=-1$, we get \cite[Cor. 5.19]{BBJ}:

\begin{cor} Suppose $Y$ is a Calabi--Yau\/ $3$-fold over a field\/
$\K,$ and\/ $\cM$ is a classical moduli $\K$-scheme of simple
coherent sheaves, or simple complexes of coherent sheaves, on $Y$.
Then for each\/ $[F]\in\cM,$ there exist a smooth\/ $\K$-scheme $U$
with\/ $\dim U=\dim\Ext^1(F,F),$ a regular function $f:U\ra\bA^1,$
and an isomorphism from $\Crit(f)\subseteq U$ to a Zariski open
neighbourhood of\/ $[F]$ in\/~$\cM$.
\label{da5cor1}
\end{cor}

Here $\dim U=\dim\Ext^1(F,F)$ comes from $A$ minimal at $p$ and $\bs
f(p)=[F]$ in Theorem \ref{sa2thm3}. This is a new result in Donaldson--Thomas theory. 
We already explained that when $\K=\C$ and
$\cM$ is a moduli space of simple coherent sheaves on $Y$, using
gauge theory and transcendental complex methods, Joyce and Song
\cite[Th.~5.4]{JoSo} prove that the underlying complex analytic
space $\cM^\an$ of $\cM$ is locally of the form $\Crit(f)$ for $U$ a
complex manifold and $f:U\ra\C$ a holomorphic function. Behrend and
Getzler announced the analogue of \cite[Th.~5.4]{JoSo} for moduli of
complexes in $D^b\coh(Y)$, but the proof has not yet appeared. Over
general $\K$, as in Kontsevich and Soibelman \cite[\S 3.3]{KoSo1}
the formal neighbourhood $\hat\cM_{[F]}$ of $\cM$ at any $[F]\in\cM$
is isomorphic to the critical locus $\Crit(\hat f)$ of a formal
power series $\hat f$ on $\Ext^1(F,F)$ with only cubic and higher
terms.

\smallskip

Here are \cite[Thm. 6.6 \& Cor. 6.7]{BBJ}:

\begin{thm} Suppose\/ $(\bX,\ti\om)$ is a $-1$-shifted symplectic
derived\/ $\K$-scheme, and let\/ $X=t_0(\bX)$ be the associated
classical\/ $\K$-scheme of\/ ${\bX}$. Then $X$ extends uniquely to
an algebraic d-critical locus\/ $(X,s),$ with the property that
whenever\/ $(\bSpec A,\om)$ is a  $-1$-shifted symplectic derived\/
$\K$-scheme in Darboux form with Hamiltonian $H\in A(0),$ as in 
\cite[Ex.s 5.8 \& 5.15]{BBJ}, and\/ $\bs
f:\bSpec A\ra\bX$ is an equivalence in $\dSch_\K$ with a Zariski
open derived\/ $\K$-subscheme\/ $\bR\subseteq\bX$ with\/ $\bs
f^*(\ti\om)\sim\om,$ writing\/ $U=\Spec A(0),$ $R=t_0(\bR),$
$f=t_0(\bs f)$ so that\/ $H:U\ra\bA^1$ is regular and\/
$f:\Crit(H)\ra R$ is an isomorphism, for $\Crit(H)\subseteq U$ the
classical critical locus of\/ $H,$ then $(R,U,H,f^{-1})$ is a
critical chart on\/~$(X,s)$.
The canonical bundle $K_{X,s}$ from Theorem\/ {\rm\ref{sa3thm3}} is
naturally isomorphic to the determinant line bundle
$\det(\bL_{\bX})\vert_{X^\red}$ of the cotangent complex\/
$\bL_{\bX}$ of\/~$\bX$.
\label{da6thm4}
\end{thm}
We can think of Theorem \ref{da6thm4} as defining a {\it truncation
functor}
\e
\begin{split}
F:\bigl\{&\text{category of $-1$-shifted symplectic derived
$\K$-schemes $(\bX,\om)$}\bigr\}\\
&\longra\bigl\{\text{category of algebraic d-critical loci
$(X,s)$ over $\K$}\bigr\},
\end{split}
\label{da6eq3}
\e
where the morphisms $\bs f:(\bX,\om)\ra(\bY,\om')$ in the first line
are (homotopy classes of) \'etale maps $\bs f:\bX\ra\bY$ with $\bs
f^*(\om')\sim\om$, and the morphisms $f:(X,s)\ra (Y,t)$ in the
second line are \'etale maps $f:X\ra Y$ with~$f^*(t)=s$.
In \cite[Ex.~2.17]{Joyc2} Joyce gives an example of $-1$-shifted
symplectic derived schemes $(\bX,\om),(\bY,\om')$, both global
critical loci, such that $\bX,\bY$ are not equivalent as derived
$\K$-schemes, but their truncations $F(\bX,\om),F(\bY,\om')$ are
isomorphic as algebraic d-critical loci. Thus, the functor $F$ in
\eq{da6eq3} is not full.

\smallskip

Suppose again $Y$ is a Calabi--Yau 3-fold over $\K$ and $\cM$ a classical
moduli $\K$-scheme of simple coherent sheaves in $\coh(Y)$. Then
Thomas \cite{Thom} defined a natural {\it perfect obstruction
theory\/} $\phi:\cE^\bu\ra\bL_\cM$ on $\cM$ in the sense of Behrend
and Fantechi \cite{BeFa}, and Behrend \cite{Behr} showed that
$\phi:\cE^\bu\ra\bL_\cM$ can be made into a {\it symmetric
obstruction theory}. More generally, if $\cM$ is a moduli
$\K$-scheme of simple complexes of coherent sheaves in $D^b\coh(Y)$,
then Huybrechts and Thomas \cite{HuTh} defined a natural symmetric
obstruction theory on~$\cM$.
Now in derived algebraic geometry $\cM=t_0(\bs\cM)$ for $\bs\cM$ the
corresponding derived moduli $\K$-scheme, and the obstruction theory
$\phi:\cE^\bu\ra\bL_\cM$ from \cite{HuTh,Thom} is
$\bL_{t_0}:\bL_{\bs\cM}\vert_\cM\ra\bL_\cM$. Pantev et al.\ \cite[\S
2.1]{PTVV} prove $\bs\cM$ has a $-1$-shifted symplectic structure
$\om$, and the symmetric structure on $\phi:\cE^\bu\ra\bL_\cM$ from
\cite{Behr} is $\om^0\vert_\cM$. So as for Corollary \ref{da5cor1},
Theorem \ref{da6thm4} implies:

\begin{cor} Suppose $Y$ is a Calabi--Yau\/ $3$-fold over\/ $\K,$
and\/ $\cM$ is a classical moduli\/ $\K$-scheme of simple coherent
sheaves in $\coh(Y),$ or simple complexes of coherent sheaves in
$D^b\coh(Y),$ with perfect obstruction theory\/
$\phi:\cE^\bu\ra\bL_\cM$ as in Thomas\/ {\rm\cite{Thom}} or
Huybrechts and Thomas\/ {\rm\cite{HuTh}}. Then $\cM$ extends
naturally to an algebraic d-critical locus $(\cM,s)$. The canonical
bundle $K_{\cM,s}$ from Theorem\/ {\rm\ref{sa3thm3}} is naturally
isomorphic to $\det(\cE^\bu)\vert_{\cM^\red}$.
\label{da6cor1}
\end{cor}

\subsection{Categorification using perverse sheaves and motives}

Here we summarizes the main results from \cite{BBDJS}.
This particular section is not really used in the sequel, but it completes the discussion started 
in \S\ref{dt3.1.2}.
The following theorems are \cite[Cor. 6.10 \& Cor. 6.11]{BBDJS}:

\begin{thm} Let\/ $(\bs X,\om)$ be a $-1$-shifted symplectic
derived scheme over\/ $\C$ in the sense of Pantev et al.\
{\rm\cite{PTVV},} and\/ $X=t_0(\bs X)$ the associated classical\/
$\C$-scheme. Suppose we are given a square root\/
$\smash{\det(\bL_{\bs X})\vert_X^{1/2}}$ for $\det(\bL_{\bs
X})\vert_X$. Then we may define $P_{\bs X,\om}^\bu\in\Perv(X),$
uniquely up to canonical isomorphism, and isomorphisms $\Si_{\bs
X,\om}:P_{\bs X,\om}^\bu\ra \bD_X(P_{\bs X,\om}^\bu),$ $\Tau_{\bs
X,\om}:P_{\bs X,\om}^\bu\ra P_{\bs X,\om}^\bu$.
The same applies for $\cD$-modules and mixed Hodge modules on $X,$
and for $l$-adic perverse sheaves and\/ $\cD$-modules on $X$ if\/
$\bs X$ is over $\K$ with\/~$\mathop{\rm char}\K=0$.
\label{sm6cor2}
\end{thm}

\begin{thm} Let\/ $Y$ be a Calabi--Yau\/ $3$-fold over\/ $\C,$
and\/ $\cM$ a classical moduli\/ $\K$-scheme of simple coherent
sheaves in $\coh(Y),$ or simple complexes of coherent sheaves in
$D^b\coh(Y),$ with natural (symmetric) obstruction theory\/
$\phi:\cE^\bu\ra\bL_\cM$ as in Behrend\/ {\rm\cite{Behr},} Thomas\/
{\rm\cite{Thom},} or Huybrechts and Thomas\/ {\rm\cite{HuTh}}.
Suppose we are given a square root\/ $\det(\cE^\bu)^{1/2}$ for
$\det(\cE^\bu)$. Then we may define $P_\cM^\bu\in\Perv(\cM),$
uniquely up to canonical isomorphism, and isomorphisms
$\Si_\cM:P_\cM^\bu\ra \bD_\cM(P_\cM^\bu),$ $\Tau_\cM:P_\cM^\bu\ra
P_\cM^\bu$.
The same applies for $\cD$-modules and mixed Hodge modules on $\cM,$
and for $l$-adic perverse sheaves and\/ $\cD$-modules on $\cM$ if\/
$Y,\cM$ are over $\K$ with\/~$\mathop{\rm char}\K=0$.

\label{sm6cor3}
\end{thm}

Theorem \ref{sm6cor3} is relevant to the {\it categorification\/}
of Donaldson--Thomas theory as discussed in \S\ref{dt3.1.2}. As in \cite[\S 1.2]{Behr}, the perverse
sheaf $P_{\cM_{\rm st}^\al(\tau)}^\bu$ has pointwise Euler
characteristic $\chi\bigl(P_{\cM_{\rm st}^\al(\tau)}^\bu\bigr)=\nu$.
This implies that when $A$ is a field, say $A=\Q$, the
(compactly-supported) hypercohomologies $\bH^*\bigl(P_{\cM_{\rm
st}^\al(\tau)}^\bu\bigr), \bH^*_{\rm cs}\bigl(P_{\cM_{\rm
st}^\al(\tau)}^\bu\bigr)$ satisfy
\begin{align*}
\ts\sum\limits_{k\in\Z}(-1)^k\dim \bH^k\bigl(P_{\cM_{\rm
st}^\al(\tau)}^\bu\bigr)&= \ts\sum\limits_{k\in\Z}(-1)^k\dim
\bH^k_{\rm cs}\bigl(P_{\cM_{\rm st}^\al(\tau)}^\bu\bigr)=\chi\bigl(\cM_{\rm st}^\al(\tau),\nu\bigr)=DT^\al(\tau),
\end{align*}
where $\bH^k\bigl(P_{\cM_{\rm st}^\al(\tau)}^\bu\bigr) \cong
\bH^{-k}_{\rm cs}\bigl(P_{\cM_{\rm st}^\al(\tau)}^\bu\bigr){}^*$ by
Verdier duality. That is, we have produced a natural graded
$\Q$-vector space $\bH^*\bigl(P_{\cM_{\rm st}^\al(\tau)}^\bu\bigr)$,
thought of as some kind of generalized cohomology of $\cM_{\rm
st}^\al(\tau)$, whose graded dimension is $DT^\al(\tau)$. This gives
a new interpretation of the Donaldson--Thomas
invariant~$DT^\al(\tau)$.

\smallskip

In fact, as discussed at length in \cite[\S 3]{Szen}, the first
natural ``refinement'' or ``quantization'' direction of a
Donaldson--Thomas invariant $DT^\al(\tau)\in\Z$ is not the
Poincar\'e polynomial of this cohomology, but its weight polynomial
\begin{equation*}
w\bigl(\bH^*(P_{\cM_{\rm st}^\al(\tau)}^\bu), t\bigr)
\in\Z\bigl[t^{\pm\frac{1}{2}}\bigr],
\end{equation*}
defined using the mixed Hodge structure on the cohomology of the
mixed Hodge module version of $P_{\cM_{\rm st}^\al(\tau)}^\bu$,
which exists assuming that $\cM_{\rm st}^\al(\tau)$ is projective.

\smallskip

The material above is related to work by other authors. The idea of
categorifying Donaldson--Thomas invariants using perverse sheaves or
$\cD$-modules is probably first due to Behrend \cite{Behr}, and for
Hilbert schemes $\mathop{\rm Hilb}^n(Y)$ of a Calabi--Yau 3-fold $Y$
is discussed by Dimca and Szendr\H oi \cite{DiSz} and Behrend, Bryan
and Szendr\H oi \cite[\S 3.4]{BBS}, using mixed Hodge modules.
Corollary \ref{sm6cor3} answers a question of Joyce and
Song~\cite[Question~5.7(a)]{JoSo}.

\smallskip

As in \cite{JoSo,KoSo1} representations of {\it quivers with
superpotentials\/} $(Q,W)$ give 3-Calabi--Yau triangulated
categories, and one can define Donaldson--Thomas type invariants
$DT^\al_{Q,W}(\tau)$ `counting' such representations, which are
simple algebraic `toy models' for Donaldson--Thomas invariants of
Calabi--Yau 3-folds. Kontsevich and Soibelman \cite{KoSo2} explain
how to categorify these quiver invariants $DT^\al_{Q,W}(\tau)$, and
define an associative multiplication on the categorification to make
a {\it Cohomological Hall Algebra}. This paper was motivated by the
aim of extending \cite{KoSo2} to define Cohomological Hall Algebras
for Calabi--Yau 3-folds.

\smallskip

The square root $\det(\cE^\bu)^{1/2}$ required in Corollary
\ref{sm6cor3} corresponds roughly to {\it orientation data\/} in the
work of Kontsevich and Soibelman \cite[\S 5]{KoSo1}, \cite{KoSo2}.

\smallskip

Finally, we point out that Kiem and Li \cite{KiLi} have recently proved an
analogue of Corollary \ref{sm6cor3} by complex analytic methods,
beginning from Joyce and Song's result \cite[Th.~5.4]{JoSo}, proved
using gauge theory, that $\cM_{\rm st}^\al(\tau)$ is locally
isomorphic to $\Crit(f)$ as a complex analytic space, for $V$ a
complex manifold and $f:V\ra\C$ holomorphic.

\medskip

Now, we summarizes the main results from \cite{BJM}.
The following theorems are \cite[Cor. 5.12 \& Cor. 5.13]{BJM}:

\begin{thm} Let\/ $(\bs X,\om)$ be a $-1$-shifted symplectic
derived scheme over\/ $\K$ in the sense of Pantev et al.\
{\rm\cite{PTVV},} and\/ $X=t_0(\bs X)$ the associated classical\/
$\K$-scheme, assumed of finite type. Suppose we are given a square
root\/ $\smash{\det(\bL_{\bs X})\vert_X^{1/2}}$ for $\det(\bL_{\bs
X})\vert_X$. Then we may define a natural motive $MF_{\bs X,\om}\in
\oM^{\hat\mu}_X$.
\label{mo5cor3}
\end{thm}

\begin{thm} Suppose $Y$ is a Calabi--Yau\/ $3$-fold over\/ $\K,$
and\/ $\cM$ is a finite type moduli\/ $\K$-scheme of simple coherent
sheaves in $\coh(Y),$ or simple complexes of coherent sheaves in
$D^b\coh(Y),$ with obstruction theory\/ $\phi:\cE^\bu\ra\bL_\cM$ as
in Thomas\/ {\rm\cite{Thom}} or Huybrechts and Thomas\/
{\rm\cite{HuTh}}. Suppose we are given a square root\/
$\det(\cE^\bu)^{1/2}$ for $\det(\cE^\bu)$. Then we may define a
natural motive $MF_\cM\in\oM^{\hat\mu}_\cM.$
\label{mo5cor4}
\end{thm}

Kontsevich and Soibelman define a motive over $\cM_{\rm
st}^\al(\tau)$, by associating a formal power series to each (not
necessarily closed) point, and taking its motivic Milnor fibre. The
question of how these formal power series and motivic Milnor fibres
vary in families over the base $\cM_{\rm st}^\al(\tau)$ is not
really addressed in \cite{KoSo1}. Corollary \ref{mo5cor4} answers
this question, showing that Zariski locally in $\cM_{\rm
st}^\al(\tau)$ we can take the formal power series and motivic
Milnor fibres to all come from a regular function $f:U\ra\bA^1$ on a
smooth $\K$-scheme~$U$.
As before, the square root $\det(\cE^\bu)^{1/2}$ required in Corollary
\ref{mo5cor4} corresponds roughly to {\it orientation data\/} in
Kontsevich and Soibelman \cite[\S 5]{KoSo1}, \cite{KoSo2}.

\subsection{Generalization to symplectic derived stacks}

Here we summarizes the main results from \cite{BBBJ}.
The following theorems are \cite[Cor. 2.11 \& Cor. 2.12]{BBBJ}:

\begin{thm} Let\/ $(\bX,\om_\bX)$ be a $-1$-shifted
symplectic derived Artin $\K$-stack, and\/ $X=t_0(\bX)$ the
corresponding classical Artin $\K$-stack. Then for each\/ $p\in X$
there exist a smooth\/ $\K$-scheme $U$ with dimension $\dim
H^0\bigl(\bL_X\vert_p\bigr),$ a point\/ $t\in U,$ a regular function
$f:U\ra\bA^1$ with\/ $\dd f\vert_t=0,$ so that\/
$T:=\Crit(f)\subseteq U$ is a closed\/ $\K$-subscheme with\/ $t\in
T,$ and a morphism $\vp:T\ra X$ which is smooth of relative
dimension $\dim H^1\bigl(\bL_X\vert_p\bigr),$ with\/ $\vp(t)=p$. We
may take\/~$f\vert_{T^\red}=0$.
\label{sa2cor1}
\end{thm}

Thus, the underlying classical stack $X$ of a $-1$-shifted symplectic derived
stack $(\bX,\om_\bX)$ admits an atlas consisting of critical loci of
regular functions on smooth schemes.

\smallskip

Now let $Y$ be a Calabi--Yau 3-fold over $\K$, and $\cM$ a classical
moduli stack of coherent sheaves $F$ on $Y$, or complexes $F^\bu$ in
$D^b\coh(Y)$ with $\Ext^{<0}(F^\bu,F^\bu)=0$. Then
$\cM=t_0(\bs\cM)$, for $\bs\cM$ the corresponding derived moduli
stack. The (open) condition $\Ext^{<0}(F^\bu,F^\bu)=0$ is needed to
make $\bs\cM$ $1$-geometric and $1$-truncated (that is, a derived Artin stack, in our terminology);
without it, $\cM,\bs\cM$ would be a higher derived stack. Pantev et al.\ \cite[\S 2.1]{PTVV}
prove $\bs\cM$ has a $-1$-shifted symplectic structure
$\om_{\bs\cM}$. Applying Theorem \ref{sa2cor1} and using
$H^i\bigl(\bL_{\bs\cM}\vert_{[F]}\bigr)\cong \Ext^{1-i}(F,F)^*$
yields a new result on classical 3-Calabi--Yau moduli stacks,
the statement of which involves no derived geometry:

\begin{cor} Suppose $Y$ is a Calabi--Yau\/ $3$-fold over\/
$\K,$ and\/ $\cM$ a classical moduli $\K$-stack of coherent sheaves
$F,$ or more generally of complexes $F^\bu$ in $D^b\coh(Y)$ with
$\Ext^{<0}(F^\bu,F^\bu)=0$. Then for each\/ $[F]\in\cM,$ there exist
a smooth\/ $\K$-scheme $U$ with\/ $\dim U=\dim\Ext^1(F,F),$ a
point\/ $u\in U,$ a regular function $f:U\ra\bA^1$ with\/ $\dd
f\vert_u=0,$ and a morphism $\vp:\Crit(f)\ra\cM$ which is smooth of
relative dimension $\dim\Hom(F,F),$ with\/~$\vp(u)=[F]$.
\label{sa2cor2}
\end{cor}

This is an analogue of \cite[Cor.~5.19]{BBJ}. When $\K=\C$, a
related result for coherent sheaves only, with $U$ a complex
manifold and $f$ a holomorphic function, was proved by Joyce and
Song \cite[Th.~5.5]{JoSo} using gauge theory and transcendental
complex methods.

\smallskip

Here is \cite[Thm. 3.18]{BBBJ}, a stack version of Theorem \ref{da6thm4}.

\begin{thm} Let\/ $\K$ be an algebraically closed field of
characteristic zero, $(\bX,\om_\bX)$ a $-1$-shifted symplectic
derived Artin $\K$-stack, and\/ $X=t_0(\bX)$ the corresponding
classical Artin $\K$-stack. Then there exists a unique d-critical
structure $s\in H^0(\cSz_X)$ on $X,$ making $(X,s)$ into a
d-critical stack, with the following properties:
\begin{itemize}
\setlength{\itemsep}{0pt}
\setlength{\parsep}{0pt}
\item[{\bf(a)}] Let\/ $U,$ $f:U\ra\bA^1,$ $T=\Crit(f)$ and\/
$\vp:T\ra X$ be as in Corollary\/ {\rm\ref{sa2cor1},} with\/
$f\vert_{T^\red}=0$. There is a unique
$s_T\in H^0(\cSz_T)$ on $T$ with\/
$\io_{T,U}(s_T)=i^{-1}(f)+I_{T,U}^2,$ and\/ $(T,s_T)$ is an
algebraic d-critical locus. Then $s(T,\vp)=s_T$ in
$H^0(\cSz_T)$.
\item[{\bf(b)}] The canonical bundle $K_{X,s}$ of\/ $(X,s)$ from
Theorem\/ {\rm\ref{sa3thm5}} is naturally isomorphic to the
restriction $\det(\bL_\bX)\vert_{X^\red}$ to $X^\red\subseteq
X\subseteq\bX$ of the determinant line bundle $\det(\bL_\bX)$ of
the cotangent complex\/ $\bL_\bX$ of\/~$\bX$.
\end{itemize}
\label{sa3thm6}
\end{thm}

We can think of Theorem \ref{sa3thm6} as defining a {\it truncation
functor}
\begin{align*}
F:\bigl\{&\text{$\iy$-category of $-1$-shifted symplectic derived
Artin $\K$-stacks $(\bX,\om_\bX)$}\bigr\}\\
&\longra\bigl\{\text{2-category of d-critical stacks $(X,s)$ over
$\K$}\bigr\}.
\end{align*}

Let $Y$ be a Calabi--Yau 3-fold over $\K$, and $\cM$ a classical
moduli $\K$-stack of coherent sheaves in $\coh(Y)$, or complexes of
coherent sheaves in $D^b\coh(Y)$. There is a natural obstruction
theory $\phi:\cE^\bu\ra\bL_\cM$ on $\cM$, where $\cE^\bu\in
D_{\qcoh}(\cM)$ is perfect in the interval $[-2,1]$, and
$h^i(\cE^\bu)\vert_F\cong\Ext^{1-i}(F,F)^*$ for each $\K$-point
$F\in\cM$, regarding $F$ as an object in $\coh(Y)$ or $D^b\coh(Y)$.
Now in derived algebraic geometry $\cM=t_0(\bcM)$ for $\bcM$ the
corresponding derived moduli $\K$-stack, and
$\phi:\cE^\bu\ra\bL_\cM$ is $\bL_{t_0}:\bL_{\bcM}
\vert_\cM\ra\bL_\cM$. Pantev et al.\ \cite[\S 2.1]{PTVV} prove
$\bcM$ has a $-1$-shifted symplectic structure $\om$. Thus Theorem
\ref{sa3thm6} implies \cite[Cor. 3.19]{BBBJ}:

\begin{cor} Suppose $Y$ is a Calabi--Yau\/ $3$-fold over\/ $\K$ of
characteristic zero, and\/ $\cM$ a classical moduli\/ $\K$-stack of
coherent sheaves $F$ in $\coh(Y),$ or complexes of coherent sheaves
$F^\bu$ in $D^b\coh(Y)$ with $\Ext^{<0}(F^\bu,F^\bu)=0,$ with
obstruction theory\/ $\phi:\cE^\bu\ra\bL_\cM$. Then $\cM$ extends
naturally to an algebraic d-critical locus $(\cM,s)$. The canonical
bundle $K_{\cM,s}$ from Theorem\/ {\rm\ref{sa3thm5}} is naturally
isomorphic to $\det(\cE^\bu)\vert_{\cM^\red}$.
\label{sa3cor2}
\end{cor}

Here is \cite[Cor. 4.13]{BBBJ}, the stack version of Theorem \ref{sm6cor2}:

\begin{thm} Let\/ $\K$ be an algebraically closed field of
characteristic zero, $(\bX,\om)$ a $-1$-shifted symplectic derived
Artin $\K$-stack, and\/ $X=t_0(\bX)$ the associated classical
Artin\/ $\K$-stack. Suppose we are given a square root\/
$\smash{\det(\bL_\bX)\vert_X^{1/2}}$.
Then working in $l$-adic perverse sheaves on stacks
\cite[\S 4]{BBBJ} we may define a perverse sheaf\/
$\check P_{\bX,\om}^\bu$ on $X$ uniquely up to canonical
isomorphism, and Verdier duality and monodromy isomorphisms
$\check\Si_{\bX,\om}:\check P_{\bX,\om}^\bu\ra \bD_X(\check
P_{\bX,\om}^\bu)$ and\/ $\check\Tau_{\bX,\om}:\check
P_{\bX,\om}^\bu\ra\check P_{\bX,\om}^\bu$. These are characterized
by the fact that given a diagram
\begin{equation*}
\xymatrix@C=60pt{ \bU=\bs\Crit(f:U\ra\bA^1) & \bV \ar[l]_(0.3){\bs i}
\ar[r]^{\bs\vp} & \bX }
\end{equation*}
such that\/ $U$ is a smooth\/ $\K$-scheme, $\bs\vp$ smooth of
dimension $n,$ $\bL_{\bV/\bU} \simeq \bT_{\bV/\bX}[2],$
$\bs\vp^*(\om_\bX)\sim \bs i^*(\om_\bU)$ for $\om_\bU$ the natural\/
$-1$-shifted symplectic structure on $\bU=\bs\Crit(f:U\ra\bA^1),$
and\/ $\vp^*(\det(\bL_\bX)\vert_X^{1/2})\!\cong\!
i^*(K_U) \ot \La^n\bT_{\bV/\bX},$ then $\vp^*(\check
P_{\bX,\om}^\bu)[n],$ $\vp^*(\check\Si_{\bX,\om}^\bu)[n],$
$\vp^*(\check\Tau_{\bX,\om}^\bu)[n]$ are canonically isomorphic to
$i^*(\PV_{U,f}),$ $i^*(\si_{U,f}),$ $i^*(\tau_{U,f}),$ for\/
$\PV_{U,f},\si_{U,f},\tau_{U,f}$ as in \cite{BBBJ} . The same
applies in the other theories of perverse sheaves and\/
$\cD$-modules on stacks.
\label{sa4cor1}
\end{thm}

Here is \cite[Cor. 4.14]{BBBJ}, the stack version of Theorem \ref{sm6cor3}:

\begin{thm} Let\/ $Y$ be a Calabi--Yau\/ $3$-fold over an
algebraically closed field\/ $\K$ of characteristic zero, and\/
$\cM$ a classical moduli\/ $\K$-stack of coherent sheaves $F$ in
$\coh(Y),$ or of complexes $F^\bu$ in $D^b\coh(Y)$ with\/
$\Ext^{<0}(F^\bu,F^\bu)=0,$ with obstruction theory\/
$\phi:\cE^\bu\ra\bL_\cM$. Suppose we are given a square root\/
$\det(\cE^\bu)^{1/2}$.
Then working in $l$-adic perverse sheaves on stacks
\cite[\S 4]{BBBJ}, we may define a natural perverse
sheaf\/ $\check P_\cM^\bu\in\Perv(\cM),$ and Verdier duality and
monodromy isomorphisms $\check\Si_\cM:\check
P_\cM^\bu\ra\bD_\cM(\check P_\cM^\bu)$ and\/ $\check\Tau_\cM:\check
P_\cM^\bu\ra\check P_\cM^\bu$. The pointwise Euler characteristic
of\/ $\check P_\cM^\bu$ is the Behrend function $\nu_\cM$ of\/ $\cM$
from Joyce and Song {\rm\cite[\S 4]{JoSo},} so that\/ $\check
P_\cM^\bu$ is in effect a categorification of the Donaldson--Thomas
theory of $\cM$. The same applies in the other theories of perverse
sheaves and\/ $\cD$-modules on stacks.
\label{sa4cor2}
\end{thm}

Here is \cite[Cor. 5.16]{BBBJ}, the stack version of Theorem \ref{mo5cor3}:

\begin{thm} Let\/ $(\bX,\om)$ be a $-1$-shifted symplectic derived
Artin $\K$-stack in the sense of Pantev et al.\ {\rm\cite{PTVV},}
and\/ $X=t_0(\bX)$ the associated classical Artin\/ $\K$-stack,
assumed of finite type and locally a global quotient. Suppose we are
given a square root\/ $\det(\bL_\bX)\vert_X^{1/2}$ for
$\det(\bL_\bX) \vert_X$. Then we may define a natural motive
$MF_{\bX,\om}\in\oM^\stm_X,$ which is characterized
by the fact that given a diagram
\begin{equation*}
\xymatrix@C=60pt{ \bU=\bs\Crit(f:U\ra\bA^1) & \bV \ar[l]_(0.3){\bs i}
\ar[r]^{\bs\vp} & \bX }
\end{equation*}
such that\/ $U$ is a smooth\/ $\K$-scheme, $\bs\vp$ is smooth of
dimension $n,$ $\bL_{\bV/\bU} \simeq \bT_{\bV/\bX}[2],$
$\bs\vp^*(\om_\bX)\sim \bs i^*(\om_\bU)$ for $\om_\bU$ the natural\/
$-1$-shifted symplectic structure on $\bU=\bs\Crit(f:U\ra\bA^1),$
and\/ $\vp^*(\det(\bL_\bX)\vert_X^{1/2})\cong
i^*(K_U)\ot\La^n\bT_{\bV/\bX},$ then~$\vp^*(MF_{\bX,\om})=\bL^{n/2}\od i^*(MF^{{\rm mot},\phi}_{U,f})$ in $\oM^\stm_V$.
\label{sa5cor1}
\end{thm}

Here is \cite[Cor. 5.17]{BBBJ}, the stack version of Theorem \ref{mo5cor4}:

\begin{thm} Let\/ $Y$ be a Calabi--Yau\/ $3$-fold over\/ $\K,$
and\/ $\cM$ a finite type classical moduli\/ $\K$-stack of coherent
sheaves in $\coh(Y),$ with natural obstruction theory\/
$\phi:\cE^\bu\ra\bL_\cM$. Suppose we are given a square root\/
$\det(\cE^\bu)^{1/2}$ for $\det(\cE^\bu)$. Then we may define a
natural motive $MF_\cM\in\oM^\stm_\cM$.
\label{sa5cor2}
\end{thm}

Theorem \ref{sa5cor2} is relevant to Kontsevich and Soibelman's
theory of {\it motivic Donaldson--Thomas invariants\/} \cite{KoSo1}.
Again, our square root $\det(\cE^\bu)^{1/2}$ roughly coincides with their
{\it orientation data\/} \cite[\S 5]{KoSo1}. In \cite[\S
6.2]{KoSo1}, given a finite type moduli stack $\cM$ of coherent
sheaves on a Calabi--Yau 3-fold $Y$ with orientation data, they
define a motive $\int_\cM 1$ in a ring $D^\mu$ isomorphic to our
$\oM^\stm_\K$. We expect this should agree with $\pi_*(MF_\cM)$ in
our notation, with $\pi:\cM\ra\Spec\K$ the projection. This
$\int_\cM 1$ is roughly the motivic Donaldson--Thomas invariant of
$\cM$. Their construction involves expressing $\cM$ near each point
in terms of the critical locus of a formal power series. Kontsevich
and Soibelman's constructions were partly conjectural, and our
results may fill some gaps in their theory.

\section{The main results} 
\markboth{Statements of main results}{Statements of main results}
\label{main.1}

We will prove and use the algebraic analogue of Theorem \ref{dt5thm1}, 
which we can state as follows:

\begin{thm} 
Let\/ $X$ be a Calabi--Yau $3$-fold over $\K,$ and write $\fM$ for the moduli stack of coherent sheaves on $X$. Then for each\/ $[E]\in\fM(\K),$ there exists a smooth\/ affine $\K$-scheme $U,$ a point\/ $p\in\U(\K),$ an \'etale morphism $u:U\ra\Ext^1(E,E)$ with $u(p)=0,$ a regular function $f:U\ra\bA^1$ with\/ $f\vert_p=\pd f\vert_p=0,$ and a $1$-morphism $\xi:\Crit(f)\ra\fM$ smooth of relative dimension $\dim\Aut(E),$ with $\xi(p)=[E]\in\fM(\K),$ such that if $\io:\Ext^1(E,E)\ra T_{[E]}\fM$ is the natural isomorphism, then $\d\xi\vert_p=\io\ci\d u\vert_p:T_pU\ra T_{[E]}\fM$.
  Moreover, let $G$ be a maximal algebraic torus in $\Aut(E)$, acting on $\Ext^1(E,E)$ by $\ga:\ep\mapsto\ga\ci\ep\ci\ga^{-1}$. Then we can choose $U,p,u,f,\xi$ and a $G$-action on $U$ such that $u$ is $G$-equivariant and $p,f$ are $G$-invariant, so that $\Crit(f)$ is $G$-invariant, and\/ $\xi:\Crit(f) \ra\fM$ factors through the projection $\Crit(f)\ra[\Crit(f)/G]$.
\label{dt5thm2}
\end{thm}
\index{reductive group!maximal}
\index{almost closed $1$-form}
\index{Zariski topology}
\nomenclature[1oi]{$\xi$}{the $1$-morphism $\xi:\om^{-1}(0)\ra\fM$}\index{coarse moduli scheme}\index{moduli
scheme!coarse} 

Note that you can regard $u:U\ra\Ext^1(E,E)$ as an \'etale open neighbourhood of $0$ in $\Ext^1(E,E).$
Theorem \ref{dt5thm2} will be proved in \S\ref{localdes}, using \S\ref{dcr}.
Next, we will use this to prove the algebraic analogue of Theorem \ref{dt6thm1}:

\begin{thm} Let\/ $X$ be a Calabi--Yau $3$-fold over $\K,$ and
$\fM$ the moduli stack of coherent sheaves on\/ $X$. The
Behrend function $\nu_{\fM}:
\fM(\K)\ra\Z$ is a natural locally constructible function on $\fM$.
For all\/ $E_1,E_2\in\coh(X),$ it satisfies:
\bigskip
\begin{equation}
\nu_{\fM}(E_1\op E_2)=(-1)^{\bar\chi([E_1],[E_2])}
\nu_{\fM}(E_1)\nu_{\fM}(E_2), 
\label{dt6eq1.1}
\end{equation}
\smallskip
\begin{equation}
\displaystyle \int\limits_{\small{\begin{subarray}{l}  [\la]\in\mathbb{P}(\Ext^1(E_2,E_1)):\\
 \la\; \Leftrightarrow\; 0\ra E_1\ra F\ra E_2\ra
0\end{subarray}}}\!\!\!\!\! \!\!\!\! \!\!\!\! \!\!\!\! \nu_{\fM}(F)\,\rd\chi \quad - \!\!\!\! \!\!\!\! \!\!\!\! 
\displaystyle \int\limits_{\small{\begin{subarray}{l}[\mu]\in\mathbb{P}(\Ext^1(E_1,E_2)):\\
\mu\; \Leftrightarrow\; 0\ra E_2\ra D\ra E_1\ra
0\end{subarray}}}\!\!\!\!\! \!\!\!\! \!\!\!\! \!\!\!\! \nu_{\fM}(D)\,\rd\chi \;\;
= \;\; (e_{21}-e_{12})\;\;
\nu_{\fM}(E_1\op E_2),
\label{dt6eq2.1}
\end{equation}
\bigskip

where $e_{21}=\dim\Ext^1(E_2,E_1)$ and $e_{12}=\dim\Ext^1(E_1,E_2)$ for $E_1,E_2\in\coh(X).$ 
Here\/ $\bar\chi([E_1],[E_2])$ in \eq{dt6eq1.1} is the Euler form as in \eq{eu}, \index{Euler form}
and in \eq{dt6eq2.1} the correspondence between\/
$[\la]\in\mathbb{P}(\Ext^1(E_2,E_1))$ and\/ $F\in\coh(X)$ is that\/
$[\la]\in\mathbb{P}(\Ext^1(E_2,E_1))$ lifts to some\/
$0\ne\la\in\Ext^1(E_2,E_1),$ which corresponds to a short exact
sequence\/ $0\ra E_1\ra F\ra E_2\ra 0$ in\/ $\coh(X)$ in the usual
way. The function $[\la]\mapsto\nu_{\fM}(F)$ is a constructible
function\/ $\mathbb{P}(\Ext^1(E_2,E_1))\ra\Z,$ and the integrals in
\eq{dt6eq2.1} are integrals of constructible functions using the Euler
characteristic as measure. \index{constructible function} \index{Euler characteristic}
\label{dt6thm1.1}
\end{thm}

As in \S\ref{dt4}, the identities \eq{dt6eq1.1}--\eq{dt6eq2.1} are crucial for the whole program in \cite{JoSo}, and will be proved in \S\ref{dt.1}.

\medskip

In the next theorem, the condition that $\Ext^{<0}(E^\bu,E^\bu)=0$ is necessary for $\ti\fM$ to be an Artin stack, rather than a higher stack. Note that this condition is automatically satisfied by complexes $E^\bu$ which are semistable in any stability condition, for example Bridgeland stability conditions \cite{Brid1}. Therefore to prove wall-crossing formulae for Donaldson-Thomas invariants in the derived category $D^b\coh(X)$ under change of stability condition by the ``dominant stability condition'' method of \cite{Joyc.4,Joyc.5,Joyc.6,Joyc.7,KaSc}, it is enough to know the Behrend function identities \eq{dt6eq1.1}--\eq{dt6eq2.1} for complexes $E^\bu$ with $\Ext^{<0}(E^\bu,E^\bu)=0$, and we do not need to deal with complexes $E^\bu$ with  $\Ext^{<0}(E^\bu,E^\bu)\ne 0$, or with higher stacks.

\begin{thm} 
Let\/ $X$ be a Calabi--Yau 3-fold over $\K,$ and write $\ti\fM$ for the moduli stack of complexes $E^\bu$ in $D^b\coh(X)$ with $\Ext^{<0}(E^\bu,E^\bu)=0$. This is an Artin stack by \cite{HuTh}.
Let\/ $[E^\bu]\in\ti\fM(\K),$ and suppose that a Zariski open neighbourhood of $[E^\bu]$ in $\ti\fM(\K)$ is equivalent to a global quotient $[S/\GL(n,\K)]$ for $S$ a $\K$-scheme with a $\GL(n,\K)$-action.
Then the analogues of Theorems \ref{dt5thm2} and \ref{dt6thm1.1}
 hold with $\ti\fM,E^\bu$ in place of\/ $E,\fM$.
\label{dt6thm1.1.bis}
\end{thm}

The condition on $\ti\fM$ that it should be {\it locally a global quotient}, is known for the moduli stack of coherent sheaves $\fM$ using Quot schemes.
A proof of that can be found in \cite[\S 9.3]{JoSo},
where 
Joyce and Song uses 
the standard method for constructing coarse
moduli schemes\index{coarse moduli scheme}\index{moduli
scheme!coarse} of semistable coherent sheaves in Huybrechts and Lehn
\cite{HuLe2}, adapting it for Artin stacks, and an argument similar to parts of that of
Luna's Etale Slice Theorem~\cite[\S III]{Luna}.\index{Luna's Etale Slice
Theorem} 
However, this is not known for the moduli stack of complexes. The author expects Theorem \ref{dt6thm1.1.bis} to hold without this technical assumption, but currently can't prove it.

\smallskip

The proof of Theorem \ref{dt6thm1.1.bis} is the same as the proof of Theorem \ref{dt6thm1.1},
substituting sheaves with complexes of sheaves, and accordingly making the obvious modifications.

\medskip

Finally, in \S\ref{def} we will characterize the numerical Grothendieck group of a Calabi--Yau 3-fold
in terms of a deformation invariant lattice described using the Picard group.
First of all, using existence results, and smoothness and properness 
properties of the relative Picard scheme in a family of Calabi--Yau 3-folds, 
one proves that the Picard groups form a local system. 
Actually, it is a local system with finite monodromy, so it can be made trivial after passing to a finite \'etale cover of the base scheme, as formulated in the analogue of 
\cite[Thm. 4.21]{JoSo}, which studies the monodromy of the Picard scheme 
instead of the numerical Grothendieck group in a family.
Then, Theorem \ref{definv}, a substitute for \cite[Thm. 4.19]{JoSo}, which does not need 
the integral Hodge conjecture result by Voisin \cite{Vois} for Calabi--Yau 3-folds over $\C$
and which is valid over $\K,$ characterizes the numerical Grothendieck 
group of a Calabi--Yau 3-fold in terms of a globally constant lattice described using the Picard scheme:

\begin{thm} Let\/ $X$ be a Calabi--Yau $3$-fold over $\K$ with\/
$H^1(\cO_X)\!=\!0$. Define\nomenclature[1l]{$\La_X$}{lattice associated to a Calabi--Yau 3-fold $X$}
\begin{equation*}
\La_X=\ts\bigl\{ (\la_0,\la_1,\la_2,\la_3) 
\textrm{ where } \la_0,\la_3\in\Q, \; \la_1\in\Pic(X)\ot_{\Z}\Q, \; \la_2\in \Hom(\Pic(X),\Q)
\textrm{ such that } 
\end{equation*}
\begin{equation*}
\la_0\in \Z,\;\>
\la_1\in \Pic(X)/ {\textrm{torsion}}, \;
 \la_2-\ha\la_1^2\in  \Hom(\Pic(X),\Z),\;\> \la_3+\ts\frac{1}{12}\la_1
c_2(TX)\in \Z\bigr\},
\end{equation*}
\smallskip
where $\la_1^2$ is defined as the map $\al\in\Pic(X)\ra \frac{1}{2}c_1(\la_1)\cdot c_1(\la_1)\cdot c_1(\al)\in A^3(X)_{\Q}\cong \Q,$
and $\frac{1}{12}\la_1 c_2(TX)$ is defined as $\frac{1}{12}c_1(\la_1)\cdot c_2(TX)\in A^3(X)_{\Q}\cong\Q.$
Then\/ 
for any family of Calabi-Yau 3-folds $\pi : {\cal X} \ra S$ over a connected base $S$ with $X=\pi^{-1}(s_0),$ the lattices $\La_{X_s}$ form a local system of abelian groups over $S$ with fibre $\La_X$. Furthermore, the monodromy of this system lies in a finite subgroup of $\Aut(\La_X)$, so after passing to an \'etale cover $\ti S\ra S$ of S, we can take the local system to be trivial, and coherently identify $\La_{X_{\ti s}}\cong \La_X$ for all $\ti s\in \ti S$.
Finally, the Chern character gives an injective morphism
$\ch:K^\num(\coh(X))\!\hookra\!\La_X$.
\label{definv}
\end{thm}

Following \cite{JoSo}, this yields

\begin{thm}
The generalized Donaldson--Thomas invariants $\bar{DT}{}^\al(\tau)$ over $\K$ for $\al\in\Lambda_X$
are unchanged under deformations of the underlying Calabi--Yau 3-fold X, 
by which we mean the following: let $\fX\stackrel{\vp}{\longra}T$ a smooth projective morphism of
algebraic $\K$-varieties $X,T$, with $T$ connected. 
Let $\cO_\fX(1)$
be a relative very ample line bundle for
$\fX\stackrel{\vp}{\longra}T$. For each $t\in T(\K)$, write $X_t$ for
the fibre $X\times_{\vp,T,t}\Spec\K$ of $\vp$ over $t$, and
$\cO_{X_t}(1)$ for $\cO_\fX(1)\vert_{X_t}$. Suppose that $X_t$ is a
smooth Calabi--Yau 3-fold over $\K$ for all $t\in T(\K)$, with $H^1(\cO_{X_t})= 0$. 
Then the generalized Donaldson--Thomas invariants
$\bar{DT}{}^\al(\tau)_t$ are independent of $t\in T(\K).$
\label{defthm}
\end{thm}

\smallskip

More precisely, the isomorphism $\La_{X_t} = \La_X$ is canonical up to action of a finite group $\Gamma,$ the monodromy on $T,$ and $DT^\al(\tau)_t$ are independent of the action of $\Gamma$ on $\alpha,$ so whichever identification $\La_{X_t}=\La_X$ is chosen, it is still true $DT^\al(\tau)_t$ independent of $t.$

\medskip

Now, recall that in \cite{JoSo} Joyce and Song used the assumption that the base field is the field of complex numbers $\K=\C$
for the Calabi--Yau 3-fold $X$ in three main ways:\index{field $\K$|(} \index{field $\K$!algebraically closed}
\begin{itemize}\index{gauge theory}\index{Behrend function!Behrend identities}
\setlength{\itemsep}{0pt}
\setlength{\parsep}{0pt}
\item[\bf(a)] Theorem \ref{dt5thm1} in
\S\ref{dt4} is proved using gauge theory and transcendental
complex analytic methods, and work only over $\K=\C$. It is
used to prove the Behrend function identities 
\eq{dt6eq1}--\eq{dt6eq2}, which are vital for
much of their results, including the wall crossing
formula for the $\bar{DT}{}^\al(\tau)$, and the
relation between
$PI^{\al,n}(\tau'),\bar{DT}{}^\al(\tau)$. 
\item[\bf(b)] In \cite[\S 4.5]{JoSo}, when $\K=\C$ the Chern
character\index{Chern character} embeds $K^\num(\coh(X))$ in $H^{\rm
even}(X;\Q)$, and they use this to show $K^\num(\coh(X))$ is
unchanged under deformations of $X$. This is important for the
results that $\bar{DT}{}^\al(\tau)$ and
$PI^{\al,n}(\tau')$ for $\al\in K^\num(\coh(X))$ are invariant
under deformations of $X$ even to make sense.
\item[\bf(c)] Their notion of `compactly embeddable' noncompact Calabi-Yau 3-folds  \index{compactly embeddable} in \cite[\S 6.7]{JoSo} is complex
analytic and does not make sense for general~$\K$. This constrains the noncompact Calabi--Yau 3-folds they can define 
generalized Donaldson--Thomas invariants for.
\end{itemize}

Now Theorem \ref{dt5thm2} and Theorem \ref{dt6thm1.1} extend the results in {\bf(a)} 
over algebraically closed field $\K$
of characteristic zero.
As noted in \cite{Joyc.1}, constructible
functions\index{constructible function!in positive characteristic}
methods fail for $\K$ of positive characteristic.\index{field
$\K$!positive characteristic} Because of this, the alternative
descriptions \eq{dt3eq5} and \eq{dt4eq11},  for
$DT^\al(\tau)$ and $PI^{\al,n}(\tau')$ as weighted Euler characteristics,
and the definition of $\bar{DT}{}^\al(\tau)$ in \S\ref{dt4}, cannot work in positive characteristic, so working over an algebraically closed field of characteristic zero is about as general as is reasonable. 

\smallskip

The point {\bf(a)} above has consequences also on {\bf(c)}, because Joyce and Song only need the notion of `compactly embeddable'
as their complex analytic proof of \eq{dt6eq1}--\eq{dt6eq2} requires
$X$ compact. Unfortunately 
the given algebraic version of \eq{dt6eq1}--\eq{dt6eq2} in Theorem \ref{dt6thm1.1} 
uses results from derived algebraic geometry, and the author does not know if
they apply also for compactly supported sheaves 
\index{coherent sheaf!compactly supported} on a
noncompact~$X$.\index{field $\K$|)}
\index{Calabi--Yau 3-fold!noncompact}
We can prove a version of that under some technical assumptions, as stated in 
\S\ref{dt7}.
Observe, also, that in the noncompact case you cannot expect to have the deformation invariance property 
unless in some particular cases in which the moduli space is proper.
The extension of {\bf(b)} to $\K$ is given in Section \ref{def}, which yields Theorem \ref{defthm}, thanks to which
it is possible to extend \cite[Cor. 5.28]{JoSo} about the deformation invariance of the generalized Donaldson--Thomas 
invariants in the compact case to algebraically closed fields $\K$ of characteristic zero.
Thus, this proves our main theorem:

\begin{thm}The theory of generalized Donaldson--Thomas invariants defined in \cite{JoSo} is valid 
over algebraically closed fields of characteristic zero.
\label{mainthm}
\end{thm}

Next, we will respectively prove Theorems \ref{dt5thm2}, \ref{dt6thm1.1} and \ref{definv}
in \S\ref{localdes}, \S\ref{dt.1} and \S\ref{def}.

\subsection{Local description of the Donaldson--Thomas moduli space}
\label{localdes}

Let us fix a moduli stack $\fM$ which is locally a global quotient. 
In particular, $\fM$ can be the moduli stack of coherent sheaves
over a Calabi-Yau 3-fold $X$, so that
the theory exposed in \S\ref{dcr} and \S\ref{ourpapers} applies.

\smallskip

The first step in order to proving Theorem \ref{dt5thm2} is 
to show the existence of a quasiprojective $\K$-scheme $S,$ an action of\/
$G$ on $S,$ a point\/ $s\in S(\K)$ fixed by $G,$
and a $1$-morphism of Artin $\K$-stacks $\xi:[S/G]\ra\fM,$
which is smooth of relative dimension $\dim\Aut(E)-\dim G,$ 
\index{Artin stack}\index{quotient stack}\index{stabilizer group}
where $[S/G]$ is the quotient stack, such that\/
$\xi(s\,G)=[E],$ the induced morphism on stabilizer groups
$\xi_*:\Iso_{[S/G]}(s\,G)\ra\Iso_{\fM}([E])$ is
the natural morphism $G\hookra\Aut(E)\cong\Iso_{\fM}([E]),$
and\/ $\rd\xi\vert_{s\,G}:T_sS\cong T_{s\,G}
[S/G]\ra T_{[E]}\fM\cong \Ext^1(E,E)$ is an isomorphism.

\smallskip

As $\fM$ is locally a global quotient,
let's say
$\fM$ is locally $[Q/H]$ with $H=\GL(n,\K),$ and a $\K$-scheme $Q$ 
which is $H$-invariant, so that 
the projection $[Q/H]\ra\fM$ is a 1-isomorphism with an open $\K$-substack
${\mathfrak Q}$ of $\fM$. This $1$-isomorphism
identifies the stabilizer groups $\Iso_{\fM}([E])=\Aut(E)$ and\/
$\Iso_{[Q/H]}(sH)=\Stab_H(s),$ and the Zariski tangent spaces
$T_{[E]}\fM\cong\Ext^1(E,E)$ and\/ $T_{sH}[Q/H]\cong T_sQ/T_s(sH),$
so one has natural isomorphisms $\Aut(E)\cong\Stab_H(s)$
and\/~$\Ext^1(E,E)\cong T_sQ/T_s(sH)$, and $G$ is identified as a subgroup of $H.$

\smallskip

To obtain the 1-morphism with the required properties, following \cite[\S 9.3]{JoSo} 
and Luna's Etale Slice Theorem~\cite[\S III]{Luna}, we obtain an atlas 
$S$ as a $G$-invariant, locally closed\/ $\K$-subscheme in $Q$ with\/
$s\in S(\K),$ such that\/ $T_sQ=T_sS\op T_s(sH),$ 
and the morphism $\mu: S\times H\ra Q$ induced by the inclusion $S\hookra Q$ and the $H$-action on $Q$ is smooth of relative dimension $\dim\Aut(E).$
Here $s\in Q(\K)$ project to the point\/ $sH$ in ${\mathfrak
Q}(\K)$ identified with\/ $[E]\in\fM(\K)$ under the
$1$-isomorphism\/ ${\mathfrak Q}\cong [Q/H]$ and $G$, a $\K$-subgroup of the $\K$-group H, is as in the statement 
of Theorem \ref{dt5thm2}, that is, a maximal torus in $\Aut(E).$ 
Since $S$ is invariant under the $\K$-subgroup $G$ of the $\K$-group
$H$ acting on $Q$, the inclusion $i:S\hookra Q$ induces a
representable 1-morphism of quotient stacks $i_*:[S/G]\ra [Q/H]$. In \cite{JoSo},
Joyce and Song found that $i_*$ is smooth of relative dimension $\dim\Aut(E)-\dim
G$. 
Combining the 1-morphism
$i_*:[S/G]\ra [Q/H]$, the 1-isomorphism ${\mathfrak Q}\cong [Q/H]$,
and the open inclusion ${\mathfrak Q}\hookra\fM$, yields a
1-morphism $\xi:[S/G]\ra\fM$, as required for Theorem \ref{dt5thm2}.
This $\xi$ is smooth of relative dimension $\dim\Aut(E)-\dim
G$, as $i_*$ is. If $\Aut(E)$ is reductive, so that
$G=\Aut(E)$, then $\xi$ is smooth of dimension 0, that is,
$\xi$ is \'etale.\index{etale morphism@\'etale morphism}
The conditions that $\xi(s\,G)=[E]$ and that
$\xi_*:\Iso_{[S/G]}(s\,G)\ra\Iso_{\fM}([E])$ is
the natural $G\hookra\Aut(E) \cong\Iso_{\fM}([E])$ in
Theorem \ref{dt5thm1} are immediate from the construction. That
$\rd\xi\vert_{s\,G}:T_sS\cong T_{s\,G}
[S/G]\ra T_{[E]}\fM\cong \Ext^1(E,E)$ is an isomorphism
follows from $T_{[E]}\fM\cong T_{sH}[Q/H]\cong T_sQ/T_s(sH)$ and
$T_sQ=T_sS\op T_s(sH)$.\index{Artin
stack!atlas|)}\index{moduli stack!atlas|)} 

\smallskip

In conclusion, we can summarize as follows: 
given a point $[E]\in \fM(\K)$, that is an equivalence class of a (complex of)
coherent sheaves, we will denote by $G$ a maximal torus in $\Aut(E).$ 
As $\fM$ is locally a global quotient,
there exists an atlas $S$, which is a scheme over $\K$, and a smooth morphism $\pi : S \ra \fM$, with $\pi$ smooth of relative dimension $\dim G.$ 
If $x \in S$ is the point corresponding to $E\in \fM(\K),$ then $\pi$ smooth of $\dim G$ means that $\pi$ has {\it minimal dimension} near $E,$ that is $T_x S = \Ext^1(E,E).$ 
Moreover, the atlas $S$ is endowed with a $G$-action, so that $\pi$ descends to a morphism $[S/G] \ra \fM.$ 

\smallskip

Note next that the maximal torus $G$ acts on $S$ preserving $s$ and fixing $x.$ By replacing $S$ by a $G$-equivariant \'etale open neighbourhood $S'$ of $s,$ we can suppose $S$ is affine.
Then, from material in \S\ref{dcr} and \S\ref{ourpapers} we deduce that 
the atlas $S'$
in the sense of Theorems \ref{sa2cor1} and \ref{sa3thm6} for the moduli stack $\fM$
carries a d-critical locus structure $(S',s_{S'})$ which is $\GL(n,\K)$-equivariant in the sense of \S\ref{eqdcr}.

\smallskip

Using Proposition \ref{dc2prop14}, there exists a $G$-invariant critical chart $(R,U,f,i)$ in the sense of \S\ref{dcr} for $(S,s)$ with $x$ in $R,$ and $\dim U$ to be minimal so that 
$T_{i(x)}U = T_x R = \Ext^1(E,E)$. 

\smallskip

Making $U$ smaller if necessary, we can choose $G$-equivariant \'etale coordinates $U \ra \bA^n = \Ext^1(E,E)$ near $i(x),$ sending $i(x)$ to $0,$ and with $T_{i(x)}U = \Ext^1(E,E)$ the given identification. Then we can regard $U \ra \Ext^1(E,E)$ as a $G$-equivariant \'etale open neighbourhood of $0$ in $\Ext^1(E,E),$ which concludes the proof of Theorem \ref{dt5thm2}.

\subsection{Behrend function identities}
\label{dt.1}

Now we are ready to prove Theorem \ref{dt6thm1.1}.
Let $X$ be a Calabi--Yau $3$-fold over an algebraically closed field $\K$ of 
characteristic zero, $\fM$ the moduli stack of
coherent sheaves on $X$, and $E_1,E_2$ be coherent sheaves on $X$.
Set $E=E_1\op E_2$. Using the
splitting 
\e
\Ext^1(E,E)\!=\!\Ext^1(E_1,E_1)\!\op\!\Ext^1(E_2,E_2)\!\op\!\Ext^1(E_1,E_2)
\!\op\!\Ext^1(E_2,E_1),
\label{dt6eq3}
\e
write elements of $\Ext^1(E,E)$ as
$(\ep_{11},\ep_{22},\ep_{12},\ep_{21})$ with
$\ep_{ij}\in\Ext^1(E_i,E_j)$. For simplicity, we will write $e_{ij}=\dim\Ext^1(E_i,E_j).$
Choose a maximal torus \index{
reductive group!maximal} $G$ of $\Aut(E)$ which contains the
subgroup 
 $T=\bigl\{\id_{E_1}+\la\id_{E_2}:
\la\in\bG_m\bigr\}$, which acts on $\Ext^1(E,E)$ by
\e
\la:(\ep_{11},\ep_{22},\ep_{12},\ep_{21})
\mapsto(\ep_{11},\ep_{22},\la^{-1}\ep_{12},\la\ep_{21}).
\label{dt6eq4}
\e

Apply Theorem \ref{dt5thm2} with these $E$ and $G$. This gives an
\'etale morphism $u:U\ra \Ext^1(E,E)$ with $U$ a smooth affine $\K$-scheme, and $u(p)=0,$
for $p\in U(\K),$
a $G$-invariant regular function $f:U\ra \bA^1_\K$ on $U$ with
$f\vert_p=\partial f\vert_p=0,$ an open
neighbourhood\/ $V$ of\/ $s$ in $S,$ and a $1$-morphism $\xi:\Crit(f) \ra\fM$ smooth of relative
dimension $\dim\Aut(E),$ with\/ $\xi(p)=[E]\in\fM(\K)$ and\/
$\rd\xi\vert_p:T_p(\Crit(f))=\Ext^1(E,E)\ra T_{[E]}\fM$ the
natural isomorphism.
Then the Behrend function $\nu_\fM$ at $[E]=[E_1\op E_2]$ satisfies \index{moduli stack!local structure}
\e
\nu_{\fM}(E_1\op E_2)=(-1)^{\dim\Aut(E)}
\nu_{\Crit(f)}(0),
\label{dt6eq5}
\e
where one uses that $\xi$ is
smooth of relative dimension $\dim\Aut(E),$ and Theorem \ref{dt3thm3}
to say that $$\nu_{\Crit(f)}=(-1)^{\dim(\Aut(E))}\xi^*(\nu_{\fM}).$$

On the other hand, the last part of the proof of \eq{dt6eq1.1} in \cite[Section 10.1]{JoSo} uses algebraic methods and gives
\e
\nu_\fM(E_1)\nu_\fM(E_2)=\nu_{\fM\times\fM}(E_1,E_2)=(-1)^{\dim\Aut(E_1)+\dim\Aut(E_2)}\nu_{Crit(f^G)}(0),
\label{dt6eq6}
\e
where $\nu_{Crit(f^G)}(0)=\nu_{\Crit(f)^G}(0)=\nu_{\Crit(f\vert_{U\cap\Ext^1(E,E)^G})}(0)$ and $U$ is as in Theorem \ref{dt5thm2} 
and $\Ext^1(E,E)^G$ denotes the fixed point locus of $\Ext^1(E,E)$ for the $G$-action.
Thus what actually remains to prove in order to establish identity \eq{dt6eq1.1} is 
\e
\nu_{\Crit(f)}(0)=(-1)^{\dim\Ext^1(E_1,E_2)+\dim\Ext^1(E_2,E_1)}\nu_{Crit(f^G)}(0).
\label{dt6eq7}
\e

This is a generalization of a result in \cite{BeFa2} over $\C$ in the case of an isolated $\C^*$-fixed point.
Combining equations \eq{dt6eq5}, \eq{dt6eq6} and \eq{dt6eq7} and sorting out the signs as in \cite[Section 10.1]{JoSo} proves equation \eq{dt6eq1.1}. Equation \eq{dt6eq7} will be crucial also for the proof of the second Behrend identity \eq{dt6eq2.1}.

\medskip

Let us start by recalling an easy result similar to \cite[Prop. 10.1]{JoSo},
but now in the \'etale topology.
Let $u:U \ra \Ext^1(E,E)$ be the \'etale map as in \S\ref{localdes}, and $p\in U$ such that $u(p)=0.$
We will consider points $(0,0,\ep_{12},0),(0,0,0,\ep_{21})\in \Ext^1(E,E)$
like basically points in $U$. This is because we consider a unique lift $\al(e_{12})$
of $(0,0,\ep_{12},0)\in \Ext^1(E,E)$ to $U$, such that $u(\al(e_{12}))=(0,0,e_{12},0)$
and $\lim_{\la \ra 0} \la . \al(e_{12}) =p,$ using that
$\lim_{\la\ra 0} (0,0,\la^{-1}\ep_{12},0) = (0,0,0,0).$
So we can state the following result, for the proof of which we 
cite \cite[Prop.10.1]{JoSo}, with appropriate obvious modifications, working in the \'etale topology.

\begin{prop} Let\/ $\ep_{12}\in\Ext^1(E_1,E_2)$ and\/
$\ep_{21}\in\Ext^1(E_2,E_1)$. Then
\begin{itemize}
\setlength{\itemsep}{0pt}
\setlength{\parsep}{0pt}
\item[{\bf(i)}] $(0,0,\ep_{12},0),(0,0,0,\ep_{21})\in\Crit(f)
\subseteq U\subseteq\Ext^1(E,E),$ and\/ $(0,0,\ep_{12},0),\ab
(0,\ab 0,\ab 0,\ab\ep_{21})\in V\subseteq
S(\K)\subseteq\Ext^1(E,E);$
\item[{\bf(ii)}] $\xi$ maps $(0,0,\ep_{12},0) \mapsto
(0,0,\ep_{12},0)$ and\/
$(0,0,0,\ep_{21})\mapsto(0,0,0,\ep_{21});$ and
\item[{\bf(iii)}] the
induced morphism on closed points $[S/\Aut(E)] (\K)\ra\fM(\K)$ maps $[(0,\ab 0,\ab
0,\ab\ep_{21})]\mapsto[F]$ and\/ $[(0,0,\ep_{12},0)]\mapsto
[F'],$ where the exact sequences $0\ra E_1\ra F\ra E_2\ra 0$
and\/ $0\ra E_2\ra F'\ra E_1\ra 0$ in $\coh(X)$ correspond to
$\ep_{21}\in\Ext^1(E_2,E_1)$ and\/ $\ep_{12}\in\Ext^1(E_1,E_2),$
respectively.
\end{itemize}
\label{dt10prop2}
\end{prop}

Now use the idea in \cite[\S10.2]{JoSo}. 
Set $U'=\bigl\{(\ep_{11},\ep_{22},\ep_{12},\ep_{21})\in
U:\ep_{21}\ne 0\bigr\}$, an open set in $U$, and write $V'$ for the
submanifold of $(\ep_{11},\ep_{22},\ep_{12},\ep_{21})\in U'$ with
$\ep_{12}=0$. Let $\ti U'$ be the blowup of $U'$ along $V'$, with
projection $\pi':\ti U'\ra U'$. Points of $\ti U'$ may be written
$(\ep_{11},\ep_{22},[\ep_{12}],\la\ep_{12},\ep_{21})$, where
$[\ep_{12}]\in\mathbb{P} (\Ext^1(E_1,E_2))$, and $\la\in\K$, and
$\ep_{21}\ne 0$. Write $f'=f\vert_{U'}$ and $\ti f'=f'\ci\pi'$. Then
applying Theorem \ref{blowup} to $U',V',f',\ti U',\pi',\ti f'$ at
the point $(0,0,0,\ep_{21})\in U'$ gives
\e
\begin{split}
\nu_{\Crit(f)}(0,0,0,\ep_{21}) \quad=\displaystyle\int\limits_{\small{[\ep_{12}]
\in\mathbb{P}(\Ext^1(E_1,E_2))}}\!\!\!\!\!\!\nu_{\Crit(\ti f')}(0,0,[\ep_{12}],0,\ep_{21})\,\d\chi +  (-1)^{e_{12}}\bigl(1- e_{12}\bigr)
\nu_{\Crit(f\vert_{V'})}(0,0,0,\ep_{21}).
\label{dt6eq26}
\end{split}
\e
Here $\nu_{\Crit(f)}(0,0,0,\ep_{21})$ is independent of the choice of
$\ep_{21}$ representing the point $[\ep_{21}]\in\mathbb{P}
(\Ext^1(E_2,E_1))$, and is a constructible function\index{constructible
function} of $[\ep_{21}]$, so the integrals in \eq{dt6eq26} are
well-defined. Note that $\nu_{\Crit(f)}$ and the other Behrend functions in the sequel are nonzero just on the zero loci of the corresponding functions, so here and in the sequel the integrals over the whole $\mathbb{P}(\Ext^1(\ldots))$ actually are just over the points that lie in these zero loci. Adopt this convention for the whole section.

\smallskip

Similarly consider the analogous situation exchanging the role of $\ep_{12}$ and $\ep_{21}.$
Set $U''=\bigl\{(\ep_{11},\ep_{22},\ep_{12},\ep_{21})\in
U :\ep_{12}\ne 0\bigr\}$, an open set in $U$, and write $V'' =\bigl\{(\ep_{11},\ep_{22},\ep_{12},\ep_{21})\in U'' : \ep_{21}=0\bigr\}$. 
Let $\ti U''$ be the blowup of $U''$ along $V''$, with
projection $\pi'':\ti U''\ra U''$. Points of $\ti U''$ may be written
$(\ep_{11},\ep_{22},\ep_{12},[\ep_{21}],\la\ep_{21})$, where
$[\ep_{21}]\in\mathbb{P} (\Ext^1(E_2,E_1))$, and $\la\in\K$, and
$\ep_{12}\ne 0$. Write $f''=f\vert_{U''}$ and $\ti f''=f''\ci\pi''$. 
Similarly to the previous situation, we can apply
Theorem \ref{blowup} to $U'',V'',f'',\ti U'',\pi'',\ti f''$ at
the point $(0,0,\ep_{12},0)\in U''$ which gives
\e
\begin{split}
\nu_{\Crit(f)}(0,0,\ep_{12},0)\quad =\displaystyle\int\limits_{\small{[\ep_{21}]
\in\mathbb{P}(\Ext^1(E_2,E_1))}}\!\!\!\!\!\!\nu_{\Crit(\ti f'')}(0,0,\ep_{12},0,[\ep_{21}])\,\d\chi + (-1)^{e_{21}}\bigl(1- e_{21}\bigr)
\nu_{\Crit(f\vert_{V''})}(0,0,\ep_{12},0).
\label{dt6eq27}
\end{split}
\e
Let $L_{12}\ra\mathbb{P}(\Ext^1(E_1,E_2))$ and $L_{21}\ra\mathbb{P}
(\Ext^1(E_2,E_1))$ be the tautological line bundles, so that the
fibre of $L_{12}$ over a point $[\ep_{12}]$ in $\mathbb{P}
(\Ext^1(E_1,E_2))$ is the 1-dimensional subspace $\{\la\,\ep_{12}:
\la\in\K\}$ in $\Ext^1(E_1,E_2)$. Consider the fibre product
$$
\xymatrix@C=70pt@R=25pt{Z\;  \ar[r]^{\textrm{\'etale}\quad\quad\quad\quad\quad\quad\quad\quad\quad}  \ar[d] & \Ext^1(E_1,E_1)\times \Ext^1(E_2,E_2)\times (L_{12}\op L_{21}) \ar[d] \\
U \; \ar[r]^{\textrm{\'etale}\quad}   & \Ext^1(E,E)}
$$
where the horizontal maps are \'etale morphisms. Informally, this defines 
$Z\subseteq\Ext^1(E_1,E_1)\times \Ext^1(E_2,E_2)\times (L_{12}\op
L_{21})$ to be the \'etale open subset of points
$\bigl(\ep_{11},\ep_{22},[\ep_{12}],\la_{1} \, \ep_{12},
[\ep_{21}],\la_{2} \, \ep_{21}\bigr)$ for $\la_i \in \K,$ for which
$(\ep_{21},\ep_{22},\la_{1}\,\ep_{12},\la_{2} \, \ep_{21})$ lies in $U.$
Observe that $Z$ contains both $\ti U'$ and $\ti U'',$ which respectively have subspaces $\Crit(\ti f')$ and $\Crit(\ti f'').$

\smallskip

Define also an \'etale open set of points 
$W\subseteq\Ext^1(E_1,E_1)\times \Ext^1(E_2,E_2)\times (L_{12}\ot
L_{21})$ fitting into the following cartesian square:
$$
\xymatrix@C=70pt@R=25pt{Z\;  \ar[r]^{\textrm{\'etale}\quad\quad\quad\quad\quad\quad\quad\quad\quad}  \ar[d]_\Pi & \Ext^1(E_1,E_1)
\times \Ext^1(E_2,E_2)\times (L_{12}\op L_{21}) \ar[d]^{\Pi'} \\
W \; \ar[r]^{\textrm{\'etale}\quad\quad\quad\quad\quad\quad\quad\quad\quad}   & \Ext^1(E_1,E_1)\times \Ext^1(E_2,E_2)\times (L_{12}\ot L_{21}) }
$$
where the line bundle
$$L_{12}\ot L_{21}\ra\mathbb{P}(\Ext^1(E_1,E_2))\times \mathbb{P}
(\Ext^1(E_2,E_1))$$ has fibre over
$([\ep_{12}],[\ep_{21}])$ which is $\{\la\,\ep_{12}\ot\ep_{21}:
\la\in\K\}$. Write points of the total space of $L_{12}\ot L_{21}$ as
$\bigl([\ep_{12}],[\ep_{21}],\la\,\ep_{12}\ot\ep_{21}\bigr)$. 
Informally, $W$ is defined as the open subset of points
$\bigl(\ep_{11},\ep_{22},[\ep_{12}],
[\ep_{21}],\la\,\ep_{12}\ot\ep_{21}\bigr)$ for which
$(\ep_{21},\ep_{22},\la\,\ep_{12},\ep_{21})$ lies in $U$. Since $U$
is $G$-invariant, this definition is independent of the
choice of representatives $\ep_{12},\ep_{21}$ for $[\ep_{12}],
[\ep_{21}]$, since any other choice would replace
$(\ep_{11},\ep_{22},\la\,\ep_{12},\ep_{21})$ by $(\ep_{11},\ep_{22},
\la\mu\,\ep_{12},\mu^{-1}\ep_{21})$ for some $\mu\in\bG_m$. 
The map $\Pi: Z\ra W$ is \'etale equivalent to $$\Pi':(\ep_{11},\ep_{22},[\ep_{12}],\la_{1} \, \ep_{12},
[\ep_{21}],\la_{2} \, \ep_{21})\mapsto(\ep_{11},\ep_{22},\ab
[\ep_{12}],\ab[\ep_{21}],\la\ep_{12}\ot\ep_{21})$$ which is a
smooth projection of relative dimension $1$ except at the points such that $\la_{1} =\la_{2} =0.$
However it is smooth at $(0 ,\la_{2} )$ with $\la_{2} \neq 0$ and similarly at $(\la_{1},0 )$ with $\la_{1} \neq 0,$ 
that is, the two restrictions of $\Pi$ to $\ti U'$ and $\ti U''$ are both
smooth of relative dimension $1.$ \index{almost closed $1$-form!invariant}

\smallskip

Here is the crucial point: $\Crit(\ti f')\subset \ti U'$ and $\Crit(\ti f'')\subset \ti U''$ are $\bG_m$-invariant subschemes, so there exists a subscheme $Q$ of $W$ such that $\Crit(\ti f')=\Pi^{-1}(Q)\cap\ti U'$ and $\Crit(\ti f'')=\Pi^{-1}(Q)\cap\ti U''$
and both $\Pi: \Crit(\ti f')\ra Q$ and $\Pi: \ti \om''^{-1}(0)\ra Q$ are smooth of relative dimension $1.$
Thus Theorem \ref{dt3thm3} yields that $\nu_{\Crit(\ti f')}=-\Pi^*(\nu_Q)$ and $\nu_{\Crit(\ti f'')}=-\Pi^*(\nu_Q)$
and then
\e
\nu_{\Crit(\ti f')}(0,0,[\ep_{12}],0,\ep_{21})=-\nu_{Q}(0,0,[\ep_{12}],[\ep_{21}],0)=\nu_{\Crit(\ti f'')}(0,0,\ep_{12},0,[\ep_{21}]),
\label{dt6eq29}
\e
where the sign comes from the fact that the map $\Pi$ is smooth of relative dimension $1.$ Moreover observe that
\e
\nu_{\Crit(f\vert_{V'})}(0,0,0,\ep_{21})=(-1)^{ e_{21}}\nu_{\Crit(f)^G}(0,0,0,0).
\label{dt6eq30}
\e

This is because the $T$-invariance of $f$ imply that its values on $(\ep_{11},\ep_{22},0,\ep_{21})$ and $(\ep_{11},\ep_{22},0,0)$ are the same 
and the projection $\Crit(f\vert_{V'}) \ra\Crit(f\vert_{U^T})$
is smooth of relative dimension $e_{21}.$
For the same reason, one has 
\e
\nu_{\Crit(f\vert_{V''})}(0,0,\ep_{12},0)=(-1)^{e_{12}}\nu_{\Crit(f)^G}(0,0,0,0).
\label{dt6eq31}
\e

\smallskip

Now, substitute equations (\ref{dt6eq29}), (\ref{dt6eq30}) and (\ref{dt6eq31}) into (\ref{dt6eq26}) and (\ref{dt6eq27}).
One gets
\e
\begin{split}
\nu_{\Crit(f)}(0,0,0,\ep_{21})\quad=\quad-\displaystyle\int\limits_{\small{[\ep_{12}]
\in\mathbb{P}(\Ext^1(E_1,E_2))}}\!\!\!\!\!\!\nu_{Q}(0,0,[\ep_{12}],[\ep_{21}],0) \,\d\chi
+(-1)^{e_{12}+e_{21}}\bigl(1-e_{12}\bigr)
\nu_{\Crit(f)^G}(0,0,0,0),
\end{split}
\label{dt6eq32}
\e
\e
\begin{split}
\nu_{\Crit(f)}(0,0,\ep_{12})\quad=\quad-\displaystyle\int\limits_{\small{[\ep_{21}]
\in\mathbb{P}(\Ext^1(E_2,E_1))}}\!\!\!\!\!\!\nu_{Q}(0,0,[\ep_{12}],[\ep_{21}],0) \,\d\chi
+(-1)^{e_{12}+e_{21}}\bigl(1-e_{21}\bigr)
\nu_{\Crit(f)^G}(0,0,0,0).
\end{split}
\label{dt6eq33}
\e

\smallskip

Finally integrating \eq{dt6eq32} over $[\ep_{21}]\in\mathbb{P}(\Ext^1(E_2,E_1))$ and \eq{dt6eq33} over $[\ep_{12}]
\in\mathbb{P}(\Ext^1(E_1,E_2)),$ yields

\e
\begin{split}
\int\limits_{\small{[\ep_{21}]\in\mathbb{P}(\Ext^1(E_2,E_1))}} &
\!\!\!\!\!\!\!\!\!\nu_{\Crit(f)}(0,0,0,\ep_{21}) \,\d\chi\quad  =\quad - \int\limits_{\small{([\ep_{12}],[\ep_{21}])\in\mathbb{P}(\Ext^1(E_1,E_2))\times
\mathbb{P}(\Ext^1(E_2,E_1))}}
\!\!\!\!\!\!\!\!\!\!\!\!\!\!\!\!\!\!\!\!\!\!\!\!\!\!\!\nu_{Q}(0,0,[\ep_{12}],[\ep_{21}],0) \,\d\chi 
\\ & +(-1)^{e_{12}+e_{21}}\bigl(1-e_{12}\bigr) 
e_{21}
\nu_{\Crit(f)^G}(0),
\label{dt6eq34}
\end{split}
\e
\e
\begin{split}
\int\limits_{\small{[\ep_{12}]\in\mathbb{P}(\Ext^1(E_1,E_2))}} &
\!\!\!\!\!\!\!\!\!\nu_{\Crit(f)}(0,0,\ep_{12},0) \,\d\chi\quad=\quad - \int\limits_{\small{([\ep_{12}],[\ep_{21}])\in\mathbb{P}(\Ext^1(E_1,E_2))\times
\mathbb{P}(\Ext^1(E_2,E_1))}}
\!\!\!\!\!\!\!\!\!\!\!\!\!\!\!\!\!\!\!\!\!\!\!\!\!\!\!\nu_{Q}(0,0,[\ep_{12}],[\ep_{21}],0) \,\d\chi 
\\& +(-1)^{e_{12}+e_{21}}\bigl(1-e_{21}\bigr) 
e_{12}
\nu_{\Crit(f)^G}(0),
\label{dt6eq35}
\end{split}
\e

since $\chi\bigl(\mathbb{P}(\Ext^1(E_2,E_1))\bigr)=e_{21}$ and  $\chi\bigl(\mathbb{P}(\Ext^1(E_1,E_2))\bigr)=e_{12}.$ 
Subtracting \eq{dt6eq34} from \eq{dt6eq35}, gives 

\e
\begin{split}
\int\limits_{\small{[\ep_{21}]\in\mathbb{P}(\Ext^1(E_2,E_1))}} &
\!\!\!\!\!\!\!\!\!\nu_{\Crit(f)}(0,0,0,\ep_{21}) \,\d\chi\quad - \int\limits_{\small{[\ep_{12}]\in\mathbb{P}(\Ext^1(E_1,E_2))}} 
\!\!\!\!\!\!\!\!\!\nu_{\Crit(f)}(0,0,\ep_{12},0) \,\d\chi= \\ & (-1)^{e_{12}+e_{21}}\bigl(e_{21}-e_{12}\bigr) 
\nu_{\Crit(f)^G}(0).
\end{split}
\label{dt6eq25}
\e

Consider equation \eq{dt6eq25} applied substituting
$\mathbb{P}(\Ext^1(E_2,E_1)\op\K)$ to $\mathbb{P}(\Ext^1(E_2,E_1)).$
This adds one dimension to $\Ext^1(E,E).$ Denote $\ti{\ti f}$ the lift of $f$ to $\Ext^1(E,E)\op\K.$
In this case equation \eq{dt6eq25} becomes 
\e
\begin{split}
\int\limits_{\small{[\ep_{21}]\in\mathbb{P}(\Ext^1(E_2,E_1)\op\K)}} &
\!\!\!\!\!\!\!\!\!\nu_{\Crit({\ti{\ti f}})}(0,0,0,\ep_{21}\op\la) \,\d\chi\quad - \int\limits_{\small{[\ep_{12}]\in\mathbb{P}(\Ext^1(E_1,E_2))}} 
\!\!\!\!\!\!\!\!\!\nu_{\Crit({\ti{\ti f}})}(0,0,\ep_{12},0) \,\d\chi= \\ &(-1)^{1+e_{12}+e_{21}}\bigl(1+e_{21}-e_{12}\bigr) 
\nu_{\Crit({\ti{\ti f}})^G}(0),
\end{split}
\label{dt6eq25.1.1}
\e

Now, observe that $\nu_{\Crit(f)}=-\nu_{\Crit({\ti{\ti f}})}$ from Theorem \ref{dt3thm3} and 
$\nu_{\Crit({\ti{\ti f}})^G}(0)=\nu_{\Crit(f)^G}(0)$ 
as $(\Ext^1(E,E)\op\K)^G=\Ext^1(E,E)^G\op 0$
and the map $\Crit({\ti{\ti f}})^G \ra \Crit(f)^G$ is \'etale. Thus
\e
\begin{split}
-\!\!\!\!\!\!\!\!\!\int\limits_{\small{[\ep_{21}]\in\mathbb{P}(\Ext^1(E_2,E_1))}} &
\!\!\!\!\!\!\!\!\!\nu_{\Crit(f)}(0,0,0,\ep_{21}) \,\d\chi \; -\;  \nu_{\Crit(f)}(0,0,0,0) \quad + \int\limits_{\small{[\ep_{12}]\in\mathbb{P}(\Ext^1(E_1,E_2))}} 
\!\!\!\!\!\!\!\!\!\nu_{\Crit(f)}(0,0,\ep_{12},0) \,\d\chi= \\ & (-1)^{1+e_{12}+e_{21}}\bigl(1+e_{21}-e_{12}\bigr) 
\nu_{\Crit(f)^G}(0).
\end{split}
\label{dt6eq25.1}
\e

\smallskip

Here, $\nu_{\Crit(f)}(0)$ on the l.h.s. comes from the fact that the $\bG_m$-action 
over $\mathbb{P}(\Ext^1(E_2,E_1)\op\K)$ fixes
$\mathbb{P}(\Ext^1(E_2,E_1))$ and $[0,1];$ the free orbits of the $\bG_m$-action 
contribute zero to the weighted Euler characteristic. Then one uses that $\nu_{\Crit({\ti{\ti f}})}$ valued over $[0,1]$ 
is equal to $-\nu_{\Crit(f)}(0).$
Adding \eq{dt6eq25} and \eq{dt6eq25.1} yields \eq{dt6eq7}, which concludes the proof of identity \eq{dt6eq1.1}.

\medskip

The conclusion of the proof of identity \eq{dt6eq2.1} is now easy. 
Let $0\ne\ep_{21}\in\Ext^1(E_2,E_1)$ correspond to the short exact
sequence $0\ra E_1\ra F\ra E_2\ra 0$ in $\coh(X)$. Then
\e
\nu_{\fM}(F)=(-1)^{\dim\Aut(E)}\nu_{\Crit(f)}(0,0,0,\ep_{21})
\label{dt6eq24}
\e
using $\xi_*:[(0,0,0,\ep_{21})]\mapsto[F]$ from Proposition
\ref{dt10prop2} and $\xi$ smooth of relative dimension $\dim(\Aut(E))$ 
and properties of Behrend function in Theorem \ref{dt3thm3}.
Substituting \eq{dt6eq24} and its analogue for $D$ in the place of $F$ into
\eq{dt6eq2.1}, using equation \eq{dt6eq5} and identity \eq{dt6eq7} to substitute for $\nu_{\fM}(E_1\op E_2)$, and
cancelling factors of $(-1)^{\dim\Aut(E)} $, one gets
that \eq{dt6eq2.1} is equivalent to \eq{dt6eq25}, which concludes the proof.

\subsection{Deformation invariance issue}
\markboth{Deformation invariance issue}{Deformation invariance issue}
\label{def} \index{Picard scheme}\index{Picard scheme!relative}

Thomas' original definition \eq{dt2eq1} of \index{Donaldson--Thomas invariants!original $DT^\al(\tau)$}
$DT^\al(\tau)$, and Joyce and Song's definition \eq{dt4eq11} of the pair \index{stable pair invariants $PI^{\al,n}(\tau')$}
invariants $PI^{\al,n}(\tau')$, are both valid over $\K$. Joyce and Song suggest to solve problem (b) in \S\ref{main.1} to work in \cite[Rmk 4.20 (e)]{JoSo}, replacing $H^*(X;\Q)$ by the {\it algebraic
de Rham cohomology\/}\index{algebraic de Rham cohomology} $H^*_{\rm
dR}(X)$ of Hartshorne \cite{Hart1}.  \index{cohomology}
Here we suggest another argument which is based on 
the theory of {\it Picard schemes} by Grothendieck \cite{Grot4,Grot5}.
Other references are \cite{Art,Kle}.
Even if our argument will not prove that the numerical Grothendieck groups
are deformation invariant, as this last fact depend deeply on
the integral Hodge conjecture type result \cite{Vois} which we are not able to 
prove in this more general context, we will however find a deformation invariant lattice $\Lambda_{X_t}$
containing its image through the Chern character map and define $\bar{DT}{}^\al(\tau)_t$ for $\al\in \Lambda_{X_t}$ which will be deformation invariant.

\nomenclature[Pic]{$\Pic_{\fX/T}$}{relative Picard scheme of a family $\fX\ra T$}
\nomenclature[Picc]{$\Pic(X)$}{Picard scheme of a $\K$-scheme}
\medskip

To prove deformation-invariance we
need to work not with a single Calabi--Yau 3-fold $X$ over $\K$, but
with a {\it family\/} of Calabi--Yau 3-folds
$\fX\stackrel{\vp}{\longra}T$ over a base $\K$-scheme $T$. Taking
$T=\Spec\K$ recovers the case of one Calabi--Yau 3-fold. Here are
our assumptions and notation for such families.
Let $\fX\stackrel{\vp}{\longra}T$ be a smooth projective morphism of
algebraic $\K$-varieties $X,T$, with $T$ connected. Let $\cO_\fX(1)$
be a relative very ample line bundle for
$\fX\stackrel{\vp}{\longra}T$. For each $t\in T(\K)$, write $X_t$ for
the fibre $X\times_{\vp,T,t}\Spec\K$ of $\vp$ over $t$, and
$\cO_{X_t}(1)$ for $\cO_\fX(1)\vert_{X_t}$. Suppose that $X_t$ is a
smooth Calabi--Yau 3-fold over $\K$ for all $t\in T(\K)$, with $H^1(\cO_{X_t})= 0$. 

\medskip

There are some important existence theorems which refine the original Grothendieck's theorem \cite[Thm. 3.1]{Grot4}.
In \cite[Thm. 7.3]{Art}, Artin proves that given $f:X\ra S$ 
a flat, proper, and finitely presented map of algebraic spaces
cohomologically flat in dimension zero, then 
the relative Picard scheme 
$\Pic_{X/S}$ exists as an algebraic space 
which is locally of finite presentation over S.
Its fibres are the Picard schemes $\Pic(X_s)$ of the fibres. They form a family whose total space is $\Pic_{X/S}.$
In \cite[Prop. 2.10]{Grot5} Grothendieck shows that if $H^2(X_s,\cO_{X_{s}})=0$ for some $s\in S,$ there exists a neighborhood $U$ of $s$ such that the scheme $\Pic_{{X/S}_{|_U}}$ is smooth, and in this case $\dim(\Pic(X_s))=\dim(H^1(X_s,\cO_{X_s})).$

\medskip

In our case, $\Pic_{\fX/T}$ exists and is smooth with $0$-dimensional fibres which are the Picard schemes $\Pic(X_t).$
Moreover the morphism $(\pi,P): \Pic_{\fX/T}\longra T\times \Q[s],$ 
where $\pi$ is the projection to the base scheme and $P$ assigns 
to an isomorphism class of a line bundle $[L]$ its Hilbert polynomial $P_L(s)$ with respect to $\cO_\fX(1),$
is proper. This implies an upper semicontinuity result for $t\ra\dim(\Pic(X_t))$ \cite[Cor. 2.7]{Grot5}.
These results yield that the Picard
schemes $\Pic(X_t)$ for $t\in T(\K)$ are canonically
isomorphic {\it locally} in $T(\K)$. Observe that at the moment we don't have canonical
isomorphisms $\Pic(X_t)\cong \Pic(X)$ for all $t\in
T(\K)$ (this would be canonically isomorphic {\it globally} in
$T(\K)$). Instead, we mean that the groups $\Pic(X_t)$ for
$t\in T(\K)$ form a {\it local system of abelian groups\/} over
$T(\K)$, with fibre $\Pic(X)$.

\medskip

When $\K=\C$, Joyce and Song proved \cite[\S 4]{JoSo}
that  $K^\num(\coh(X_t))$ form a local system of abelian groups over
$T(\K)$, with fibre $K^\num(\coh(X))$.
This means that in simply-connected regions of $T(\C)$
in the complex analytic topology the $K^\num(\coh(X_t))$ are all
canonically isomorphic, and isomorphic to $K(\coh(X))$. But around
loops in $T(\C)$, this isomorphism with $K(\coh(X))$ can change by
{\it monodromy},\index{monodromy} by an automorphism
$\mu:K(\coh(X))\ra K(\coh(X))$ of $K(\coh(X))$. In \cite[Thm 4.21]{JoSo} they showed that the group
of such monodromies $\mu$ is finite, and so it is possible to make it trivial by
passing to a finite cover $\ti T$ of $T$. If they worked instead with
invariants $PI^{P,n}(\tau')$ counting pairs $s:\cO_{X}(-n)\ra E$ in
which $E$ has fixed Hilbert polynomial\index{Hilbert polynomial}
$P$, rather than fixed class $\al\in K^\num(\coh(X))$, as in Thomas'
original definition of Donaldson--Thomas invariants \cite{Thom},
then they could drop the assumption on $K^\num(\coh(X_t))$ in
Theorem \cite[Thm. 5.25]{JoSo}.

\medskip

Similarly, we now study monodromy phenomena for $\Pic(X_t)$ in families
of smooth $\K$-schemes $\fX\ra T$ following the idea of \cite[Thm. 4.21]{JoSo}. We
find that we can always eliminate such monodromy by passing to a
finite cover $\ti T$ of $T$. This is crucial to prove deformation-invariance of
the~$\bar{DT}{}^\al(\tau),PI^{\al,n}(\tau')$ in \cite[\S 12]{JoSo}.

\begin{thm} Let\/ $\K$ be an algebraically closed field of characteristic zero, $\vp:\fX\ra
T$ a smooth projective morphism of\/ $\K$-schemes with\/ $T$
connected, and\/ $\cO_\fX(1)$ a relative very ample line bundle on $\fX,$ so
that for each\/ $t\in T(\K),$ the fibre $X_t$ of\/ $\vp$ is a smooth
projective $\K$-scheme with very ample line bundle\/ $\cO_{X_t}(1)$.
Suppose the Picard schemes $\Pic(X_t)$ are
locally constant in $T(\K),$ so that\/ $t\mapsto \Pic(X_t)$
is a local system of abelian groups on\/~$T$.
Fix a base point\/ $s\in T(\K),$ and let\/ $\Ga$ be the monodromy  group of $\Pic(X_s).$ 
Then $\Ga$ is a finite group. There exists a finite
\'etale cover $\pi:\ti T\ra T$ of degree $\md{\Ga},$ with\/ $\ti T$
a connected $\K$-scheme, such that writing $\ti \fX=\fX\times_T\ti T$
and $\ti\vp:\ti \fX\ra\ti T$ for the natural projection, with fibre
$\ti X_{\ti t}$ at $\ti t\in\ti T(\K),$ then $\Pic(\ti X_{\ti
t})$ for all $\ti t\in\ti T(\K)$ are all globally canonically
isomorphic to $\Pic(X_s)$. That is, the local system $\ti
t\mapsto \Pic(\ti X_{\ti t})$ on $\ti T$ is trivial.
\label{finmon}
\end{thm}

\begin{proof} 
As $\Pic(X_s)$ is finitely generated, one can choose classes 
$[L_1],\ldots, [L_k]\in \Pic(X_s)$ as generators. Let $P_1,\ldots, P_k$ be the Hilbert polynomials 
respectively of $[L_1],\ldots, [L_k]$ with respect to $\cO_{X_s}(1)$.
Let $\ga\in\Ga$, and consider the images $\ga\cdot [L_i] \in
\Pic(X_s)$ for $i=1,\ldots,k$. As we assume $\cO_\fX(1)$ is
globally defined on $T$ and does not change under monodromy, it
follows that the Hilbert polynomials $P_1,\ldots, P_k$
do not change under monodromy. Hence
$\ga\cdot [L_i]$ has Hilbert polynomial $P_i$.
Again one uses properness to show that the set $\Pic^{P_i}(X_s)$ composed 
by isomorphism classes of line bundles in $\Pic(X_s)$ with Hilbert polynomial $P_i$ for some $i=1,\ldots, k$
is a finite set, that is, every $P_i$ is the Hilbert polynomial of 
only finitely many classes $[R_1],\ldots,[R_{n_i}]$ in
$\Pic(X_s).$ It follows that for each $\ga\in\Ga$ we have
$\ga\cdot [L_i]\in\{[R_1],\ldots,[R_{n_i}]\}$. So there are at
most $n_1\cdots n_k$ possibilities for $(\ga\cdot [L_1],\ldots,
\ga\cdot [L_k])$. But $(\ga\cdot [L_1],\ldots,\ga\cdot [L_k])$
determines $\ga$ as $[L_1],\ldots, [L_k]$ generate
$\Pic(X_s)$. Hence $\md{\Ga}\le n_1\cdots n_k$, and $\Ga$ is
finite.

We can now construct an \'etale cover $\pi:\ti T\ra T$ which is a
principal $\Ga$-bundle, and so has degree $\md{\Ga}$, such that the
$\K$-points of $\ti T$ are pairs $(t,\io)$ where $t\in T(\K)$ and
$\io:\Pic(X_t)\ra \Pic(X_s)$ is an isomorphism
from the properness and smoothness argument above, and $\Ga$ acts freely on $\ti
T(\K)$ by $\ga:(t,\io)\mapsto (t,\ga\ci\io)$, so that the
$\Ga$-orbits correspond to points $t\in T(\K)$. Then for $\ti
t=(t,\io)$ we have $\ti X_{\ti t}=X_t$, with canonical
isomorphism~$\io:\Pic(\ti X_{\ti t})\ra \Pic(X_s)$.
\end{proof}
\index{Grothendieck group!numerical}\index{monodromy}
\index{cohomology}\index{Picard scheme}

So the conclusion is that from properness and smoothness argument,
$\Pic(X_t)$ are canonically isomorphic locally in
$T(\K)$. But by Theorem \ref{finmon}, one can pass to a finite cover
$\ti T$ of $T$, so that the $\Pic(\ti X_{\ti t})$ are
canonically isomorphic globally in $\ti T(\K)$. So,
replacing $\fX,T$ by $\ti \fX,\ti T$, we will assume from here 
that the Picard schemes $\Pic(X_t)$ for $t\in
T(\K)$ are all canonically isomorphic globally in $T(\K)$,
and we write $\Pic(X)$ for this group $\Pic(X_t)$ up to
canonical isomorphism. 

\medskip

In Theorem \cite[Thm. 4.19]{JoSo} Joyce and Song showed that when $\K=\C$ and
$H^1(\cO_X)=0$ the numerical Grothendieck group $K^\num(\coh(X))$ is
unchanged under small deformations of $X$ up to canonical
isomorphism. As we said, here we will not prove this result.
So, the idea is to construct a globally constant 
lattice $\Lambda_{X}$ using the globally constancy of the Picard schemes 
such that there exist an inclusion $K^\num(\coh(X))\hookra\Lambda_{X}.$
It could happen that the image of the numerical Grothendieck group varies with $t$ as
it has to do with the integral Hodge conjecture as in \cite[Thm. 4.19]{JoSo}, but this does not affect the deformation invariance of
$\bar{DT}{}^\al(\tau)$ as for them to be deformation invariant is enough to find a deformation invariant lattice 
in which the classes $\al$ vary. Next, we describe such lattice $\Lambda_X$ and explain how 
the numerical Grothendieck group $K^\num(\coh(X))$ is contained in it. Our idea follows \cite[Thm. 4.19]{JoSo}.

\medskip\index{Chern character}

Let $X$ be a Calabi--Yau 3-fold over $\K$, with $H^1(\cO_X)=0$ and consider
the {\it Chern character},
as in Hartshorne \cite{Hart2}: for
each $E\in\coh(X)$ we have the rank $r(E)\in A^0(X)\cong \Z,$ 
and the
Chern classes $c_i(E)\in A^{i}(X)$ for $i=1,2,3$. It is useful
to organize these into the Chern character
$\ch(E)$\nomenclature[ch(E)]{$\ch(E)$}{Chern character of a coherent sheaf $E$}
\nomenclature[chi(E)]{$\ch_i(E)$}{$i^{\rm th}$
component of Chern character of $E$} in $A^{*}(X)_\Q$,
\nomenclature[Heven(X)]{$H_{dR}^{\rm even}(X)$}{even cohomology of
a $\K$-scheme $X$} where
$\ch(E)=\ch_0(E)+\ch_1(E)+\ch_2(E)+\ch_3(E)$ with $\ch_i(E)\in A^{i}(X)_\Q:$ 
\e
\ch_0(E)=r(E),\quad \ch_1(E)=c_1(E),\quad
\ch_2(E)=\ts\frac{1}{2}\bigl(c_1(E)^2-2c_2(E)\bigr),\quad
\ch_3(E)=\ts\frac{1}{6}\bigl(c_1(E)^3-3c_1(E)c_2(E)+3c_3(E)\bigr).
\label{chernch}
\e

By the Hirzebruch--Riemann--Roch Theorem\index{Hirzebruch--Riemann--Roch
Theorem} \cite[Th.~A.4.1]{Hart2}, the Euler form\index{Euler form} on
coherent sheaves $E,F$ is given in terms of their Chern characters
by
\e
\bar\chi\bigl([E],[F]\bigr)=\deg\bigl(\ch(E)^\vee\cdot\ch(F)\cdot{\rm
td}(TX)\bigr){}_3,
\label{euform}
\e
where $(\cdot)_3$ denotes the component of degree $3$ 
in $A^*(X)_\Q$ and where ${\rm td}(TX)$\nomenclature[td(TX)]{${\rm td}(TX)$}{Todd class of $TX$ in
$H_{dR}^{\rm even}(X)$} is the {\it Todd class\/}\index{Todd class} of
$TX$, which is $1+\frac{1}{12}c_2(TX)$ as $X$ is a Calabi--Yau
3-fold, and
$(\la_0,\la_1,\la_2,\la_3)^\vee=(\la_0,-\la_1,\la_2,-\la_3)$,
writing $(\la_0,\ldots,\la_3)\in A^{*}(X)$ 
with $\la_i\in A^{i}(X)$.
Define:

\begin{equation*}
\La_X=\ts\bigl\{ (\la_0,\la_1,\la_2,\la_3) 
\textrm{ where } \la_0,\la_3\in\Q, \; \la_1\in\Pic(X)\ot_{\Z}\Q, \; \la_2\in \Hom(\Pic(X),\Q)
\textrm{ such that } 
\end{equation*}
\begin{equation*}
\la_0\in \Z,\;\>
\la_1\in \Pic(X)/ {\textrm{torsion}}, \;
 \la_2-\ha\la_1^2\in  \Hom(\Pic(X),\Z),\;\> \la_3+\ts\frac{1}{12}\la_1
c_2(TX)\in \Z\bigr\},
\end{equation*}

\smallskip

where $\la_1^2$ is defined as the map $\al\in\Pic(X)\ra \frac{1}{2}c_1(\la_1)\cdot c_1(\la_1)\cdot c_1(\al)\in A^3(X)_{\Q}\cong \Q,$
and $\frac{1}{12}\la_1 c_2(TX)$ is defined as $\frac{1}{12}c_1(\la_1)\cdot c_2(TX)\in A^3(X)_{\Q}\cong\Q.$
Theorem \ref{definv} states that $\La_X$ is deformation invariant and the Chern character gives an injective morphism
$\ch:K^\num(\coh(X))\!\hookra\!\La_X$. \index{Picard scheme} \index{deformation invariance}
The proof of Theorem \ref{definv} is straightforward:

\begin{proof} 
The proof 
follows exactly as in \cite[Thm. 4.19]{JoSo} and 
the fact the Picard scheme $\Pic(X)$ is globally constant in families from the argument above yields that the lattice $\Lambda_X$ is deformation invariant.
Moreover, the proof that $\ch\bigl(K^\num(\coh(X))\bigr)\subseteq\La_X$ is again as in \cite[Thm. 4.19]{JoSo}.
Observe that we do not prove that $\ch\bigl(K^\num(\coh(X))\bigr)=\La_X,$ fact which uses Voisin's Hodge conjecture proof for Calabi--Yau 3-folds over $\C$ \cite{Vois}.
\end{proof}

\begin{quest}
Does Voisin's result \cite{Vois} work over $\K$ in terms of $\Hom(\Pic(X),\Z)$?
\end{quest}

This concludes the discussion of problem (b) in \S\ref{main.1} and yields the deformation-invariance of
$DT^\al(\tau),$ $PI^{\al,n}(\tau')$ over $\K.$

\section{Implications and conjectures}
\markboth{Implications and conjectures}{Implications and conjectures}
\label{dt7}\index{non-Archimedean geometry}\index{formal power series}

In this final section we sketch some exciting and far reaching implications of the theory and propose new ideas for further research.
One proposal is in the direction of extending Donaldson--Thomas invariants to
compactly 
supported coherent sheaves on noncompact quasi-projective Calabi--Yau 3-folds.
A second idea is in the derived categorical framework trying to establish a theory of generalized Donaldson--Thomas invariants for objects in the derived category of coherent sheaves. Here we expose the problems and illustrate some possible approaches when known.

\subsection{Noncompact Calabi--Yau 3-folds}

We start by recalling the following definition from \cite[Def. 6.27]{JoSo}:

\begin{dfn}
Let $X$ be a noncompact Calabi-Yau $3$-fold over $\C.$ We call $X$ compactly embeddable if whenever $K \subset X$ is a compact subset, in the analytic topology, there exists an open neighbourhood $U$ of $K$ in $X$ in the analytic topology, a compact Calabi-Yau 3-fold $Y$ over $\C$ with $H^1(\cO_Y ) = 0,$ an open subset $V$ of $Y$ in the analytic topology, and an isomorphism of complex manifolds $\varphi : U \ra V.$
\end{dfn}

Joyce and Song only need the notion of `compactly embeddable'
as their complex analytic proof of \eq{dt6eq1}--\eq{dt6eq2} 
recalled in \S\ref{dtBehid}, requires
$X$ compact; but unfortunately 
the given algebraic version of \eq{dt6eq1}--\eq{dt6eq2} in Theorem \ref{dt6thm1.1} 
uses results from derived algebraic geometry  \cite{ToVe1,ToVe2,PTVV,Toen,Toen2,Toen3,Toen4}, and the author does not know if
they apply also for compactly supported sheaves 
\index{coherent sheaf!compactly supported} on a
noncompact~$X$.\index{field $\K$|)}
\index{Calabi--Yau 3-fold!noncompact}

\medskip

More precisely, in \cite{PTVV} it is shown that if $X$ is a projective Calabi-Yau $m$-fold then the derived moduli stack $\fM_{\Perf(X)}$ of perfect complexes of coherent sheaves on $X$ is $(2-m)$-shifted symplectic.
It is not obvious that if $X$ is a quasi-projective Calabi-Yau $m$-fold, possibly noncompact, then the derived moduli stack $\fM_{\Perf_{cs}(X)}$ of perfect complexes on $X$ with compactly-supported cohomology is also $(2-m)$-shifted symplectic.

\medskip

At the present, we can state the following result. We thank Bertrand To\"en for explaining this to us. 

\begin{thm}
Suppose $Z$ is smooth projective of dimension $m,$ and $s \in H^0(K_Z^{-1}),$ and $X \subset Z$ is Zariski open with $s$ nonvanishing on $X,$ so that $X$ is a (generally non compact) quasi-projective Calabi-Yau $m$-fold. Then the derived moduli stack $\fM_{\Perf_{cs}}(X)$ of compactly-supported coherent sheaves on $X,$ or of perfect complexes on $X$ with compactly-supported cohomology, is $(2-m)$-shifted symplectic.
\label{noncpt}
\end{thm}

\begin{proof}
Let $Z$ be smooth and projective of dimension $m,$ and $s$ be any section of $K_Z^{-1}$. Let $Y$ be the derived scheme of zeros of $s$ and $X=Z\setminus Y.$ Then, $Y$ is equipped with a canonical $O$-orientation in the sense of \cite{PTVV} of
dimension $m-1,$ so $\fM_{\Perf}(Y)$ is $(2-m-1)$-symplectic, even if $Y$ is not smooth. The restriction map $\fM_{\Perf}(Z) \ra \fM_{\Perf}(Y)$
is moreover Lagrangian. The map $\ast \ra \fM_{\Perf}(Y),$ corresponding to the zero
object is \'etale, and thus its pull-back provides a Lagrangian map $\fM_{\Perf_{cs}}(X) \ra\ast$, or, equivalently, a $(2-m)$-symplectic structure on $\fM_{\Perf_{cs}}(X).$ Now if $X'$  is open in
$X,$ then $\fM_{\Perf_{cs}}(X') \ra \fM_{\Perf_{cs}}(X)$ is an open immersion, so $\fM_{\Perf_{cs}}(X')$ is also $(2-m)$-symplectic.
\end{proof}

We remark the following:

\begin{itemize}
\item[$(a)$] We point out that the condition of Theorem \ref{noncpt} is similar to the 
compactly-embeddable condition in \cite[Def. 6.27]{JoSo}, but more general, as we do not require $Z$
to be a Calabi-Yau.
\smallskip
\item[$(b)$] Observe that in the non-compact case we cannot expect to have the deformation invariance property 
unless in some particular cases in which the moduli space is proper.
\smallskip
\item[$(c)$] Note that we need the noncompact Calabi--Yau to be quasi-projective in order to have a quasi
projective Quot scheme \cite[Thm. 6.3]{NN}.
\end{itemize}

%\begin{ex}if you look at $O(-1) + O(-1) \ra CP^1,$ or at $K_{CP^{m-1}}$, I think you'll be able to find such projective Z containing these X with suitable s quite easily. For the first, try Z a small resolution of a nodal quadric in $CP^4;$ for the second, try Z the resolution of a weighted projective space $CP^m_{1,...,1,m}.$
%\end{ex}

We conclude the section with the following:

\begin{conj}The theory of generalized Donaldson--Thomas invariants defined in \cite{JoSo} is valid 
over algebraically closed fields of characteristic zero for compactly 
supported coherent sheaves on noncompact quasi-projective Calabi--Yau 3-folds.
In this last case, one can define $\bar{DT}{}^\al(\tau)$ and prove the wall--crossing formulae 
and the relation with $PI^{\al,n}(\tau')$ is still valid, while one loses the deformation invariance property and the 
properness of moduli spaces. 
\end{conj}

\subsection{Derived categorical framework} 
\label{dercat}

Our algebraic method could lead to the extension of generalized Donaldson--Thomas theory to the derived categorical context. 
The plan to extend from abelian to derived categories the theory of Joyce and Song \cite{JoSo}
starts by reinterpreting the series of papers by Joyce
\cite{Joyc.1,Joyc.2,Joyc.3,Joyc.4,Joyc.5,Joyc.6,Joyc.7,Joyc.8} in this new general setup. In particular:

\smallskip

\begin{itemize}

\item[$(a)$]Defining configurations in triangulated categories $\mathcal{T}$ requires to replace the exact sequences by distinguished triangles.

\smallskip

\item[$(b)$]Constructing moduli stacks of objects and configurations in $\mathcal{T}$. Again, the theory of derived algebraic geometry \cite{Toen,Toen2,Toen3,Toen4,ToVe1,ToVe2,PTVV} can give us a satisfactory answer.

\smallskip

\item[$(c)$]Defining stability conditions on triangulated categories can be approached using  Bridgeland's results, and its extension by Gorodentscev et al., which combines Bridgeland's idea with Rudakov's definition for abelian categories \cite{Rud}. Since Joyce's stability conditions \cite{Joyc.3} are based on Rudakov, the modifications should be straightforward.

\end{itemize}

\begin{itemize}
\item[$(d)$]The `nonfunctoriality of the cone' in triangulated categories causes that the triangulated category versions of some operations on configurations are defined up to isomorphism, but not canonically, which yields that corresponding diagrams may be commutative, but not Cartesian as in the abelian case. In particular, one loses the associativity of the Ringel-Hall algebra of stack functions, which is a crucial object in Joyce and Song framework. 
We expect that derived Hall algebra approach of To\"en \cite{Toen3} resolve this issue. See also \cite{PL}.
\end{itemize}

\smallskip

The list above does not represent a big difficulty.  
The main issues actually are: proving existence of Bridgeland stability conditions (or other type) on the derived category; proving that semistable moduli schemes and stacks are finite type (permissible), and proving that two stability conditions can be joined by a path of permissible stability conditions.

\smallskip

Theorem \ref{dt6thm1.1.bis} is just one of the steps in developing this program. The author thus expects that a well-behaved theory of invariants counting $\tau$-semistable objects in triangulated categories in the style of Joyce's theory exists, that is, Theorem \ref{mainthm} should be valid also in the derived categorical context:

\begin{conj} 
The theory of generalized Donaldson--Thomas invariants defined in \cite{JoSo} is valid 
for complexes of coherent sheaves on Calabi-Yau $3$-folds over algebraically closed fields of characteristic zero.
\end{conj}

\small{

 }

\end{document}